\documentclass[10pt,a4paper, final]{article}

\usepackage[a4paper, left=31mm, right=31mm]{geometry}
\usepackage[T1]{fontenc}							% Output encodings (e.g. when copying from PDF, should come before inputenc)
\usepackage[utf8]{inputenc}							% Input encodings (Umlaute, ß)
\usepackage{amsmath, amsfonts, amssymb, amsthm, empheq}		% AMS packages (get documentation with $ texdoc amsldoc)
\usepackage{graphicx}								% Include graphics
\usepackage{microtype}								% Adjusts spacings
\usepackage{hyperref}								% Create links
\usepackage{cleveref}								% \Cref{} or \cref{} automatically includes environment of reference
\usepackage{booktabs}								% Nicer tables
\usepackage{todonotes}
\usepackage{float}
\usepackage{caption}
\usepackage{subcaption}						% \todo{Rewrite this answer.}
\usepackage[notref,notcite]{showkeys}				% Adds the equation names.
\usepackage{siunitx}
\usepackage{array} % needed for the 'm' column type in tabular
\usepackage{amssymb}
\usepackage{amsmath, amsfonts, amssymb, amsthm}
\usepackage{caption}
\usepackage{subcaption}
\usepackage{xcolor}
\usepackage{siunitx}
\usepackage[T1]{fontenc}

\newcommand{\reviewerOne}[1]{\textcolor{black}{#1}}
\newcommand{\reviewerTwo}[1]{\textcolor{black}{#1}}

\def\electrons{\text{n}}
\def\holes{\text{p}}

\def\ions{\text{a}}

\newcommand*{\QEDB}{\null\nobreak\hfill\ensuremath{\square}}%

\newcommand{\overbar}[1]{\mkern 1.5mu\overline{\mkern-1.5mu#1\mkern-1.5mu}\mkern 1.5mu}

\theoremstyle{definition}								% amsthm only: italics
\newtheorem{theorem}{Theorem}
\newtheorem{lemma}[theorem]{Lemma}					% [theorem] means we use the same counter as for {theorem}

\newtheorem{remark}[theorem]{Remark}

\numberwithin{equation}{section}
\numberwithin{figure}{section}
\numberwithin{theorem}{section}

\renewcommand{\appendix}{\par
  \setcounter{section}{0}
  \setcounter{subsection}{0}
  \gdef\thesection{\Alph{section}}
}

\newcommand*\samethanks[1][\value{footnote}]{\footnotemark[#1]} % People at the same institution

\title{Numerical analysis and simulation of lateral memristive devices: Schottky, ohmic, and multi-dimensional electrode models}

\author{
    Dilara Abdel\thanks{Weierstrass Institute for Applied Analysis and Stochastics (WIAS), Mohrenstr. 39, 10117 Berlin, Germany}
    \and Maxime Herda\thanks{Univ. Lille, CNRS, Inria, UMR 8524 - Laboratoire Paul Painlevé, F-59000 Lille, France}
    \and Martin Ziegler\thanks{Christian-Albrechts-Universität zu Kiel, Christian-Albrechts-Platz 4, 24118 Kiel, Germany}
    \and Claire Chainais-Hillairet\samethanks[2]
    \and Benjamin Spetzler\samethanks[3]
    \and Patricio Farrell\samethanks[1]
}

\date{\today}

\begin{document}

    \maketitle
    %%%%%%%%%%%%%%%%%%%%%%%%%%%%%%%%%%%%%%%%%%%%%%%%%%%%%%%%%%%%%%%%%%%%%%%%%%%
    \begin{abstract}
        In this paper, we present the numerical analysis and simulations of a multi-dimensional memristive device model.
        Memristive devices and memtransistors based on two-dimensional (2D) materials have demonstrated promising potential for neuromorphic computing and next-generation memory technologies.
        Our charge transport model describes the drift-diffusion of electrons, holes, and ionic defects self-consistently in an electric field. We incorporate two types of boundary models: ohmic and Schottky contacts. The coupled drift-diffusion partial differential equations are discretized using a physics-preserving Voronoi finite volume method. It relies on an implicit time-stepping scheme and the excess chemical potential flux approximation. We demonstrate that the fully discrete nonlinear scheme is unconditionally stable, preserving the free-energy structure of the continuous system and ensuring the non-negativity of carrier densities. Novel discrete entropy-dissipation inequalities for both boundary condition types in multiple dimensions allow us to prove the existence of discrete solutions. We perform multi-dimensional simulations to understand the impact of electrode configurations and device geometries, focusing on the hysteresis behavior in lateral 2D memristive devices. Three electrode configurations -- side, top, and mixed contacts -- are compared numerically for different geometries and boundary conditions. These simulations reveal the conditions under which a simplified one-dimensional electrode geometry can well represent the three electrode configurations. This work lays the foundations for developing accurate, efficient simulation tools for 2D memristive devices and memtransistors, offering tools and guidelines for their design and optimization in future applications.
    \end{abstract}

%%%%%%%%%%%%%%%%%%%%%%%%%%%%%%%%%%%%%%%%%%%%%%%%%%%%%%%%%%%%%%%%%%%%
\section{Introduction}
%%%%%%%%%%%%%%%%%%%%%%%%%%%%%%%%%%%%%%%%%%%%%%%%%%%%%%%%%%%%%%%%%%%%

With the increasing importance of artificial intelligence (AI) in today's technological landscape, there is a need for hardware that can efficiently process vast amounts of data while minimizing energy consumption and latency \cite{Jones.2018, Lin.2021, Bughin.2018}.
In particular, as AI models grow in complexity, conventional computing architectures struggle to meet the demands of memory and energy efficiency \cite{Rao.2023, Xu.2018}.
A promising solution lies in developing memristive AI accelerators, which implement artificial neural networks (ANNs) directly in hardware using memristive devices \cite{Ziegler.2018, Sung.2018, Choi.2020}.
Memristive devices are microelectronic devices with a pinched hysteresis in their current-voltage (I-V) characteristics as illustrated in Fig.~\ref{fig:intro-TMDC}a,b.
This feature allows them to store multilevel, nonvolatile memory states and offer a low-power, scalable alternative to traditional components \cite{ Zidan.2018, Song.2023}.
When assembled in crossbar arrays, memristive devices enable matrix-vector multiplication, a key operation in ANNs, to be carried out in a single step \cite{Prezioso.2015, Huang.2024, Choi.2020, Boybat.2018}.
This concept permits computing directly in memory and can dramatically reduce energy consumption and latency by avoiding the costly data shuttling between memory and processing units inherent to conventional von Neumann computing architectures \cite{Gebregiorgis.2023, Li.2023b}.
Recent advances in the development of memristive AI accelerators have demonstrated their potential in various proof-of-concept applications, from image classification to speech recognition and other complex AI tasks, achieving near-software-level accuracy while maintaining low energy consumption; see, e.g., \cite{Huang.2024, Lanza.2022, Xia.2019, Aguirre.2024}.

\begin{figure}[ht]
    \centering
   \hspace*{-1.2cm} \includegraphics[width=1.2\textwidth]{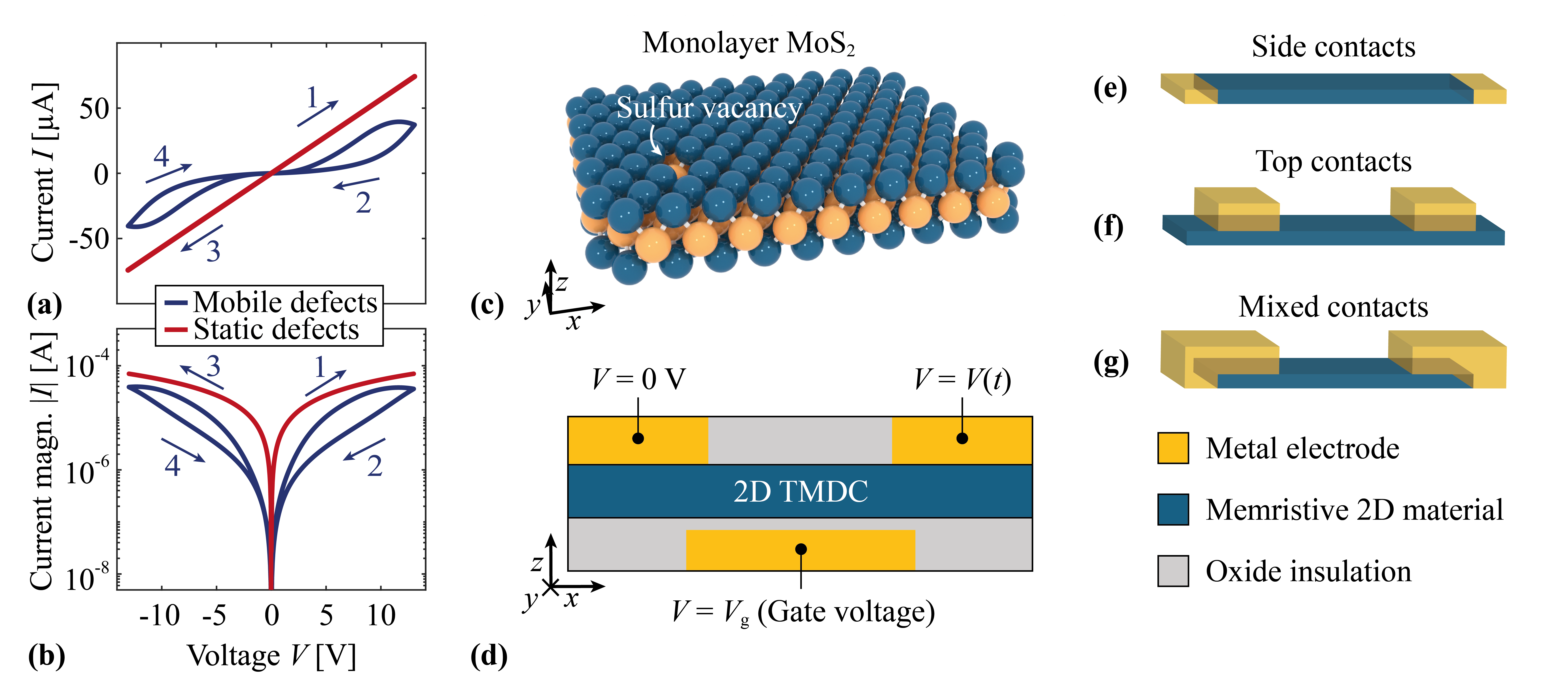}
    \caption{
        \textbf{(a)} Simulated I-V curve of an MoS$_2$ memristive device showing a pinched hysteresis if mobile ionic defects are present (darkblue) and no hysteresis if the defects are immobile (red).
        \textbf{(b)} Typical representation of the curve in (a) on a semilogarithmic scale.
        \textbf{(c)} Atomic structure of 2D MoS$_\mathrm{2}$ (monolayer), as an example for a TMDC widely used for lateral memristive devices and memtransistors with indicated sulfur vacancy.
        \textbf{(d)} Cross sectional illustration of a three terminal memtransistor based on 2D MoS$_\mathrm{2}$, with two top electrodes and one bottom gate electrode.
        In contrast to the top electrodes, the gate electrode is electrically insulated from the 2D TMDC in the center.
        \textbf{(e-g)} Illustration of three different electrode configurations investigated in this work that are used in lateral memristive devices and memtransistors: side contacts, top contacts, and mixed contacts.
        }
    \label{fig:intro-TMDC}
\end{figure}

In parallel with these advances, the exploration of 2D materials has significantly extended the scope of memristive device technologies \cite{Song.2023, Berggren.2021}.
These materials consist of atomically thin crystalline layers (Fig.~\ref{fig:intro-TMDC}c), which results in various beneficial features such as tunable electronic behavior, atomically sharp interfaces, reduced electrostatic screening, and ultimate scalability down to atomic dimensions \cite{Berggren.2021, Yao.2019, Han.2011}.
Furthermore, 2D materials have been investigated for multi-terminal memristive devices, often called memtransistors, because they unify memristive properties with the gate tunability of field-effect transistors, see, e.g., \cite{Sangwan.2018c, Feng.2021, Leng.2023, Ding.2021, Wali.2023}.
Memtransistors employ lateral geometries with additional gate electrodes (terminals) as illustrated in Fig.~\ref{fig:intro-TMDC}d-g.
The gate electrodes extend the device functionality by permitting electrical tuning of device characteristics such as linearity, symmetry, and learning rates by applying voltages to the gates \cite{Sangwan.2018c, Ding.2021, Lee.2020b, Yan.2022, Rodder.2020}.
While this technology is still at an early stage, recent empirical work on memtransistors based on 2D materials has achieved remarkable improvements in performance metrics, such as ultra-low switching energies of 20 fJ/bit, low switching voltages $<1$~V, and reduced sneak path currents in array architectures \cite{Feng.2021, Lee.2020b, Rahimifard.2022, Liu.2024}.

Further advances in memtransistor performance can be expected by developing a detailed understanding of the physical mechanisms of memristive hysteresis \cite{Song.2023}.
A critical challenge is to develop a model that captures all essential aspects of charge transport to identify the influence of geometry and material parameters on the I-V curve.
For example, most 2D memristive devices and memtransistors are based on transition metal dichalcogenides (TMDCs) as a memristive material \cite{Su.2021, Wang.2022b, Wang.2020c, Sangwan.2018, Geim.2013}.
For many devices, the migration of mobile atomistic defects was suggested to be the origin of memristive hysteresis \cite{Sangwan.2018c, Jadwiszczak.2019, DaLi.2018, Ge.2018, Sangwan.2015}.
In 2D MoS$\mathrm{_2}$ (a widely studied TMDC), such defects are thought to be charged sulfur vacancies, i.e., missing sulfur atoms in the lattice structure (see Fig.~\ref{fig:intro-TMDC}c) \cite{Sangwan.2018c, Jadwiszczak.2019, DaLi.2018, Sangwan.2015}.
The coupled dynamics of defects, electrons and holes complicates the charge transport processes compared to other semiconductor devices \cite{Thakkar.2024}.
Because the field effect from the gate electrodes further adds complexity, the first step is to develop a computational model with realistic memtransistor and electrode geometries but without the gate electrode.
In fact, previous theoretical studies have explained the pinched hysteresis in such devices by the accumulation and depletion of ionic defects at the electrodes \cite{Sivan.2022, Spetzler.2024}.
However, the large variety of electrode and device structures used for memtransistors combined with the complex microscopic processes lead to a lack of insights into these devices' charge transport and hysteresis.

%% model%%
The few computational models presented for memtransistors and lateral memristive devices based on 2D materials include compact models \cite{Sangwan.2018c, Spetzler.2022, Zhou.2023b}, kinetic Monte Carlo models \cite{Aldana.2023b, Aldana.2023}, and continuous charge transport models \cite{Sivan.2022,  Spetzler.2024}.
Compact models use simplified equivalent circuit approximations inspired by physical laws \cite{Ielmini.2017}.
Such models can efficiently integrate the electrical device characteristics in circuit simulations on the system level but do not capture the complex microscopic charge transport dynamics \cite{Ielmini.2017, Gao.2021}.
The kinetic Monte Carlo models consider the stochastic migration of defects in nanoscale volumes at the expense of strongly simplified semiconductor physics and omitting the electrodes \cite{Aldana.2023b, Aldana.2023}.

%% vacancy-assisted model %%
Recently, vacancy-assisted drift-diffusion models with Schottky boundary conditions have been introduced to describe the charge transport in 2D memristive devices on the micro- to mesoscale, encompassing fully time-dependent \cite{Spetzler.2024} and quasi-static approximations \cite{Sivan.2022}.
The major difference between these models and other classical drift-diffusion approaches \cite{Farrell.2017, Markowich.1990} is the additional equation for mobile defects.
These defects migrate on different timescales with different nonlinear dynamics than electrons and holes \cite{Abdel.2023}.
While vacancy-assisted drift-diffusion models exist for various materials and applications \cite{Strukov.2009, Marchewka.2016, Jourdana.2023, CuestaLopez.2024, Calado.2022, Aoki.2014, Cances.2023, Bataillon.2010}, most of them are not generally applicable to memristive devices.
And even applicable models often assume simplified scenarios, such as side contacts depicted in \Cref{fig:intro-TMDC}e, employ ohmic boundary conditions, or neglect the limitation of vacancy accumulation, commonly referred to as volume exclusion effects.
An exception is the model presented in \cite{Spetzler.2024} which offers a fully time-dependent quasi Fermi level formulation of the drift-diffusion equations of electrons, holes, and ionic defects coupled to the nonlinear Poisson's equation.

Concerning the mathematical study of vacancy-assisted drift–diffusion models, the existence of weak solutions under Ohmic boundary conditions has been proven for memristive devices with general nonlinear dynamics in \cite{Herda.2024}.
Similar results have been established for perovskite solar cells \cite{Abdel.2024b, glitzky2024uniqueness}.
The latter are directly applicable to memristor and memtransistor configurations, since only the computational domain and the electron/hole source terms differ.
To our knowledge, there is currently no corresponding mathematical analysis for Schottky boundary conditions for such semiconductor devices.

%% discretzation %%
Spatial discretization techniques like the finite difference (FD), finite element (FE), and finite volume (FV) methods are commonly employed to solve systems of coupled partial differential equations (PDEs) \cite{Selberherr.1984, Bank.1983, Brezzi.2005, Mock.1983}, thereby connecting microscopic features with device performance.
In particular, the FE and FV methods are widely used for device simulations because they can easily integrate irregular meshes to resemble complex geometries accurately and include geometric mesh refinement over orders of magnitude \cite{Selberherr.1984, Bank.1983, Brezzi.2005, Mock.1983}.
The FV method additionally reflects physical core principles correctly, e.g., it locally conserves fluxes and is consistent with thermodynamic laws, making it particularly suitable for flow problems \cite{Farrell2017a, Eymard.2000}.
This has led to significant research into the design and analysis of numerical FV schemes for drift-diffusion models (e.g., \cite{Cances.2021, Kantner.2020, BessemoulinChatard.2012, BessemoulinChatard.2017, Glitzky.2011, Moatti.2023}).
In \cite{Abdel2023Existence}, an implicit in-time two-point flux approximation (TPFA) finite volume scheme for the charge transport in perovskite solar cells was analyzed, and the existence of discrete solutions was proven.
As the computational models for memristive devices and memtransistors differ in the non-dimensionalization of the PDE system, the boundary conditions, and the device geometry, we adapt in this work the results of \cite{Abdel2023Existence} to prove the existence of discrete solutions for the memristive charge transport model based on both contact boundary models, ohmic and Schottky.
The analysis relies on deriving a key \emph{a priori} estimate for solutions of the system, referred to as entropy-dissipation inequality.
The discrete version of this estimate is instrumental in showing the existence of solutions to the numerical scheme and in assessing the stability of the discretization.
The inequality is derived from a functional related to the physical  \textit{free energy}.
Still, we adopt the term \textit{entropy} to align with the terminology commonly used in the mathematical literature on nonlinear dissipative PDEs \cite{Jungel.2016}.

%% Simulation %%
In contrast to earlier one-dimensional (1D) numerical studies \cite{Spetzler.2024,Jourdana.2023}, 2D simulations offer a more realistic representation of the semiconductor-metal contacts in memtransistors, as illustrated by the device configurations in \Cref{fig:intro-TMDC}e-g. This motivates us to complement our theoretical results with multi-dimensional simulations.
Given that both, Schottky \cite{Spetzler.2024, Sivan.2022} and ohmic boundary models \cite{Jourdana.2023, Jungel.2023, Herda.2024}, have been employed in the literature, we explore the impact of these boundary conditions on device characteristics in detail.
Lastly, we examine the three contact configurations shown in \Cref{fig:intro-TMDC}e-g to determine and quantify the conditions under which simplified contact geometries can accurately approximate the behavior of full 2D contact geometries.
These simulations provide valuable insights into the influence of contact geometry and boundary conditions on device performance, offering guidance for the design and optimization of memristive devices and memtransistors.

%% outline %%
This work is organized as follows:
In \Cref{sec:model}, we introduce the underlying charge transport equations, the two different boundary models (ohmic and Schottky), and a suitable non-dimensionalized version of the model.
Next, in \Cref{sec:discrete_model}, we present the finite volume scheme and prove its properties: unconditional stability in entropy, non-negativity of densities, conservation of mass for the ions and existence of solutions.
This is done for Schottky as well as ohmic boundary conditions.
Finally, in \Cref{sec:numerics}, we compare the introduced contact boundary models, ohmic and Schottky, numerically.
Additionally, we analyze the impact of three different contact geometries on the charge transport and the memristive hysteresis before we conclude in \Cref{sec:conclusion}.

%%%%%%%%%%%%%%%%%%%%%%%%%%%%%%%%%%%%%%%%%%%%%%%%%%%%%%%%%%%%%%%%%%%%%%%%%%%
%%%%%%%%%%%%%%%%%%%%%%%%%%%%%%%%%%%%%%%%%%%%%%%%%%%%%%%%%%%%%%%%%%%%%%%%%%%
\section{Modeling charge transport in memristive devices} \label{sec:model}
%%%%%%%%%%%%%%%%%%%%%%%%%%%%%%%%%%%%%%%%%%%%%%%%%%%%%%%%%%%%%%%%%%%%%%%%%%%
%%%%%%%%%%%%%%%%%%%%%%%%%%%%%%%%%%%%%%%%%%%%%%%%%%%%%%%%%%%%%%%%%%%%%%%%%%%

This section formulates the charge transport model in \Cref{subsec:model}, incorporating both the Schottky and ohmic contact boundary models and a suitable non-dimensionalization for the numerical analysis.
Following this, \Cref{sec:cont_entropy_dissipation} introduces the definition of entropy functions and establishes a continuous entropy-dissipation inequality.

%%%%%%%%%%%%%%%%%%%%%%%%%%%%%%%%%%%%%%%%%%%%%%%%%%%%%%%%%%%%%%%%%%%%%%%%%%%
%%%%%%%%%%%%%%%%%%%%%%%%%%%%%%%%%%%%%%%%%%%%%%%%%%%%%%%%%%%%%%%%%%%%%%%%%%%
\subsection{Charge transport model} \label{subsec:model}
%%%%%%%%%%%%%%%%%%%%%%%%%%%%%%%%%%%%%%%%%%%%%%%%%%%%%%%%%%%%%%%%%%%%%%%%%%%
%%%%%%%%%%%%%%%%%%%%%%%%%%%%%%%%%%%%%%%%%%%%%%%%%%%%%%%%%%%%%%%%%%%%%%%%%%%

Let $\mathbf{\Omega} \subset \mathbb{R}^d$, $d \in \{ 1, 2, 3\}$, be an open, connected, and bounded spatial domain representing the TMDC material layer.
The TMDC layer is typically placed on a substrate and laterally sandwiched between two electrode contacts, as shown in \Cref{fig:device-schematics-memristor}.
The common electrode configurations include side, top, and mixed contacts, with the interface between the electrodes and the TMDC layer denoted by $\mathbf{\Gamma}^C$.
We consider the transport of three carriers: electrons ($\alpha = \electrons$), holes ($\alpha = \holes$), and ionic defects ($\alpha = \ions$).
Our approach closely follows the charge transport model for TMDC-based memristive devices formulated in \cite{Spetzler.2024}.
The set of unknowns consists of $(\varphi_\electrons, \varphi_\holes, \varphi_\ions, \psi)$, where $\varphi_\alpha$ represents the quasi Fermi potential of carrier species $\alpha \in \{ \electrons, \holes, \ions \}$, and $\psi$ is the electrostatic potential.

\begin{figure}[ht]
    \hspace*{-0.8cm}
    \includegraphics[scale = 1.0]{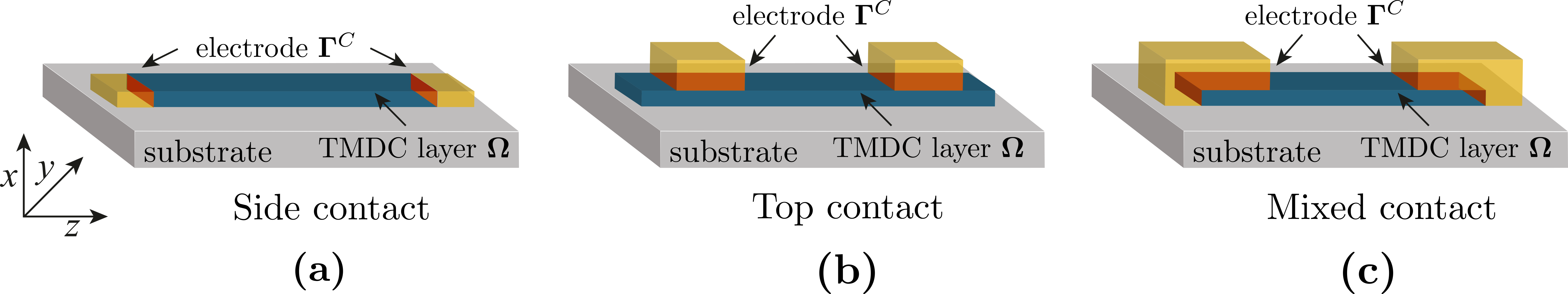}
    \caption{Illustration of a three-dimensional geometry of a memristive device for a \textbf{(a)} side, \textbf{(b)} top, and \textbf{(c)} mixed contact with indicated substrate and TMDC layer $\mathbf{\Omega}$.
    Furthermore, the contact boundary $\mathbf{\Gamma}^C$, defined as the interfaces of the contacts (gold) and the TMDC material (blue), are highlighted in red for each configuration.
        }
    \label{fig:device-schematics-memristor}
\end{figure}

%%%%%%%%%%%%%%%%%%%%%%

For $t\geq 0$ the electron, hole, and ionic defect densities $n_\electrons, n_\holes, n_\ions$ satisfy the continuity equations
\begin{subequations} \label{eq:model-memristor}
    \begin{align} \label{eq:model-cont-eq-memristor}
        z_\alpha q \partial_t n_\alpha + \nabla\cdot \mathbf{j}_\alpha = 0, \quad { \mathbf{x} }\in\mathbf{\Omega},\ t\geq 0, \quad \text{for}\; \alpha \in \{ \electrons, \holes, \ions\},
    \end{align}
    which are self-consistently coupled via the electrostatic potential $\psi$ to the nonlinear Poisson equation
    \begin{align}\label{eq:model-poisson-memristor}
        - \nabla \cdot (\varepsilon_s \nabla \psi ) =
        \!\!\!\! \sum_{\alpha \in \{ \electrons, \holes, \ions \}} \!\!\!\! z_\alpha q n_\alpha
        + z_C q C({ \mathbf{x} }),
        ~
        \quad   \mathbf{x} \in\mathbf{\Omega},\ t\geq 0.
    \end{align}
The quantity $z_\alpha$ is the charge number of the moving charge carrier species $\alpha \in \{ \electrons, \holes, \ions\}$, $q$ the positive elementary charge and $\mathbf{j}_\alpha$ the carrier dependent current density.
For electrons and holes, we assume $z_\electrons=-1$ and $z_\holes=1$.
For the ionic defects, we have $z_\ions \in \mathbb{Z}$.
Moreover, we define $\varepsilon_s = \varepsilon_0 \varepsilon_r$ as the dielectric permittivity of the TMDC material given as a product of the vacuum permittivity $\varepsilon_0$ and the relative permittivity $\varepsilon_r$.
Furthermore, $C \in L^{\infty}(\mathbf{\Omega})$ denotes one type of background charge density with a respective sign $z_C \in \{ -1, 1\}$.
For all carriers, we assume no reaction or production rates in \eqref{eq:model-cont-eq-memristor}.
The current densities contributing to the continuity equations \eqref{eq:model-cont-eq-memristor} are defined as
\begin{equation}  \label{eq:cont-flux-memristor}
    \mathbf{j}_\alpha = - z_\alpha^2 q \mu_\alpha n_\alpha \nabla\varphi_\alpha,
    \quad \mathbf{x} \in\mathbf{\Omega},\ t\geq 0,
    ~
    \quad \text{for}\; \alpha \in \{ \electrons, \holes, \ions \},
\end{equation}
where $\mu_\electrons, \mu_\holes, \mu_\ions$ are the carrier-dependent mobilities.
We have formulated the model equations based on the quasi Fermi potentials, the electrostatic potential, and the charge carrier densities.
A state equation connects these quantities
\begin{equation} \label{eq:state-eq}
    n_\alpha = N_\alpha \mathcal{F}_\alpha \Bigl(\eta_\alpha(\varphi_\alpha, \psi) \Bigr), \quad \eta_\alpha = z_\alpha \frac{q(\varphi_\alpha - \psi) + E_\alpha}{k_B T}, \quad \alpha \in \{ \electrons, \holes, \ions \}.
\end{equation}
\end{subequations}
Here, $N_\electrons, N_\holes$ represent the effective densities of states of the conduction and valence bands, respectively, while $N_\ions$ denotes the saturation limit for the ionic defect concentration.
The function $\mathcal{F}_\alpha$ is referred to as statistics function, and its argument $\eta_\alpha$ is called the dimensionless chemical potential.
The dimensionless chemical potential depends on the intrinsic energy level $E_\alpha$, the temperature $T$, and the Boltzmann constant $k_B$.
For electrons and holes, we also call $E_\electrons, E_\holes$ the intrinsic band edge energies of the conduction and valence bands, respectively.
Lastly, the choice of statistics functions $\mathcal{F}_\alpha$ in \eqref{eq:state-eq} depends on the carrier species $\alpha\in\{\electrons, \holes, \ions\}$.
For inorganic three-dimensional semiconductor devices, the Fermi-Dirac integral of order $1/2$ is typically used for electrons and holes
\begin{equation} \label{eq:Fermi-Dirac}
    F_{1/2}(\eta) = \frac{2}{ \sqrt{\pi} } \int_{0}^{\infty} \frac{\xi^{1/2}}{\exp(\xi - \eta) + 1}\: \mathrm{d}\xi, \quad\eta \in \mathbb{R},
\end{equation}
i.e., $\mathcal{F}_\electrons = \mathcal{F}_\holes = F_{1/2}$.
Contrarily, for the ionic defects $\alpha = \ions$, we choose the Fermi-Dirac integral of order $-1$,
\begin{align} \label{eq:FD-1-statistics}
    F_{-1}(\eta) = \frac{1}{  \exp(-\eta) + 1}, \quad \eta \in \mathbb{R},
\end{align}
i.e., $\mathcal{F}_\ions = F_{-1}$, which limits the ionic defect accumulation to a maximum defect density \cite{Abdel.2021, Abdel.2023}.
Generally, for the upcoming mathematical analysis of the model, we assume for $\mathcal{F}_\alpha$, $\alpha\in\{\electrons, \holes, \ions\}$,
\begin{equation}\tag{H1} \label{hyp:statistics-n-p}
    \left\{
    \begin{aligned}
        &\mathcal{F}_\electrons, \mathcal{F}_\holes: \mathbb{R} \rightarrow (0,\infty)\text{ are }  C^1\text{- diffeomorphisms};\\[.5em]
        &0 < \mathcal{F}_\alpha'(\eta) \leq \mathcal{F}_\alpha(\eta) \leq \exp(\eta), \quad\eta \in \mathbb{R},\ \alpha \in\{\electrons, \holes\} .
    \end{aligned}
    \right.
\end{equation}
and
\begin{equation}\tag{H2} \label{hyp:statistics-a}
    \left\{
    \begin{aligned}
        &\mathcal{F}_\ions: \mathbb{R} \rightarrow (0,1)\text{ is a }  C^1\text{- diffeomorphism};\\[.5em]
        &0 < \mathcal{F}_\ions'(\eta) \leq \mathcal{F}_\ions(\eta) \leq \exp(\eta), \quad\eta \in \mathbb{R}.
    \end{aligned}
    \right.
\end{equation}
Indeed, the introduced statistics functions \eqref{eq:Fermi-Dirac} and \eqref{eq:FD-1-statistics} satisfy these hypotheses \cite{Abdel.2023}.
Moreover, direct implications for $\mathcal{F}_\alpha$, $\alpha \in \{ \electrons, \holes, \ions \}$, are the inequalities
\begin{align} \label{eq:statistics-ineq}
        ( \mathcal{F}_\alpha^{-1} )' \bigl(\frac{n_\alpha}{N_\alpha} \bigr)
    = \left( \mathcal{F}'_\alpha \left( \mathcal{F}_\alpha^{-1} \bigl(\frac{n_\alpha}{N_\alpha} \bigr) \right) \right)^{-1} \geq \frac{1}{n},
    \quad
    ~
    \log\bigl(\frac{n_\alpha}{N_\alpha} \bigr) \leq\mathcal{F}_\alpha^{-1} \bigl(\frac{n_\alpha}{N_\alpha} \bigr).
\end{align}
We can formulate a generalized Einstein relation for $\alpha \in \{\electrons, \holes, \ions \}$,
\begin{equation} \label{eq:generalized-Einst-electric}
	D_\alpha\bigl(\frac{n_\alpha}{N_\alpha} \bigr) =\mu_\alpha U_T g_\alpha \bigl(\frac{n_\alpha}{N_\alpha} \bigr),
    \quad  g_\alpha \bigl(\frac{n_\alpha}{N_\alpha} \bigr)= \frac{n_\alpha}{N_\alpha} \left( \mathcal{F}_\alpha^{-1} \right)'\bigl(\frac{n_\alpha}{N_\alpha} \bigr),
\end{equation}
where $U_T$ denotes the thermal voltage and the quantity $g_\alpha$ refers to the nonlinear diffusion enhancement \cite{Abdel.2021}.
With the generalized Einstein relation and \eqref{eq:state-eq} we can rewrite the current densities \eqref{eq:cont-flux-memristor} in a drift-diffusion form with a density-dependent diffusion
\begin{align} \label{eq:fluxes-DD-form}
	\mathbf{j}_\alpha  = -z_\alpha q \Bigl( D_\alpha\bigl(\frac{n_\alpha}{N_\alpha} \bigr)  \nabla n_\alpha + z_\alpha \mu_\alpha n_\alpha \nabla \psi \Bigr).
\end{align}
%%%%%%%%%%%%%%%%%%%%%%%%%%%%%%%%%%%%%%
Finally, the system \eqref{eq:model-memristor} is supplied with initial conditions for $t = 0$
\begin{align} \label{eq:initial-cond-memristor}
    \varphi_\electrons(\mathbf{x}, 0) = \varphi_\electrons^0(\mathbf{x}),
    ~
    \quad\varphi_\holes(\mathbf{x}, 0) = \varphi_\holes^0(\mathbf{x}),
    ~
    \quad \varphi_\ions(\mathbf{x}, 0) = \varphi_\ions^0(\mathbf{x}),
\quad  \mathbf{x} \in \boldsymbol{\Omega},
\end{align}
where we assume $\varphi_\electrons^0, \varphi_\holes^0, \varphi_\ions^0 \in L^{\infty}(\mathbf{\Omega})$.
We define the initial densities via $n_\alpha^0(\mathbf{x}) = N_\alpha\mathcal{F}_\alpha(\eta_\alpha(\varphi_\alpha^0,\psi(\mathbf{x},0)))$ for $\alpha \in \{ \electrons, \holes, \ions \}$.

%%%%%%%%%%%%%%%%%%%%%%%%%%%%%%%%%%%%%%%%%%%%%%%%%%%%%%%%%%%%%%%%%%%%%%%%%%%
\subsubsection{Boundary conditions} \label{sec:BC}
%%%%%%%%%%%%%%%%%%%%%%%%%%%%%%%%%%%%%%%%%%%%%%%%%%%%%%%%%%%%%%%%%%%%%%%%%%%
We divide the outer boundary of the device geometry $\mathbf{\Omega}$ into two parts:  $\mathbf{\Gamma}^C$ and $\mathbf{\Gamma}^N$.
Here, $\mathbf{\Gamma}^C$ denotes the interface between the semiconductor and the electrodes, as highlighted in red in \Cref{fig:device-schematics-memristor}, while $\mathbf{\Gamma}^N$ refers to the remaining boundaries.
We assume that $\boldsymbol{\Gamma}^C$ and $\boldsymbol{\Gamma}^N$ are closed subsets of $\partial\boldsymbol{\Omega}$ with $\partial \mathbf{\Omega} = \mathbf{\Gamma}^C \cup \mathbf{\Gamma}^N$.
For the ionic point defects all boundaries represent physical barriers described by an isolating Neumann boundary condition
\begin{subequations}
    \begin{align} \label{eq:BC-a-memristor}
        \mathbf{j}_\ions(\mathbf{x}, t) \cdot {\boldsymbol{\nu}}(\mathbf{x}) =
        % ~
        0, \quad \mathbf{x} \in \partial \mathbf{\Omega}, \ t \geq 0,
    \end{align}
    where $\boldsymbol{\nu}$ is the outward pointing unit normal to $\partial \mathbf{\Omega}$.
    At the isolating Neumann boundary we impose for the other species zero flux conditions
    \begin{align} \label{eq:BC-GammaN-memristor}
        \nabla \psi(\mathbf{x}, t) \cdot {\boldsymbol{\nu}}(\mathbf{x}) =
        ~
        \mathbf{j}_\electrons(\mathbf{x}, t) \cdot {\boldsymbol{\nu}}(\mathbf{x}) =
        ~
        \mathbf{j}_\holes(\mathbf{x}, t) \cdot {\boldsymbol{\nu}}(\mathbf{x}) =
        ~
        0, \quad \mathbf{x} \in \mathbf{\Gamma}^N, \ t \geq 0.
    \end{align}
\end{subequations}

%%%%%%%%%%%%%%%%%%%%%%%%%%%%%%%%%%%%%%%%%%%%%%%%%%%%%%%%%%%%%%%%%%%%%
\noindent
\textbf{1. Schottky boundary model.}
A Schottky model is commonly assumed at the metal-semiconductor contact $\mathbf{\Gamma}^C$ for memristive devices.
The Schottky boundary conditions describe thermionic emission of electrons and holes via
\begin{subequations} \label{eq:BC-Schottky-time-dependent}
    \begin{align}
        \mathbf{j}_\electrons(\mathbf{x}, t) \cdot {\boldsymbol{\nu}}(\mathbf{x}) &= z_\electrons q v_\electrons (n_{\mathrm{n}}(\mathbf{x}, t) - n_{\mathrm{n}, 0}),  \quad \mathbf{x} \in \mathbf{\Gamma}^C, \ t \geq 0, \label{eq:BC-cond-n-memristor} \\
        ~
        \mathbf{j}_\holes(\mathbf{x}, t) \cdot {\boldsymbol{\nu}}(\mathbf{x}) &= z_\holes q v_\holes (n_\holes(\mathbf{x}, t) - n_{\holes, 0} ), \quad \mathbf{x} \in  \mathbf{\Gamma}^C, \ t \geq 0. \label{eq:BC-cond-p-memristor}
    \end{align}
    Here, $v_\electrons, v_\holes \geq 0 $ are the electron and hole recombination velocities, respectively.
    The equilibrium carrier densities $n_{\electrons, 0}$, and $n_{\holes, 0}$ at the contacts are given by $n_{\alpha, 0} := N_\alpha \mathcal{F}_\alpha( \eta_\alpha (\varphi_0, \psi_0))$ for $\alpha \in \{ \electrons, \holes\}$, where $\varphi_0$ is the constant quasi Fermi potential at equilibrium, and $\psi_0 = -(\phi_0 - E_\electrons)/q $ denotes the intrinsic electrostatic potential barrier.
    In the expression for $\psi_0$, we have $\phi_0 = \phi_0(\mathbf{x}) >0$ as the intrinsic Schottky energy barrier constant.
    For the electrostatic potential at the metal-semiconductor interface $\mathbf{\Gamma}^C$, we apply the Dirichlet condition
    \begin{align} \label{eq:BC-psi-Schottky}
    \psi(\mathbf{x}, t) = \psi_0 ( \mathbf{x}) + V(\mathbf{x}, t)
    ~
    , \quad \mathbf{x} \in  \mathbf{\Gamma}^C, \ t \geq 0,
    \end{align}
    where $V$ is a time-dependent voltage applied at the contact.
\end{subequations}

%%%%%%%%%%%%%%%%%%%%%%%%%%%%%%%%%%%%%%%%%%%%%%%%%%%%%%%%%%%%%%%%%%%%%
\noindent
\textbf{2. Ohmic boundary model.}
As we will see in the simulations, for specific scenarios, we can replace the Schottky contact boundary model with an ohmic contact boundary model
\begin{subequations}\label{eq:BC-Ohmic-time-dependent}
    \begin{alignat}{2}
        \psi(\mathbf{x}, t) &= \psi_0 ( \mathbf{x}) + V(\mathbf{x}, t)
        ~
        , \quad &&\mathbf{x} \in  \mathbf{\Gamma}^C, \ t \geq 0,\\
        ~
        \varphi_\electrons(\mathbf{x}, t)  = \varphi_\holes(\mathbf{x}, t) &= V (\mathbf{x}, t), \quad
        ~
        &&\mathbf{x} \in \mathbf{\Gamma}^C,  \ t \geq 0, \label{eq:BC-psi-Dirichlet}
    \end{alignat}
\end{subequations}
where $V$ denotes again a time-dependent applied voltage.
In other words, the Robin boundary conditions for electrons and holes \eqref{eq:BC-cond-n-memristor}, \eqref{eq:BC-cond-p-memristor} are replaced by a Dirichlet condition.

%%%%%%%%%%%%%%%%%%%%%%%%%%%%%%%%%%%%%%%%%%%%%%%%%%%%%%%%%%%%%%%%%%%%%
\noindent
\textbf{3. Time-constant applied voltage for Schottky and ohmic contact boundary models.}
In order to avoid unnecessary technicalities in the exposition we make the additional assumption of a time-constant applied voltage, i.e., $V(\mathbf{x}, t) = \overline{V} (\mathbf{x})$, $\mathbf{x} \in \mathbf{\Gamma}^C$.
Typically, $\overline{V}$ is constant along each electrode.
The proofs are valid for both contact boundary models, the Schottky boundary conditions \eqref{eq:BC-Schottky-time-dependent} and the ohmic contact boundary conditions \eqref{eq:BC-Ohmic-time-dependent}.
Our stability and existence results can be extended to time-varying applied voltages under some mild assumptions on the boundary data.
We refer to Remark~\ref{rem.time_varying} for more details.
This extension is particularly relevant for modeling pulse measurements \cite{Spetzler.2024}, where devices are operated by applying repetitively short constant-voltage pulses.

\noindent
In the following, let $\psi^D, \varphi^D \in W^{1, \infty}(\boldsymbol{\Omega})$ be given Dirichlet functions defined on $\mathbf{\Omega}$.

%%%%%%%%%%%%%%%%%%%%%%%%%%%%%%%%%%%%%%%%%%%%%%%%%%%%%%%%%%%%%%%%%%%%%
\noindent
\textbf{3a. Schottky boundary model.}
On the one hand, for the Schottky boundary model, we adjust the outer boundary conditions \eqref{eq:BC-Schottky-time-dependent} to
\begin{subequations}\label{eq:BC-Schottky-const-volt}
    \begin{alignat}{2}
        \psi(\mathbf{x}, t) &= \tilde{\psi}^D(\mathbf{x}), \label{eq:BC-Schottky-const-volt-psi}\\
        ~
        \mathbf{j}_\alpha (\mathbf{x}, t) \cdot \boldsymbol{\nu} (\mathbf{x}) &= z_\alpha v_\alpha \left( n_\alpha(\mathbf{x}, t) - n_\alpha^D \right), \quad \text{for} \; \alpha \in \{\electrons, \holes \}, \label{eq:BC-Schottky-const-volt-n-p}
    \end{alignat}
\end{subequations}
where $\mathbf{x} \in \mathbf{\Gamma}^C,  t \geq 0$ and $n_\alpha^D := N_\alpha \mathcal{F}_\alpha \left( \eta_\alpha(\varphi^D, \psi^D) \right)$ and $ \tilde{\psi}^D(\mathbf{x}) := \psi^D(\mathbf{x}) + \overline{V}$ with a fixed voltage $\overline{V}$.

%%%%%%%%%%%%%%%%%%%%%%%%%%%%%%%%%%%%%%%%%%%%%%%%%%%%%%%%%%%%%%%%%%%%%
\noindent
\textbf{3b. Ohmic boundary model.}
On the other hand, for the ohmic contact boundary model, we adjust the outer boundary conditions \eqref{eq:BC-Ohmic-time-dependent} to
\begin{alignat}{2}\label{eq:BC-ohmic-const-volt}
    \psi(\mathbf{x}, t) = \psi^D(\mathbf{x}), \;\; \varphi_\electrons(\mathbf{x}, t)  = \varphi_\holes(\mathbf{x}, t) &= \varphi^D(\mathbf{x}), \quad
    ~
    &&\mathbf{x} \in \mathbf{\Gamma}^C,  \ t \geq 0.
\end{alignat}
Note that we consider in \eqref{eq:BC-Schottky-const-volt} and \eqref{eq:BC-ohmic-const-volt} the traces of the functions $\psi^D, \varphi^D \in W^{1, \infty}(\boldsymbol{\Omega})$.
Until the end of \Cref{sec:discrete_model}, we focus exclusively on two charge transport models, supplied with Schottky and ohmic boundary conditions with a voltage constant in time at the semiconductor-metal contact $\mathbf{\Gamma}^C$.

%%%%%%%%%%%%%%%%%%%%%%%%%%%%%%%%%%%%%%%%%%%%%%%%%%%%%%%%%%%%%%%%%%%%%%%%%%%
%%%%%%%%%%%%%%%%%%%%%%%%%%%%%%%%%%%%%%%%%%%%%%%%%%%%%%%%%%%%%%%%%%%%%%%%%%%
\subsubsection{Non-dimensionalization of the model} \label{sec:nondimens-model}
%%%%%%%%%%%%%%%%%%%%%%%%%%%%%%%%%%%%%%%%%%%%%%%%%%%%%%%%%%%%%%%%%%%%%%%%%%%
%%%%%%%%%%%%%%%%%%%%%%%%%%%%%%%%%%%%%%%%%%%%%%%%%%%%%%%%%%%%%%%%%%%%%%%%%%%

In this section, we derive a non-dimensionalized version of the charge transport model \eqref{eq:model-memristor}, which serves as a reference model for the numerical analysis.
Following \cite[Section 2.4]{markowich1985stationary}, the equations are expressed using the scaled variables, defined as the ratio of the unscaled quantity to the scaling factors in \Cref{tab:scalingfactors-memristors}.
We use the values in \Cref{tab:TMDC-general} for a MoS$_2$-based memristive device as a reference parameter set.

%%%%%%%%%%%%%%%%%%%%%%%%%%%%%%%%
\begin{table}[!ht]
\centering
{\footnotesize
    \begin{tabular}{|c|l|c|l|}
            \hline Symbol & Meaning & Scaling factor & Order of magnitude \\\hline&&&\\
            $\:\:\mathbf{x}$ & Space vector & $l$ & $10^{-6}$ m\\[2.0ex]
            ~
            $\varphi_\alpha$, $\varphi^D$, $\varphi_\alpha^0$  &  Quasi Fermi potentials & $U_T$& $10^{-2}$ V\\[2.0ex]
            ~
            $\psi$, $\psi^D$ & Electric potential & $U_T$ & $10^{-2}$ V\\[2.0ex]
            ~
            $\overline{V}$ & Applied voltage & $U_T$ & $10^{-2}$ V\\[2.0ex]
            ~
            $C$ & Doping profile & $\tilde{C}$ & $10^{21}$ m$^{-3}$\\[2.0ex]
            ~
            $n_\electrons$ & Electron density & $\tilde{N}_\electrons$ & $10^{25}$ m$^{-3}$\\[2.0ex]
            ~
            $n_\holes$ & Hole density & $\tilde{C}$ & $10^{21}$ m$^{-3}$\\[2.0ex]
            ~
            $n_\ions$ & Defect density & $\tilde{N}_\ions$ & $10^{28}$ m$^{-3}$\\[2.0ex]
            ~
            $\mu_\electrons,\mu_\holes$ & Electron and hole mobility & $\tilde{\mu}$ & $10^{-4}$ m$^2$/(Vs)\\[2.0ex]
            ~
            $\mu_\ions$ & Defect mobility & $\tilde{\mu}_\ions$&$10^{-14}$ m$^2$/(Vs)\\[2.0ex]
            ~
            $t$ & Time variable & $\displaystyle\frac{l^2}{\tilde{\mu}_\ions U_T}$& $10^{4}$ s\\[2.0ex]
            ~
            $\mathbf{j}_\electrons$ & Electron current density & $\displaystyle\frac{q U_T \tilde{N}_\electrons \tilde{\mu}}{l}$&  $10^{6}$ C/(m$^2$s)\\[2.0ex]
            ~
            $\mathbf{j}_\holes$ & Hole current density & $\displaystyle\frac{q U_T \tilde{C} \tilde{\mu}}{l}$&  $10^{2}$ C/(m$^2$s)\\[2.0ex]
            ~
            $\mathbf{j}_\ions$ & Defect current density & $\displaystyle\frac{q U_T \tilde{N}_\ions\tilde{\mu}_\ions}{l}$&  $10^{-1}$ C/(m$^2$s)\\[2.0ex]
            ~
            $q v_\electrons, q v_\holes$ & Scaled recombination velocity & $\displaystyle\frac{q \tilde{\mu}U_T}{l}$&  $10^{-18}$ Cm/s\\[2.0ex]
            ~
            $n_{\electrons, 0}$ & Electron equilibrium density & $\tilde{N}_\electrons$ & $10^{25}$ m$^{-3}$\\[2.0ex]
            ~
            $n_{\holes, 0}$ & Hole equilibrium density & $\tilde{C}$ & $10^{21}$ m$^{-3}$\\[2.0ex]
            \hline
    \end{tabular}
    }
    \caption{
        Scaling factors of a MoS$_2$-based memristive device related to the parameters in \Cref{tab:TMDC-general}.
        Since the doping and hole densities are small compared to the electron and defect densities, we use the same scaling factor $\tilde{C}$.
        }
    \label{tab:scalingfactors-memristors}
\end{table}

%%%%%%%%%%%%%%%%%%%%%%%%%%%%%%%%%%%%%%%%%%%%
We assume that the electron and hole mobilities are equal, i.e., $\mu_\electrons = \mu_\holes$.
However, in contrast to \cite{Abdel2023Existence}, we allow for different magnitudes for electron and hole densities.
We set all intrinsic energies $E_\alpha = 0$ for $\alpha = { \electrons, \holes, \ions }$, although in practice, this is not the case.
This assumption is made to simplify notation, and the fundamental ideas of our subsequent analysis hold even if  $E_\alpha \neq 0$ and $\mu_\electrons \neq \mu_\holes$.
The timescale of the defect migration is chosen, and the spatial vector is scaled with respect to the channel length as the main flow direction.
In the following, the scaled quantities are denoted with the same symbol as the corresponding unscaled quantities.
The dimensionless version of the mass balance equations \eqref{eq:model-cont-eq-memristor} reads
\begin{subequations} \label{eq:model-dimless-memristor}
    \begin{align}
        \nu\,z_{\electrons} \partial_t n_{\electrons} + \nabla\cdot \mathbf{j}_{\electrons} &=  0, &&{\mathbf{x}}\in\mathbf{\Omega},\ t\geq 0, \label{eq:model-cont-eq-n-dimless-memristor}\\
        ~
        \nu\,z_{\holes} \partial_t n_{\holes} + \nabla\cdot \mathbf{j}_{\holes} &= 0, &&{\mathbf{x}}\in\mathbf{\Omega},\ t\geq 0, \label{eq:model-cont-eq-p-dimless-memristor}\\
        ~
        z_{\ions} \partial_t n_{\ions} + \nabla\cdot \mathbf{j}_{\ions} &= 0,&&{\mathbf{x}}\in\mathbf{\Omega},\ t\geq 0, \label{eq:model-cont-eq-a-dimless-memristor}
    \end{align}
    which are self-consistently coupled to the non-dimensionalized Poisson's equation %(see \eqref{eq:model-poisson-memristor})
    \begin{equation}\label{eq:model-poisson-dimless-memristor}
            - \lambda^2\Delta\psi
            =
            \delta_\electrons \Bigl( z_{\electrons}n_{\electrons} + \delta_\holes \left( z_{\holes}n_{\holes} + z_C C \right) \Bigr) + z_{\ions}n_{\ions},
            ~
            \quad  {\mathbf{x}}\in{\mathbf{\Omega}}, \ t\geq 0.
    \end{equation}
    The dimensionless charge carrier currents are given by
    \begin{align} \label{eq:cont-flux-dimless-memristor}
        \mathbf{j}_\alpha = - z_\alpha^2  n_\alpha \nabla \varphi_\alpha, \quad \alpha \in \{ \electrons, \holes, \ions \},
    \end{align}
    and the non-dimensionalized version of the state equation \eqref{eq:state-eq} corresponds to
    \begin{equation} \label{eq:state-eq-dimless-memristor}
            n_\alpha = \mathcal{F}_\alpha \Bigl(z_\alpha(\varphi_\alpha - \psi) \Bigr), \quad \alpha \in \{ \electrons, \holes, \ions\}.
    \end{equation}
\end{subequations}
%%%%%%%%%%%%%%%%%%%%%%%%%%%%%%%%
There are four dimensionless quantities:
We have the rescaled Debye length $\lambda \approx 10^{-5}$, the mobility parameter $\nu \approx 10^{-10}$, the electron concentration parameter $ \delta_\electrons \approx10^{-3}$, and the hole concentration parameter $\delta_\holes \approx10^{-4}$.
These quantities are defined as
\begin{equation*}
    \lambda = \sqrt{\frac{\varepsilon_sU_T}{l^2q\tilde{N}_\ions}},
    \quad \nu = \frac{\tilde{\mu}_{\ions}}{\tilde{\mu}},
    \quad \delta_\electrons = \frac{\tilde{N}_\electrons}{\tilde{N}_\ions},
    \quad \delta_\holes = \frac{\tilde{C}}{\tilde{N}_\electrons}.
\end{equation*}

%%%%%%%%%%%%%%%%%%%%%%%%%%%
\textbf{Non-dimensionalized boundary and initial conditions.}
We use the non-dimensionalized version of \eqref{eq:initial-cond-memristor}  as the initial condition.
For the boundary conditions, we supply the model for both contact boundary models with a dimensionless version of the no-flux boundary condition for the ionic defects \eqref{eq:BC-a-memristor}.
Additionally, on the isolating boundary $\mathbf{\Gamma}^N$, we consider dimensionless versions of the no-flux boundary conditions \eqref{eq:BC-GammaN-memristor}.
In total, the non-dimensionalized versions of these boundary conditions are as follows for $t\geq 0$
\begin{subequations} \label{eq:memristor-BC}
    \begin{alignat}{2}
       -\lambda^2 \nabla \psi \cdot {\boldsymbol{\nu}} = \mathbf{j}_\electrons \cdot {\boldsymbol{\nu}} &= \mathbf{j}_\holes \cdot {\boldsymbol{\nu}}
        = 0, \quad
        &&\mathbf{x} \in \mathbf{\Gamma}^N, \label{eq:memristor-Neumann-BC-dimless} \\
        ~
        \mathbf{j}_\ions \cdot {\boldsymbol{\nu}}
        &=0, \quad
        &&\mathbf{x} \in \partial \mathbf{\Omega}, \label{eq:memristor-BC-anion-dimless}
    \end{alignat}
\end{subequations}
where the current densities are definied in \eqref{eq:cont-flux-dimless-memristor}.
For the Schottky contact boundary model we supply the model with a non-dimensionalized version of \eqref{eq:BC-Schottky-const-volt} for $t\geq 0$
\begin{subequations} \label{eq:memristor-Schottky-BC-dimless}
    \begin{alignat}{2}
        \psi &= \tilde{\psi}^D, \quad &&\mathbf{x} \in \mathbf{\Gamma}^C \label{eq:memristor-Schottky-BC-dimless-psi}\\
        ~
        \mathbf{j}_\alpha \cdot \boldsymbol{\nu}  &= z_\alpha v_\alpha \left( n_\alpha - n_\alpha^D \right), \quad \text{for}\; \alpha \in \{\electrons, \holes \}, \quad &&\mathbf{x} \in \mathbf{\Gamma}^C, \label{eq:memristor-Schottky-BC-dimless-n-p}
    \end{alignat}
\end{subequations}
where $\tilde{\psi}^D := \psi^D + \overline{V}$ and $n_\alpha^D = \mathcal{F}_\alpha ( z_\alpha (\varphi^D -\psi^D))$.
Contrarily, in case of the ohmic contact boundary model, we assume a dimensionless version of \eqref{eq:BC-ohmic-const-volt}
\begin{align} \label{eq:memristor-Dirichlet-BC-dimless}
    \psi = \psi^D(\mathbf{x}), \;\; \varphi_\electrons  = \varphi_\holes
        &= \varphi^D(\mathbf{x}), \quad
        &&\mathbf{x} \in \mathbf{\Gamma}^C.
\end{align}

A mathematical study of the continuous charge transport model with ohmic boundary conditions \eqref{eq:memristor-Dirichlet-BC-dimless} has been conducted in  \cite{Jourdana.2023, Jungel.2023, Herda.2024}.
In particular, in \cite{Herda.2024} the authors proved the existence of weak solutions for the same choice of statistics functions for the charge carriers as in this work.
Weak formulations based on the Schottky contact boundary model \eqref{eq:memristor-Schottky-BC-dimless} are not discussed in mathematical literature yet.

%%%%%%%%%%%%%%%%%%%%%%%%%%%%%%%%%%%%%%%%%%%%%%%%%%%%%%%%%%%%%%%%%%%%%%%%%%%
%%%%%%%%%%%%%%%%%%%%%%%%%%%%%%%%%%%%%%%%%%%%%%%%%%%%%%%%%%%%%%%%%%%%%%%%%%%
\subsection{Entropy functions and continuous entropy-dissipation inequality}
\label{sec:cont_entropy_dissipation}
%%%%%%%%%%%%%%%%%%%%%%%%%%%%%%%%%%%%%%%%%%%%%%%%%%%%%%%%%%%%%%%%%%%%%%%%%%%
%%%%%%%%%%%%%%%%%%%%%%%%%%%%%%%%%%%%%%%%%%%%%%%%%%%%%%%%%%%%%%%%%%%%%%%%%%%
A well-established tool for the structural analysis of PDE models and their discrete counterparts is the entropy method \cite{Jungel.2016}.
Typically, the entropy function $\Phi_\alpha$ of a carrier $\alpha \in \{\electrons, \holes, \ions \}$ is defined as the anti-derivative of the inverse of the statistics function
\begin{equation}\label{eq:entropyfunc}
    \Phi_\alpha'(x) = \mathcal{F}_\alpha^{-1}(x), \quad x \geq 0.
\end{equation}
Since the statistics function is strictly increasing (due to \eqref{hyp:statistics-n-p} and \eqref{hyp:statistics-a}), $\Phi_\alpha$ is strictly convex.
By \eqref{eq:entropyfunc} the entropy function $\Phi_\alpha$ is only uniquely defined up to a constant.
Thus, for the ionic defects we assume that the integration constant to uniquely determine $\Phi_\ions$ via \eqref{eq:entropyfunc} is chosen in such a way that $\Phi_\ions$ is non-negative and vanishes at only one point.
Moreover, for electrons and holes $\alpha = \electrons, \holes$  we introduce a relative entropy function to deal with the boundary conditions.
We have
\begin{align} \label{eq:H-function}
    H_\alpha(x,y) = \Phi_\alpha(x) - \Phi_\alpha(y) - \Phi_\alpha'(y)(x-y), \quad x, y \geq 0,
\end{align}
which is non-negative due to the convexity of $\Phi_\alpha$.
For example, in case of a Boltzmann approximation, i.e. $\mathcal{F}(\eta) = \exp(\eta)$, we have
\begin{align*}
    \Phi(x) = x\log(x) -x +1, \quad x \geq 0,
\end{align*}
and, in case of the Fermi-Dirac integral of order $-1$, i.e., $\mathcal{F}(\eta) = (\exp(-\eta) + 1)^{-1}$, we  have
\begin{align*}
    \Phi(x) = x\log(x) +(1-x)\log(1-x)+\log(2), \quad x \geq 0.
\end{align*}
With this, we can formulate a \textit{total relative entropy} $\mathbb{E}$ with respect to the boundary values $\psi^D, \varphi^D \in W^{1, \infty} ( \mathbf{\Omega})$.
The total relative entropy is adapted from \cite{Abdel2023Existence}.
%%%%%%%%%%%%%%%%%%%%%%%%%%%%%%%%%%%%%%%%
For both types of boundary conditions, Schottky \eqref{eq:memristor-Schottky-BC-dimless} and ohmic \eqref{eq:memristor-Dirichlet-BC-dimless}, we define for $t \geq 0$
\begin{equation}\label{eq:continuous-entropy}
    \begin{split}
        \mathbb{E}(t) =
        &\;\frac{\lambda^2}{2}\int_{\mathbf{\Omega}} |\nabla (\psi - \hat{\psi}^D)|^2\,d\mathbf{x}
        ~
        + \int_{\mathbf{\Omega}} \Phi_{\ions}(n_{\ions})\,d\mathbf{x}\\
        ~
        &+ \delta_\electrons \int_{\mathbf{\Omega}} H_\electrons(n_\electrons, n_\electrons^D)  \,d\mathbf{x}
        ~
        +  \delta_\electrons \delta_\holes  \int_{\mathbf{\Omega}} H_\holes(n_\holes, n_\holes^D)  \,d\mathbf{x},
    \end{split}
\end{equation}
where we set $\hat{\psi}^D = \psi^D + \overline{V}$ for the Schottky and $\hat{\psi}^D = \psi^D$  for the ohmic boundary conditions.
The densities $n_\electrons^D$, $n_\holes^D$ can be calculated by inserting $\varphi^D, \psi^D$ into the dimensionless state equation \eqref{eq:state-eq-dimless-memristor}.
Since we assume $\Phi_{\ions} \geq 0$ the second term is non-negative as well which implies that the total entropy is non-negative.

For each boundary condition, we define the dissipation functionals separately.
The associated dissipation $\mathbb{D}_{\reviewerTwo{\text{SC}}}$ for the Schottky boundary conditions \reviewerTwo{(SC)} \eqref{eq:memristor-Schottky-BC-dimless} reads
\begin{equation} \label{eq:continuous-dissipation-Schottky}
    \begin{split}
        \mathbb{D}_{\reviewerTwo{\text{SC}}}(t)
        =&
        \frac{\delta_\electrons}{2\nu}  \int_\mathbf{\Omega} n_\electrons \vert \nabla \varphi_\electrons \vert^2  \,d\mathbf{x}
        ~
        + \frac{\delta_\electrons \delta_\holes}{2\nu}  \int_\mathbf{\Omega} n_\holes \vert \nabla \varphi_\holes \vert^2  \,d\mathbf{x}\\
        ~
        &+ \frac{z_{\ions}^2  }{2}\int_\mathbf{\Omega} n_{\ions} \vert \nabla \varphi_{\ions}\vert^2  \,d\mathbf{x}\\
        ~
        &+ \frac{\delta_\electrons}{\nu} \int_{\mathbf{\Gamma}^C}  v_\electrons  \bigl( \mathcal{F}^{-1}_\electrons (n_\electrons) - \mathcal{F}^{-1}_\electrons (n_\electrons^D) \bigr) \left( n_\electrons - n_\electrons^D \right) \,d\gamma\\
        ~
        &+ \frac{\delta_\electrons\delta_\holes}{\nu} \int_{\mathbf{\Gamma}^C} v_\holes  \bigl( \mathcal{F}^{-1}_\holes (n_\holes) - \mathcal{F}^{-1}_\holes (n_\holes^D) \bigr) \left( n_\holes - n_\holes^D \right) \,d\gamma.
    \end{split}
\end{equation}
The boundary integrals in the last two terms are non-negative due to the monotonicity of the inverse of statistics function $\mathcal{F}_\alpha^{-1}$.
Hence, the overall dissipation for the Schottky contact boundary model is non-negative as well.
%%%%%%%%%%%%%
For ohmic boundary conditions \reviewerTwo{(OC)} \eqref{eq:memristor-Dirichlet-BC-dimless}, the non-negative dissipation for $t \geq 0$ is defined as
\begin{equation} \label{eq:continuous-dissipation-Dirichlet}
    \begin{split}
        \mathbb{D}_{\reviewerTwo{\text{OC}}}(t)
        =&
        \frac{\delta_\electrons}{2\nu}  \int_\mathbf{\Omega} n_\electrons \vert \nabla \varphi_\electrons \vert^2  \,d\mathbf{x}
        ~
        + \frac{\delta_\electrons \delta_\holes}{2\nu}  \int_\mathbf{\Omega} n_\holes \vert \nabla \varphi_\holes \vert^2  \,d\mathbf{x}\\
        ~
        &+ \frac{z_{\ions}^2  }{2}\int_\mathbf{\Omega} n_{\ions} \vert \nabla \varphi_{\ions}\vert^2  \,d\mathbf{x}.
    \end{split}
\end{equation}
With the entropy and dissipation formulations, we can establish a continuous entropy-dissipation inequality, a crucial \textit{a priori} estimate to investigate the charge transport model.
In particular, a discrete variant of this theorem helps us to prove the existence of discrete solutions.

%%%%%%%%%%%%%%%%%%%%%%%%%%%%%%
\begin{theorem}(Continuous entropy-dissipation inequality) \label{thm:cont-E-D-memristor}
    Consider a smooth solution to the model \eqref{eq:model-dimless-memristor}, with initial conditions \eqref{eq:initial-cond-memristor} and boundary conditions \eqref{eq:memristor-BC} and \eqref{eq:memristor-Schottky-BC-dimless} (or \eqref{eq:memristor-Dirichlet-BC-dimless}).
    Then, for any $\varepsilon>0$, there is a constant $c_{\varepsilon, \mathbf{\Omega}}>0$ such that
    \begin{align} \label{eq:continuous-E-D-inequality-memristor}
        \frac{\text{d}}{\text{d}t}\mathbb{E}(t)  +  \mathbb{D}(t) \leq c_{\varepsilon, \mathbf{\Omega}} + \varepsilon \mathbb{E}(t), \quad t \geq 0,
    \end{align}
    where the entropy is defined in \eqref{eq:continuous-entropy} and the dissipation \reviewerTwo{$\mathbb{D} = \mathbb{D}_{\text{SC}}$} in \eqref{eq:continuous-dissipation-Schottky} or \reviewerTwo{$\mathbb{D} = \mathbb{D}_{\text{OC}}$ in} \eqref{eq:continuous-dissipation-Dirichlet}.
    The constant $c_{\varepsilon, \mathbf{\Omega}}$ depends only on $\varepsilon$, the measure of $\mathbf{\Omega}$, the boundary data via the norms $\|\varphi^D\|_{W^{1,\infty}}$ and $\|\psi^D\|_{W^{1,\infty}}$, $z_{\ions}^2$, and the dimensionless parameters $\delta_\electrons$, $\delta_\holes$ and $\nu$.
    Under the assumption of constant boundary data,  i.e., $\nabla\varphi^D = \nabla\psi^D = \mathbf{0} $, the right-hand side of the inequality vanishes. \QEDB
\end{theorem}
%%%%%%%%%%%%%%%%%%%%%%%%%%%%%%
The proof follows similarly to the proof of \cite[Theorem 3.3]{Abdel2023Existence} and is therefore omitted here.

%%%%%%%%%%%%%%%%%%%%%%%%%%%%%%%%%%%%%%%%%%%%%%%%%%%%%%%%%%%%%%%%%%%%%%%%%%%
%%%%%%%%%%%%%%%%%%%%%%%%%%%%%%%%%%%%%%%%%%%%%%%%%%%%%%%%%%%%%%%%%%%%%%%%%%%
\section{Discretization and analysis of charge transport equations}
\label{sec:discrete_model}
%%%%%%%%%%%%%%%%%%%%%%%%%%%%%%%%%%%%%%%%%%%%%%%%%%%%%%%%%%%%%%%%%%%%%%%%%%%
%%%%%%%%%%%%%%%%%%%%%%%%%%%%%%%%%%%%%%%%%%%%%%%%%%%%%%%%%%%%%%%%%%%%%%%%%%%

In this section, we begin by formulating the implicit-in-time finite volume scheme for the model \eqref{eq:model-dimless-memristor}, incorporating either the Schottky \eqref{eq:memristor-Schottky-BC-dimless} or the ohmic boundary conditions \eqref{eq:memristor-Dirichlet-BC-dimless}.
We pay particular attention to the discretization of the boundary conditions by eliminating the boundary unknowns.
Then, in \Cref{sec:existence}, we establish that an entropy-dissipation relation can also be derived in the discrete framework, enabling us to prove the existence of discrete solutions.

%%%%%%%%%%%%%%%%%%%%%%%%%%%%%%%%%%%%%%%%%%%%%%%%%%%%%%%%%%%%%%%%%%%%%%%%%%%
%%%%%%%%%%%%%%%%%%%%%%%%%%%%%%%%%%%%%%%%%%%%%%%%%%%%%%%%%%%%%%%%%%%%%%%%%%%
\subsection{Discrete version of charge transport model}
\label{sec:FVM}
%%%%%%%%%%%%%%%%%%%%%%%%%%%%%%%%%%%%%%%%%%%%%%%%%%%%%%%%%%%%%%%%%%%%%%%%%%%
%%%%%%%%%%%%%%%%%%%%%%%%%%%%%%%%%%%%%%%%%%%%%%%%%%%%%%%%%%%%%%%%%%%%%%%%%%%

%%%%%%%%%%%%%%%%%%%%%%%%%%%%%%%%%%%%%%%%%%%%%%%%%%%%%%%%%%%%%%%%%%%%%%%%%%
%%%%%%%%%%%%%%%%%%%%%%%%%%%%%%%%%%%%%%%%%%%%%%%%%%%%%%%%%%%%%%%%%%%%%%%%%%
\subsubsection{Definition of discretization mesh} \label{sec:fvm-meshes}
%%%%%%%%%%%%%%%%%%%%%%%%%%%%%%%%%%%%%%%%%%%%%%%%%%%%%%%%%%%%%%%%%%%%%%%%%%

\reviewerTwo{
    % Let $\mathbf{\Omega}$ be polygonal (or polyhedral).
    % An admissible mesh, following the definition in \cite{Eymard.2000}, can be characterized by the triplet $\left( \mathcal{T}, \mathcal{E}, \lbrace\mathbf{x}_K\rbrace_{K \in \mathcal{T}}\right)$, where
    Let $\mathbf{\Omega} \subset \mathbb{R}^d$ be a polygonal (or polyhedral) domain. We consider an \emph{admissible finite volume mesh} in the sense of~\cite{Eymard.2000} and described in more detail in \cite{Abdel2024Thesis, Abdel2023Existence}, defined by the triplet $(\mathcal{T}, \mathcal{E}, \{\mathbf{x}_K\}_{K \in \mathcal{T}})$, where:
    \begin{enumerate}
        \item $\mathcal{T}$ is a family of non-empty, convex and open control volumes $K$, whose Lebesgue measure is denoted by $m_K$ and satisfy $ \mathbf{\overbar{\Omega}} = \bigcup_{K \in \mathcal{T}} \overbar{K}$.
        % $\mathcal{T}$ is a family of non-empty, convex, open, and polygonal (or polyhedral) \textit{control volumes} (frequently called \textit{cells}) $K \in \mathcal{T}$, whose Lebesgue measure is denoted by $m_K$.
        % The union of the closure of all control volumes is equal to the closure of the spatial domain, i.e., $\mathbf{\overbar{\Omega}} = \bigcup_{K \in \mathcal{T}} \overbar{K}$.
        %%%%%%%%%%%%%%%%%%%%%%%%%%%%%%%%%%%%%%%%
        \item $\mathcal{E}$ is a family of faces $\sigma$ with strictly positive $(d-1)$-dimensional measure $m_\sigma $.
        We use the abbreviation $\sigma = K | L = \partial K \cap \partial L$ for the intersection between two different control volumes $K \neq L$.
        Each cell $K \in \mathcal{T}$ has an associated subset of faces $\mathcal{E}_K \subset \mathcal{E}$.
        We define the subsets $\mathcal{E}^C$ and $\mathcal{E}^N$ for boundary faces lying on $\boldsymbol{\Gamma}^C$ and $\boldsymbol{\Gamma}^N$, respectively, and assume that $\mathbf{\Gamma}^C, \mathbf{\Gamma}^N$ can be well described by the union of such faces.
        %Furthermore, we call $\mathcal{E}$ a \textit{family of faces}, where $\sigma \in \mathcal{E}$ is a subset of $\overbar{\mathbf{\Omega}}$ contained in a hyperplane of $\mathbb{R}^d$.
        % Each \textit{face} (or \textit{edge}) $\sigma$ has a strictly positive $(d-1)$-dimensional measure, denoted by $m_\sigma $.
        % We use the abbreviation $\sigma = K | L = \partial K \cap \partial L$ for the intersection between two different control volumes $K \neq L$.
        % The intersection $ K | L$ is either empty or reduces to a face (or edge) contained in $\mathcal{E}$.
        % Also, for any cell $K \in \mathcal{T}$ we assume that there exists a subset of faces $ \mathcal{E}_K \subset \mathcal{E}$ so that we can describe the boundary of a control volume by $\partial K = \bigcup_{\sigma \in \mathcal{E}_K} \sigma$.
        % Consequently, we have $\mathcal{E} = \bigcup_{K \in \mathcal{T}} \mathcal{E}_K$.
        % We also distinguish the faces that are on the boundary of $\mathbf{\Omega}$ by introducing the notations
        % \begin{align*}
        %     \mathcal{E}^C = \{\sigma\in\mathcal{E}\ \text{s.t.}\ \sigma\subset \mathbf{\Gamma}^C\},
        %     \quad \mathcal{E}^N = \{\sigma\in\mathcal{E}\ \text{s.t.}\ \sigma\subset \mathbf{\Gamma}^N\},
        % \end{align*}
        % and assume that the boundaries $\mathbf{\Gamma}^C, \mathbf{\Gamma}^N$ can be well described by the union of respective boundary faces, i.e.,
        % $\boldsymbol{\Gamma}^C = \bigcup_{\sigma \in \mathcal{E}^C} \sigma$ and $\boldsymbol{\Gamma}^N = \bigcup_{\sigma \in \mathcal{E}^N} \sigma$.
        %%%%%%%%%%%%%%%%%%%%%%%%%%%%%%%%%%%%%%%%
        \item \reviewerOne{Each $K \in \mathcal{T}$ has an associated node $\mathbf{x}_K \in \overline{K}$.
        If $K$ and  $L$ share a face $\sigma =K | L$, then $\mathbf{x}_K \neq \mathbf{x}_L $ and the segment $\overline{\mathbf{x}_K \mathbf{x}_L}$ is orthogonal to $\sigma$.
        In this case, we define $d_\sigma$ as the Euclidean distance between the two nodes $\mathbf{x}_K$ and $\mathbf{x}_L$.}
        If $\sigma = \partial K \cap \partial \mathbf{\Omega} \neq \emptyset$, then $d_\sigma$ is defined as the Euclidean distance between $\mathbf{x}_K$ and the affine hyperplane spanned by $\sigma$.
        % To each control volume $K \in \mathcal{T}$ we assign a \textit{node} (or \textit{cell center})
        % $\mathbf{x}_K \in K$, and we assume that the family of nodes $\lbrace\mathbf{x}_K\rbrace_{K\in{\mathcal T}}$ satisfies the \textit{orthogonality condition}:
        % If $K$ and  $L$ share a face $\sigma =K | L$, then the vector
        % \begin{align*}
        %     \overline{\mathbf{x}_K \mathbf{x}_L}\text{ is orthogonal to }\sigma = K | L.
        % \end{align*}
        % We assume that each node $\mathbf{x}_K$ is located within the interior of $\mathbf{\Omega}$.
        % In other words, for all $K \in \mathcal{T}$ with $\mathcal{E}_K \cap \partial \mathbf{\Omega} \neq \emptyset$, it holds that $\mathbf{x}_K \notin \partial \mathbf{\Omega}$.
        %%%%%%%%%%%%%
        % Furthermore, for each edge $\sigma = K|L \in{\mathcal E}$, we define $d_\sigma$ as the Euclidean distance between two nodes $\mathbf{x}_K$ and $\mathbf{x}_L$.
        % If the edge $\sigma \in{\mathcal E}$ lies on the boundary $\partial \mathbf{\Omega}$, i.e.,  $\sigma = \partial K \cap \partial \mathbf{\Omega} \neq \emptyset$, then $d_\sigma$ is defined as the Euclidean distance between $\mathbf{x}_K$ and the affine hyperplane spanned by $\sigma$.
\end{enumerate}
}

\reviewerTwo{
    Lastly, we introduce the transmissibility through the edge $\sigma$ by $\tau_\sigma = m_\sigma / d_\sigma$.
    The notations are illustrated in \Cref{fig:control-vol}.
    % The assumption $\mathbf{x}_K \notin \partial \mathbf{\Omega}$ in the third point is not a restriction imposed by the definition of admissible meshes but a specific requirement for the numerical analysis in this work.
    % Admissible meshes can indeed have nodes $\mathbf{x}_K$ at the boundary $\partial \mathbf{\Omega}$.
    % In \Cref{sec:numerics}, we use meshes for the simulations where, if $K \in \mathcal{T}$ with $\mathcal{E}_K \cap \partial \mathbf{\Omega} \neq \emptyset$, then $\mathbf{x}_K \in \partial \mathbf{\Omega}$.
    % Such meshes are constructed using \textit{boundary conforming Delaunay triangulations} \cite{si2010boundary, Gartner2019}.
    In addition to the admissibility, we assume that the mesh $\left( \mathcal{T}, \mathcal{E}, \lbrace\mathbf{x}_K\rbrace_{K \in \mathcal{T}}\right)$ is regular in the sense of \cite[p.\ 16]{Abdel2023Existence}.
    This regularity is an asymptotic property required to establish convergence results.
}

\reviewerOne{
    In the following numerical analysis, we assume $d_\sigma>0$ for all edges $\sigma$.
    In particular, we assume that all nodes satisfy $\mathbf{x}_K \notin \partial \mathbf{\Omega}$.
    In Remark~\ref{rem.centeronboundary}, we explain how the scheme extends naturally to the case $\mathbf{x}_K \in \partial \mathbf{\Omega}$ for $K \in \mathcal{T}$ with $\mathcal{E}_K \cap \partial \mathbf{\Omega} \neq \emptyset$.
    This type of mesh will be used in the simulations in \Cref{sec:numerics}.
}

For the time discretization we decompose the interval $[0, t_F]$, for a given end time $t_F>0$ into a finite and increasing number of time steps
$
0 = t^1 < \ldots < t^M = t_F
$
with a step-size $\tau^m = t^{m} - t^{m-1}$ at time step $m = 2,\dots, M$.
%%%%%%%%%%%%%%%%%%%%%%%%%%%%%%%%%%%%%%%%%%
\begin{figure}[ht]
    \centering
    \includegraphics[width=0.92\textwidth]{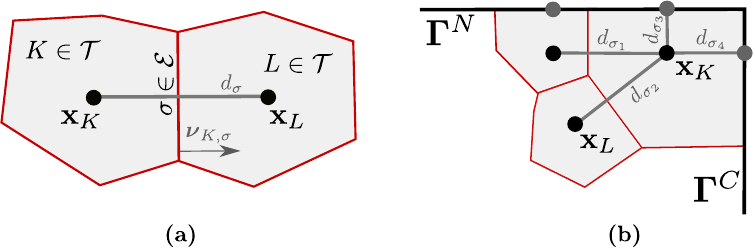}
    %%%%%%%%%%%%%%%%%%%%%%%%%%%%%%%%%%%%%%%%%%
    \caption
    {
        Neighboring control volumes in \textbf{(a)} the interior of the device domain and \textbf{(b)} near outer boundaries $\mathbf{\Gamma}^C$ and $\mathbf{\Gamma}^N$ (right).
        For our numerical analysis, we assume that the cell centers (black points) of a boundary control volume are located in the interior of the computational domain.
        The boundary of the control volumes are highlighted in red.
        }
    \label{fig:control-vol}
\end{figure}

%%%%%%%%%%%%%%%%%%%%%%%%%%%%%%%%%%%%%%%%%%

%%%%%%%%%%%%%%%%%%%%%%%%%
%%%%%%%%%%%%%%%%%%%%%%%%%%%%%%%%%%%%%%%%%%%%%%%%%%%%%%%%%%%%%%%%%%%%%%%%%%%
%%%%%%%%%%%%%%%%%%%%%%%%%%%%%%%%%%%%%%%%%%%%%%%%%%%%%%%%%%%%%%%%%%%%%%%%%%%
\subsubsection{Finite volume discretization} \label{sec:FVM-discret}
%%%%%%%%%%%%%%%%%%%%%%%%%%%%%%%%%%%%%%%%%%%%%%%%%%%%%%%%%%%%%%%%%%%%%%%%%%%
%%%%%%%%%%%%%%%%%%%%%%%%%%%%%%%%%%%%%%%%%%%%%%%%%%%%%%%%%%%%%%%%%%%%%%%%%%%

Next, we formulate the implicit-in-time finite volume discretizations for the charge transport model \eqref{eq:model-dimless-memristor}, considering the initial condition \eqref{eq:initial-cond-memristor} and the boundary conditions \eqref{eq:memristor-BC}.
Specifically, for the  Schottky boundary model, we incorporate the conditions \eqref{eq:memristor-Schottky-BC-dimless}, while for the ohmic boundary conditions, we apply \eqref{eq:memristor-Dirichlet-BC-dimless}.
In the following, the quantity $u_K^{m}$ represents an approximation of the mean value of $u(\mathbf{x}, t)$ on the cell $K$ at time $t^{m}$ and the quantity $u_\sigma^m$ stands for an approximation of the mean value of $u(\mathbf{x}, t)$ on the boundary face $\sigma \in \mathcal{E}^C$ at time $t^m$.
Here, $u$ is one of the potentials $\varphi_{\electrons}, \varphi_{\holes}$, $\varphi_\ions$, $\psi$.
With this, we can define the vectors of unknowns $\boldsymbol{u}^m= (u_K^{m})_{K\in{\mathcal T}}$
and
$ (u_\sigma^{m})_{\sigma\in{\mathcal E}^C}$.
The second vector of unknowns will be eventually eliminated.
%%%%%%%%%%%%%%%%%%%%%%%%%
The discretizations of the doping profile $C \in L^{\infty} ( \mathbf{\Omega})$, the boundary data $\varphi^D$, $\psi^D \in W^{1, \infty} (\mathbf{\Omega})$, and the initial conditions  $\varphi_\electrons^0, \varphi_\holes^0, \varphi_\ions^0 \in L^{\infty} (\mathbf{\Omega})$ are expressed as integral averages over a cell $K$
\begin{subequations} \label{eq:def-integral-average}
    \begin{align}
        \chi_K = \frac{1}{m_K}\int_K \chi(\mathbf{x}) d \mathbf{x}, \quad K\in\mathcal{T},\ \chi = C,\, \psi^D, \, \varphi^D, \, \varphi_{\electrons}^0, \, \varphi_{\holes}^0 \, \text{  or  } \, \varphi_{\ions}^0.
    \end{align}
    For the boundary data, we define the vectors $\boldsymbol{\psi}^D= (\psi_K^{D})_{K\in{\mathcal T}}$ and $\boldsymbol{\varphi}^D= (\varphi_K^{D})_{K\in{\mathcal T}}$.
    Analogously, we define the vectors ${{\boldsymbol \varphi}}_{\electrons}^0$, ${{\boldsymbol \varphi}}_{\holes}^0$ and ${{\boldsymbol \varphi}}_{\ions}^0$ for the initial conditions.
    Furthermore, for a face $\sigma \in \mathcal{E}^C$ located at the contact boundary $\mathbf{\Gamma}^C$ the discrete values for the boundary functions $\psi^D, \varphi^D $ on a face $\sigma$ are given as $(d-1)$-dimensional integral averages
    \begin{align} \label{eq:boundary-average-Dirichlet-val}
        \chi_\sigma = \frac{1}{m_\sigma}\int_\sigma \chi(\gamma) d \gamma, \quad \sigma\in\mathcal{E}^C,\  \chi = \psi^D,\, \varphi^D.
    \end{align}
\end{subequations}
For the Schottky boundary model \eqref{eq:memristor-Schottky-BC-dimless}, we define the discrete vectors $ \boldsymbol{\tilde{\psi}}^D = \boldsymbol{\psi}^D + \mathbf{\overline{V}} $, where $\mathbf{\overline{V}} = (\overline{V})_{K\in{\mathcal T}}$.
The vector $ \boldsymbol{\tilde{\psi}}^D$ corresponds to the discrete boundary values of the electric potential \eqref{eq:memristor-Schottky-BC-dimless-psi}, while $\boldsymbol{\psi}^D$ enters the discrete densities in the electron and hole boundary conditions \eqref{eq:memristor-Schottky-BC-dimless-n-p}.
Moreover, we set $ \tilde{\psi}^D_\sigma = \psi^D_\sigma + \overline{V}$.
The discrete densities can be calculated via the dimensionless state equation \eqref{eq:state-eq-dimless-memristor} inside the domain and at the contact boundary, namely,
\begin{subequations}\label{eq:discrete-state-eq}
    \begin{align}
        {\boldsymbol n}_{\alpha}^m &= \mathcal{F}_\alpha \left(z_\alpha ({\boldsymbol \varphi}_{\alpha}^m - {\boldsymbol \psi}^m)\right), && \alpha \in \{ \electrons, \holes, \ions \},\ m\in\mathbb{N},\label{eq:discrete-state-eq-m}\\
        ~
        n_{\alpha, \sigma}^m & = \mathcal{F}_\alpha \left(z_\alpha (\varphi_{\alpha, \sigma}^m - \psi_\sigma^m)\right), && \alpha \in \{ \electrons, \holes, \ions  \}, \ m\in\mathbb{N}, \ \sigma \in \mathcal{E}^C, \label{eq:discrete-state-eq-boundary}\\
        ~
        {\boldsymbol n}_{\alpha}^D &= \mathcal{F}_\alpha \left(z_\alpha ({\boldsymbol \varphi}^D - {\boldsymbol \psi}^D)\right), && \alpha \in \{ \electrons, \holes \},\label{eq:discrete-state-eq-Dirichlet-inside}\\
        ~
        n_{\alpha, \sigma}^D &= \mathcal{F}_\alpha \left(z_\alpha (\varphi_{\sigma}^D - \psi^D_\sigma)\right), && \alpha \in \{ \electrons, \holes  \},\ \sigma \in \mathcal{E}^C,\label{eq:discrete-state-eq-Dirichlet}
    \end{align}
\end{subequations}
where the statistics function is applied pointwise in \eqref{eq:discrete-state-eq-m} and \eqref{eq:discrete-state-eq-Dirichlet-inside}.
The first two definitions \eqref{eq:discrete-state-eq-m}, \eqref{eq:discrete-state-eq-boundary} depend on the vectors of unknowns, while the latter two \eqref{eq:discrete-state-eq-Dirichlet-inside}, \eqref{eq:discrete-state-eq-Dirichlet} depend on the given boundary functions.
%%%%%%%%%%%%%%%%%%%%%%%%%%%%%%%%%%%%%%%%%%%
Furthermore, we introduce the finite difference operator acting on vectors ${\boldsymbol u} = (u_K)_K$, denoted by $D_{K,\sigma}$.
It is given by
\begin{equation}\label{eq:difference-operator}
    D_{K,\sigma}{\boldsymbol u} = \left\{
    \begin{array}{ll}
        u_L - u_K, &\text{if }\sigma = K|L,\\
        u_\sigma - u_K, &\text{if }\sigma\in\mathcal{E}^C, \boldsymbol{u} \neq \boldsymbol{\varphi}_\ions, \boldsymbol{n}_\ions \\
        0, &\text{otherwise.}
    \end{array}
    \right.
\end{equation}
%%%%%%%%%%%%%%%%%%%%%%%%%%%%
Due to small parameters, such as the rescaled Debye length $\lambda$, the stiffness in drift-diffusion models favors fully implicit time discretization methods.
Among these, the implicit Euler method offers robust discretization and preserves asymptotic properties \cite{Bessemoulin.2014, ChainaisHillairet.2019}.
Consequently, it is the method of choice in this work.
For the finite volume discretization of the model, we proceed as follows:
We integrate the bulk equations \eqref{eq:model-dimless-memristor} over each control volume  $K \in \mathcal{T}$, apply the divergence theorem to the current densities, and break down the surface integrals into a sum of integrals over the faces $\sigma \in \mathcal{E}_K$.
This yields for $m \in \mathbb{N}$
\begin{subequations} \label{eq:discrete-model-memristor}
    \begin{align}
        \nu z_{\electrons} m_K \frac{n_{\electrons,K}^{m} - n_{\electrons,K}^{m-1} }{\tau^m} + \sum_{\sigma \in \mathcal{E}_K} J_{\electrons,K, \sigma}^m &= 0, && K\in\mathcal{T}, \label{eq:discrete-mass-balance-elec-memristor}\\
        ~
        \nu z_{\holes} m_K \frac{n_{\holes,K}^{m} - n_{\holes,K}^{m-1} }{\tau^m} + \sum_{\sigma \in \mathcal{E}_K} J_{\holes,K, \sigma}^m &= 0,&& K\in\mathcal{T}, \label{eq:discrete-mass-balance-holes-memristor}\\
        ~
        z_{\ions} m_K \frac{n_{\ions,K}^{m} - n_{\ions,K}^{m-1} }{\tau^m} + \sum_{\sigma \in \mathcal{E}_K} J_{\ions,K, \sigma}^m &=  0,&& K\in\mathcal{T}, \label{eq:discrete-mass-balance-anions-memristor}
    \end{align}
    which are coupled to the discrete Poisson equation which reads for $m \in \mathbb{N}$ and $K\in\mathcal{T}$
    \begin{equation}\label{eq:model-poisson-discrete-memristor}
        \begin{split}
            - \lambda^2\sum_{\sigma \in \mathcal{E}_K}\tau_\sigma D_{K,\sigma}\boldsymbol{\psi}^m
            ~
            =&
            ~
            \delta_\electrons m_K \Bigl( z_{\electrons}n_{\electrons,K}^{m} +   \delta_\holes ( z_{\holes}n_{\holes,K}^{m}  + z_C C_K ) \Bigr)\\
            %%%
            &+ m_K z_{\ions}n_{\ions, K}^{m}.
        \end{split}
    \end{equation}
\end{subequations}
We assume for $\alpha \in \{ \electrons, \holes, \ions \}$ that the discrete current densities $J_{\alpha, K, \sigma}^m$ are locally conservative and consistent approximations of $ \int_\sigma \, \mathbf{j}_\alpha \cdot \boldsymbol{\nu}_{K, \sigma} \,d S$,
where $\boldsymbol{\nu}_{K, \sigma}$ is the outward-pointing unit normal to the control volume $K$ on the face $\sigma$.
We mean with locally conservative that for $\sigma =K|L$ the flux approximation shall satisfy
\begin{equation}\label{eq:conservativity}
    J_{\alpha,K, \sigma}^m + J_{\alpha,L, \sigma}^m
    = 0,
    \quad \alpha \in \{ \electrons, \holes, \ions \}.
\end{equation}
The finite difference operator $D_{K,\sigma}$, defined in \eqref{eq:difference-operator}, is also locally conservative.
%%%%%%%%%%%%%%%%%%%%%%%%%%%%%%
Choosing the discrete current densities $J_{\alpha, K, \sigma}^m$ is delicate, as an incorrect choice can cause instability or violate thermodynamic principles.
In the following, we use the excess chemical potential approximation \cite{Yu.1988} as two-point flux approximation (TPFA) scheme for $J_{\alpha,K, \sigma}^m$ which was numerically analyzed in \cite{Cances.2021, Gaudeul.2021, Abdel2023Existence} and compared in \cite{Kantner.2020, Abdel.2021b}.
%%%%%%%%%%%%%%%%%%%%%%%%%%%%%%%%%%%%%%%%%
This flux discretization scheme reads for $\alpha \in \{ \electrons, \holes, \ions \}$
\begin{subequations} \label{eq:discrete-Sedan}
    \begin{equation}\label{eq:discrete-flux-elec-holes}
        J_{\alpha,K, \sigma}^m = \left\{
        \begin{array}{ll}
            \!\!\! - z_\alpha \tau_\sigma \Bigl( B\left(- Q_{\alpha, K, \sigma}^m \right) n_{\alpha, L}^m - B\left( Q_{\alpha, K, \sigma}^m \right) n_{\alpha, K}^m\Bigr), &\text{if }\sigma = K|L,\\[1.4ex]
            ~
            \!\!\!\!\! - z_\alpha \tau_\sigma \Bigl( B\left(- Q_{\alpha, K, \sigma}^m \right) n_{\alpha, \sigma}^m - B\left( Q_{\alpha, K, \sigma}^m \right) n_{\alpha, K}^m\Bigr), &\text{if }\sigma\in\mathcal{E}^C,  \alpha \neq \ions, \\[1.4ex]
            ~
            \!\!\!  0, &\text{otherwise},
        \end{array}
        \right.
    \end{equation}
    where the quantity $Q_{\alpha, K, \sigma}^m$ is defined as
    \begin{equation}\label{eq:inside-Bernoulli}
        Q_{\alpha, K, \sigma}^m = D_{K,\sigma} \left( z_\alpha{\boldsymbol \varphi}_{\alpha}^m - \log {\boldsymbol n}_{\alpha}^m \right).
    \end{equation}
\end{subequations}
In the previous formula, the logarithm is applied componentwise.
Moreover, the function $B$ denotes the Bernoulli function
\begin{equation} \label{eq:bernoulli}
    B(x) = \frac{x}{\exp(x) -1},\; \text{for}\; x \in \mathbb{R} \setminus \{0\} \quad \text{and} \quad B(0)=1.
\end{equation}
We now address the discretization of the boundary conditions.
Notably, the no-flux boundary conditions for the rest of the domain boundary have already been implicitly incorporated through the last conditions in \eqref{eq:difference-operator} and \eqref{eq:discrete-flux-elec-holes}.
For the ionic defects, the no-flux boundary condition \eqref{eq:memristor-BC-anion-dimless} and the thermodynamic consistency of the flux discretization scheme \eqref{eq:discrete-Sedan}, allows the ionic defect boundary unknown to be expressed as
\begin{align} \label{eq:boundary-unknown-ion}
\varphi_{\ions, \sigma}^{m} := \varphi_{\ions, K}^{m}, \quad \text{for all } \sigma \in \mathcal{E}^C {\text{ s.t. }\sigma\subset \overline{K}}.
\end{align}
Note that the previous formula is not necessary to define the scheme as the no-flux boundary conditions are implicitly contained in the third condition in \eqref{eq:discrete-flux-elec-holes}.
In case of ohmic boundary conditions \eqref{eq:memristor-Dirichlet-BC-dimless}, the boundary unknowns are set to the given Dirichlet functions
\begin{equation} \label{eq:boundary-ohmic}
       \varphi_{\electrons, \sigma}^{m} = \varphi_{\holes, \sigma}^{m}  := \varphi_{\sigma}^{D}, \quad \psi_\sigma^{m} := \psi_\sigma^D, \quad \text{for all } \sigma \in \mathcal{E}^C,
\end{equation}
where the discrete boundary values $\varphi_{\sigma}^{D}$, $\psi_\sigma^D$ are defined in \eqref{eq:boundary-average-Dirichlet-val}.
For Schottky contacts \eqref{eq:memristor-Schottky-BC-dimless}, we analogously define the discrete electric potential unknown at the boundary as
\begin{equation} \label{eq:Schottky-psi-sigma}
\psi_\sigma^{m} := \tilde{\psi}_\sigma^D = \psi_\sigma^D + \overline{V} , \quad \text{for all } \sigma \in \mathcal{E}^C.
\end{equation}
Furthermore, for Schottky contacts, the discrete approximation of $ \int_\sigma \, \mathbf{j}_\alpha \cdot \boldsymbol{\nu}_{K, \sigma} \,d S$ for $\sigma \in \mathcal{E}^C$ and $\alpha \in \{ \electrons, \holes \}$ is given by
\begin{equation} \label{eq:discrete-flux-boundary}
J_{\alpha,K, \sigma}^m =  z_\alpha v_\alpha m_\sigma (n_{\alpha, \sigma}^m - n_{\alpha, \sigma}^D), \quad \text{for all } \sigma \in \mathcal{E}^C,
\end{equation}
where we incorporate the condition \eqref{eq:memristor-Schottky-BC-dimless-n-p}.
\reviewerTwo{
    Higher-order methods in space have been investigated in related contexts, for example, in \cite{Moatti2024, Moatti2023structure}, where hybrid finite volume methods were applied to simpler drift-diffusion type models.
}
%%%%%%%%%%%%%%%%%%%%%%%%%%%%%%%%%%%%%%%%%%%%%%%%%%%%%%%%%%%%%%%%%%%%%%%%%%%%%%%%%%%%%%%%%%%%%
\subsubsection{Elimination of boundary values} \label{sec:discrete-boundary-values}
%%%%%%%%%%%%%%%%%%%%%%%%%%%%%%%%%%%%%%%%%%%%%%%%%%%%%%%%%%%%%%%%%%%%%%%%%%%%%%%%%%%%%%%%%%%%%
Rather than solving for the unknowns at the cell centers $\boldsymbol{\varphi}_\electrons^m$, $\boldsymbol{\varphi}_\holes^m$, $\boldsymbol{\varphi}_\ions^m$, $\boldsymbol{\psi}^m$ and on the boundary faces $ (\varphi_{\electrons, \sigma}^{m})_{\sigma\in{\mathcal E}^C}$, $(\varphi_{\holes, \sigma}^{m})_{\sigma\in{\mathcal E}^C}$, $(\varphi_{\ions, \sigma}^{m})_{\sigma\in{\mathcal E}^C}$, $(\psi_{\sigma}^{m})_{\sigma\in{\mathcal E}^C}$,  the boundary unknowns can be expressed in terms of neighboring cell unknowns.
The expression \eqref{eq:boundary-unknown-ion}, \eqref{eq:boundary-ohmic} and \eqref{eq:Schottky-psi-sigma} allow to eliminate the unknowns trivially.
In the case of Schottky contacts, the boundary densities $(n_{\electrons, \sigma}^{m})_{\sigma\in{\mathcal E}^C}$ and $(n_{\holes, \sigma}^{m})_{\sigma\in{\mathcal E}^C}$ are characterized implicitly not only by the discrete flux expression \eqref{eq:discrete-flux-elec-holes} but also by the discrete Schottky boundary condition \eqref{eq:discrete-flux-boundary}.
Setting both expressions equal, yields for all $\sigma \in \mathcal{E}^C$ and $\alpha\in\{\electrons, \holes\}$
\begin{equation}\label{eq:implicit_eq_boundary}
    - z_\alpha \tau_\sigma \Bigl( B\left(- Q_{\alpha, K, \sigma}^m \right) n_{\alpha, \sigma}^m - B\left( Q_{\alpha, K, \sigma}^m \right) n_{\alpha, K}^m\Bigr)
    =
    z_\alpha v_\alpha m_\sigma (n_{\alpha, \sigma}^m - n_{\alpha, \sigma}^D).
\end{equation}
In the next lemma we show that this equation has a unique solution which defines the remaining boundary unknowns.
\begin{lemma} \label{lem:Z-zero}
    Let $\alpha \in \{ \electrons, \holes \}$.
    Moreover, let $K \in \mathcal{T}$ and $\sigma \in \mathcal{E}^C$, where $\mathcal{E}^C \cap \overline{K} \neq \emptyset$.
    We define
    \begin{equation*}
        Q_{\alpha, K, \sigma}^m (s) = z_\alpha D_{K, \sigma} \boldsymbol{\psi}^m + \left( \mathcal{F}_\alpha^{-1}(s) - \log(s) \right) - \left(\mathcal{F}_\alpha^{-1}(n_{\alpha, K}^m) - \log(n_{\alpha, K}^m) \right).
    \end{equation*}
    The function
    \begin{equation*}
        \begin{split}
        Z_{\alpha, K}^m (s) =
        &- \tau_\sigma \Bigl(  B \bigl(-Q_{\alpha, K, \sigma}^m(s)\bigr) s - B\bigl(Q_{\alpha, K, \sigma}^m (s) \bigr) n_{\alpha, K}^m \Bigr)\\
        &- v_\alpha m_\sigma s + v_\alpha m_\sigma n_{\alpha, \sigma}^D
        \end{split}
    \end{equation*}
    has a unique and positive zero that we denote by $n_{\alpha, \sigma}^m$.
\end{lemma}
\begin{proof}
    The claim is established by demonstrating the following:
    i) As $s\rightarrow + \infty $, we have $Z_{\alpha, K}^m(s) \rightarrow - \infty$
    ii) $Z_{\alpha, K}^m$ approaches a non-negative limit as $s\rightarrow 0$
    iii) $Z_{\alpha, K}^m$ is strictly decreasing.
    First, we prove i).
    Due to a non-dimensionalized version of the second inequality in \eqref{eq:statistics-ineq}, we can estimate
    $
        Q_{\alpha, K, \sigma}^m
        \geq
        z_\alpha D_{K, \sigma} \boldsymbol{\psi}^m + \mathcal{F}_\alpha^{-1}(n_{\alpha, K}^m) - \log(n_{\alpha, K}^m).
    $
    Thus, there exist constants $\underline{B}_{-}, \overline{B}_{+} > 0$ independent of $s$ such that
    \begin{align*}
        B\bigl( -Q_{\alpha, K, \sigma}^m (s) \bigr) \geq \underline{B}_{-},
        \quad
        B\bigl( Q_{\alpha, K, \sigma}^m (s) \bigr) \leq \overline{B}_{+},
            \quad  \text{for all } s \geq 0,
    \end{align*}
    since $x\mapsto B(-x) $ is non-negative and increasing and the Bernoulli function $B$ is non-negative and decreasing.
    It follows that
    \begin{align*}
        Z_{\alpha, K}^m (s)
        \leq
        - \tau_\sigma  \underline{B}_{-} s - v_\alpha m_\sigma s
        + \tau_\sigma  \overline{B}_{+} n_{\alpha, K}^m + v_\alpha m_\sigma n_{\alpha, \sigma}^D,
    \end{align*}
    which tends to $-\infty$ as $s \rightarrow \infty$.
    For ii), we know that for $s \in (0,1]$, the function $s \mapsto \mathcal{F}_\alpha^{-1}(s) - \log(s)$ is bounded from above and below via $0 \leq \mathcal{F}_\alpha^{-1}(s) - \log(s) \leq \mathcal{F}_\alpha^{-1}(1)$ since $\mathcal{F}_\alpha^{-1}$ is monotonically increasing.
    Hence, there also exist constants $\overline{B}_{-}, \underline{B}_{+} > 0$ independent of $s$ such that
    \begin{align*}
        B\bigl( -Q_{\alpha, K, \sigma}^m  (s) \bigr) \leq \overline{B}_{-},
        \quad
        B\bigl( Q_{\alpha, K, \sigma}^m  (s) \bigr) \geq \underline{B}_{+},
            \quad  \text{for all } s \in (0,1].
    \end{align*}
    Consequently,  we have
    \begin{align*}
        \liminf_{s\rightarrow 0} Z_{\alpha, K}^m (s)
        \geq
        \tau_\sigma  \underline{B}_{+} n_{\alpha, K}^m + v_\alpha m_\sigma n_{\alpha, \sigma}^D > 0.
    \end{align*}
    Lastly, to prove iii), we calculate the derivative of $Z_{\alpha, K}^m$ which is
    \begin{align*}
        \frac{\text{d} Z_{\alpha, K}^m}{ \text{d} s}
        = &- \tau_\sigma B \bigl(-Q_{\alpha, K, \sigma}^m(s)\bigr)
        + \tau_\sigma \frac{\text{d} Q_{\alpha, K, \sigma}^m}{\text{d} s} B' \bigl(-Q_{\alpha, K, \sigma}^m(s)\bigr) s\\
        ~
        &+ \tau_\sigma \frac{\text{d} Q_{\alpha, K, \sigma}^m}{\text{d} s} B'\bigl(Q_{\alpha, K, \sigma}^m (s) \bigr) n_{\alpha, K}^m
        ~
        - v_\alpha m_\sigma.
    \end{align*}
    Due to a non-dimensionalized version of the first inequality in \eqref{eq:statistics-ineq}, we have
    $
    \text{d} Q_{\alpha, K, \sigma}^m/ \text{d} s = ( \mathcal{F}_\alpha^{-1} )' (s) - 1/s \geq 0
    $.
    Moreover, $B \geq 0$, $ B' \leq 0$, $s \geq 0$, $v_\alpha m_\sigma >0$, and, hence, $\text{d} Z_{\alpha, K}^m/ \text{d} s < 0$ which proves that the derivative of $Z_{\alpha, K}^m$ is strictly decreasing.
\end{proof}
The equations \eqref{eq:boundary-unknown-ion}, \eqref{eq:boundary-ohmic}, \eqref{eq:Schottky-psi-sigma},  and \Cref{lem:Z-zero} ensure that all boundary unknowns $ (\varphi_{\electrons, \sigma}^{m})_{\sigma\in{\mathcal E}^C}$, $(\varphi_{\holes, \sigma}^{m})_{\sigma\in{\mathcal E}^C}$, $(\varphi_{\ions, \sigma}^{m})_{\sigma\in{\mathcal E}^C}$, $(\psi_{\sigma}^{m})_{\sigma\in{\mathcal E}^C}$ can be uniquely expressed as functions of the cell center unknowns for both contact boundary models.
As a result, these boundary unknowns can be eliminated from the system.
Thus, the discretization scheme \eqref{eq:discrete-model-memristor} and \eqref{eq:discrete-Sedan} reduces to a nonlinear system of equations at each time step, where the unknowns are the quasi Fermi potentials and the electrostatic potential evaluated exclusively at the cell centers.

%%%%%%%%%%%%%%%%%%%%%%%%%%%%%%%%%%%%%%%%%%%%%%%%%%%%%%%%%%%%%%%%%%%%%%%%%%%%%%%%%%%%%%%%%%%%
\subsection{Existence of a discrete solution} \label{sec:existence}
%%%%%%%%%%%%%%%%%%%%%%%%%%%%%%%%%%%%%%%%%%%%%%%%%%%%%%%%%%%%%%%%%%%%%%%%%%%%%%%%%%%%%%%%%%%%

%%%%%%%%%%%%%%%%%%%%%%%%%%%%%%%%%%%%%%%%%%%%
A crucial step in proving the existence of a discrete solution is establishing the entropy-dissipation inequality, which depends on the discrete entropy $\mathbb{E}_\mathcal{T}^m$ and dissipation $\mathbb{D}_{\mathcal{T}}^m$.
Both are discrete analogues to the continuous counterparts introduced in \Cref{sec:cont_entropy_dissipation}.
Following the formulation in \eqref{eq:continuous-entropy}, the non-negative discrete entropy for $m \in \mathbb{N}$ is defined as
\begin{equation} \label{eq:discrete-entropy-memristor}
    \begin{split}
        \mathbb{E}_\mathcal{T}^m
        = \, &\,\frac{\lambda^2}{2} \sum_{\sigma \in \mathcal{E}} \tau_\sigma \left(D_{\sigma} ( \boldsymbol{\psi}^m - \boldsymbol{\hat{\psi}}^D)\right)^2
        ~
        + \sum_{K \in \mathcal{T}} m_K \Phi_{\ions}(n_{\ions, K}^m)\\
        ~
        &+  \delta_\electrons \sum_{K \in \mathcal{T}} m_K H_\electrons(n_{\electrons, K}^m, n_{\electrons, K}^D)
        ~
        +  \delta_\electrons \delta_\holes \sum_{K \in \mathcal{T}} m_K H_\holes(n_{\holes, K}^m, n_{\holes, K}^D),
    \end{split}
\end{equation}
where $\boldsymbol{\hat{\psi}}^D = \boldsymbol{\psi}^D + \overline{\mathbf{V}}$ for the Schottky boundary model and $\boldsymbol{\hat{\psi}}^D = \boldsymbol{\psi}^D$ for the ohmic boundary conditions.
Furthermore, we have $\Phi_{\ions}' = \mathcal{F}_\ions^{-1} $ (see \eqref{eq:entropyfunc}) and $ H_\alpha(x,y) = \Phi_\alpha(x) - \Phi_\alpha(y) - \Phi_\alpha'(y)(x-y)$ for $\alpha \in \{ \electrons, \holes\}$ (see \eqref{eq:H-function}).
The non-negative dissipation rate for $m \in \mathbb{N}$ for the Schottky boundary model is defined similarly to \eqref{eq:continuous-dissipation-Schottky} for $m \in \mathbb{N}$
\begin{equation} \label{eq:discrete-dissip-Schottky}
    \begin{split}
        \mathbb{D}_{\mathcal{T}, \reviewerTwo{\text{SC}}}^m
        =&\;  \frac{\delta_\electrons}{2\nu}  \sum_{ \sigma \in \mathcal{E} } \tau_\sigma  \overline{n}^{m}_{\electrons, \sigma}  (D_{\sigma} \boldsymbol{\varphi}_{\electrons}^m)^2
        ~
        + \frac{ \delta_\electrons \delta_\holes}{2\nu}  \sum_{ \sigma \in \mathcal{E} } \tau_\sigma  \overline{n}^{m}_{\holes, \sigma}  (D_{\sigma} \boldsymbol{\varphi}_{\holes}^m)^2\\
        ~
        &+ \frac{ z_{\ions}^2}{2}\sum_{\sigma \in \mathcal{E}} \tau_\sigma \overline{n}^{m}_{\ions,\sigma}  (D_{ \sigma} \boldsymbol{\varphi}_{\ions}^m)^2\\
        ~
        &+ \frac{\delta_\electrons}{\nu} \sum_{\sigma \in \mathcal{E}^C}  v_\electrons  m_\sigma  \bigl( \mathcal{F}^{-1}_\electrons (n_{\electrons, \sigma}^m) - \mathcal{F}^{-1}_\electrons  (n_{\electrons, \sigma}^D)  \bigr)  \left( n_{\electrons, \sigma}^m - n_{\electrons, \sigma}^D \right) \\
        ~
        &+ \frac{\delta_\electrons\delta_\holes}{\nu} \sum_{\sigma \in \mathcal{E}^C} v_\holes m_\sigma  \bigl(\mathcal{F}^{-1}_\holes (n_{\holes, \sigma}^m) - \mathcal{F}^{-1}_\holes  (n_{\holes, \sigma}^D)  \bigr)  \left( n_{\holes, \sigma}^m - n_{\holes, \sigma}^D \right),
    \end{split}
\end{equation}
and for the Dirichlet boundary model we have the discrete counterpart of \eqref{eq:continuous-dissipation-Dirichlet} for $m \in \mathbb{N}$
\begin{equation} \label{eq:discrete-dissip-Dirichlet}
    \begin{split}
        \mathbb{D}_{\mathcal{T}, \reviewerTwo{\text{OC}}}^m
        =&\;  \frac{\delta_\electrons}{2\nu}  \sum_{\sigma \in \mathcal{E}} \tau_\sigma  \overline{n}^{m}_{\electrons, \sigma}  (D_{\sigma} \boldsymbol{\varphi}_{\electrons}^m)^2
        ~
        + \frac{ \delta_\electrons \delta_\holes}{2\nu}  \sum_{\sigma \in \mathcal{E}} \tau_\sigma  \overline{n}^{m}_{\holes, \sigma}  (D_{\sigma} \boldsymbol{\varphi}_{\holes}^m)^2\\
        ~
        &+ \frac{ z_{\ions}^2}{2}\sum_{\sigma \in \mathcal{E}} \tau_\sigma \overline{n}^{m}_{\ions,\sigma}  (D_{ \sigma} \boldsymbol{\varphi}_{\ions}^m)^2.
    \end{split}
\end{equation}
In particular, the contact boundary terms over $\sigma \in \mathcal{E}^C$ in \eqref{eq:discrete-dissip-Schottky} are non-negative due to the monotonicity of the inverse of statistics function.
The entropy-dissipation inequality reads as follows and can be proven similarly to \cite[Theorem 4.4]{Abdel2023Existence}.
%%%%%%%%%%%%%%%%%%%%%%%%%%%%%%%%%
\begin{theorem}(Discrete entropy-dissipation inequality) \label{thm:discrete-E-D-memristor}
    For any solution to the finite volume scheme \eqref{eq:discrete-model-memristor}, \eqref{eq:discrete-Sedan} one has the following entropy-dissipation inequality: For any $\varepsilon>0$, there is a constant $c_{\varepsilon, \mathbf{\Omega}}>0$ such that for any $m\in\mathbb{N}$, one has
    \begin{equation} \label{eq:discrete-E-D-inequality-memristor}
        \frac{\mathbb{E}_{\mathcal{T}}^m - \mathbb{E}_{\mathcal{T}}^{m-1}}{\tau^m} +  \mathbb{D}_{\mathcal{T}}^m
        \leq c_{\varepsilon, \mathbf{\Omega}} + \varepsilon \mathbb{E}_\mathcal{T}^m,
    \end{equation}
    where the entropy $\mathbb{E}_{\mathcal{T}}^m$ is defined in \eqref{eq:discrete-entropy-memristor} and the dissipation $\mathbb{D}_{\mathcal{T}}^m \reviewerTwo{ = \mathbb{D}_{\mathcal{T}, \text{SC}}^m}$ in \eqref{eq:discrete-dissip-Schottky} (or  \reviewerTwo{$\mathbb{D}_{\mathcal{T}}^m = \mathbb{D}_{\mathcal{T}, \text{OC}}^m$ in} \eqref{eq:discrete-dissip-Dirichlet}).
    The constant $c_{\varepsilon, \mathbf{\Omega}}$ depends solely on $\varepsilon$, the measure of $\mathbf{\Omega}$, the mesh regularity, the boundary data via the norms $ \|\varphi^D\|_{W^{1,\infty}}$ and $\|\psi^D\|_{W^{1,\infty}}$, as well as on $z_{\ions}^2$ and the dimensionless parameters $\delta_\electrons$, $\delta_\holes$ and $\nu$.
    If $\nabla\varphi^D = \nabla\psi^D = \mathbf{0}$, then the right-hand side of \eqref{eq:discrete-E-D-inequality-memristor} vanishes. \QEDB
\end{theorem}
%%%%%%%%%%%%%%%%%%%%%%%%%%%%%%%%%
The following main theoretical result of our work, \Cref{thm.exresult-memristor}, establishes the existence of discrete solutions.
The proof is divided into lemmas, each using different arguments, as shown in \Cref{fig:proof-schematics}.
Most of the necessary auxiliary results are provided in \Cref{app:auxiliary-results} which follow in a straightforward manner \cite{Abdel2023Existence}.
However, to prove for the Schottky contact boundary model the \textit{a priori estimates} on the electron and hole quasi Fermi potentials in \Cref{lem.bounds.phialpha} we are in need of another lemma which guarantees the boundedness of the boundary potentials $\varphi_{\electrons, \sigma}^m$, $\varphi_{\holes, \sigma}^m$.
Particularly in the proof of \Cref{lem.bounds.phialpha}, we set $a = z_\alpha (\varphi_{\alpha, \sigma}^m - \psi_\sigma^m)$ and $b = z_\alpha (\varphi_{\alpha, \sigma}^D - \psi_\sigma^D)$.
\begin{lemma} \label{lem:bound-boundary-density}
    Let $\alpha \in \{ \electrons, \holes \}$ and $a, b \in \mathbb{R}$.
    Assume that there exists $\underline{b}, \overline{b} \in \mathbb{R}$ and $M \geq 0$, such that
    \begin{equation*}
        \left\{
        \begin{aligned}
            &\underline{b} \leq b \leq \overline{b},\\[.5em]
            & (a-b) \left( \mathcal{F}_\alpha(a) - \mathcal{F}_\alpha (b) \right) \leq M.
        \end{aligned}
        \right.
    \end{equation*}
    Then,
    \begin{equation*}
        \mathcal{F}^{-1}_\alpha \Bigl( \bigl(- \frac{\sqrt{M}}{2} + \sqrt{ \mathcal{F}_\alpha(\underline{b})}\bigr)^2 \Bigr)
        \leq
        a
        \leq
        \mathcal{F}^{-1}_\alpha \Bigl( \bigl( \frac{\sqrt{M}}{2} + \sqrt{ \mathcal{F}_\alpha(\overline{b})}\bigr)^2 \Bigr).
    \end{equation*}
\end{lemma}
\begin{proof}
    We start by estimating
    \begin{align*}
        M \geq (a-b) \left( \mathcal{F}_\alpha(a) - \mathcal{F}_\alpha (b) \right)
        &= \left( \int_{a}^{b} 1 \, ds \right) \left(  \int_{a}^{b} \mathcal{F}'_\alpha(s) \, ds\right)\\
        %%%%%
        &\geq \left(  \int_{a}^{b} \sqrt{\mathcal{F}'_\alpha(s)} \, ds\right)^2,
    \end{align*}
which is due to the Hölder inequality.
The next estimate follows from the assumption \eqref{hyp:statistics-n-p}    \begin{align*}
    \left(  \int_{a}^{b} \sqrt{\mathcal{F}'_\alpha(s)} \, ds\right)^2
    =
    4 \left(  \int_{a}^{b} \frac{\mathcal{F}'_\alpha(s)}{2 \sqrt{\mathcal{F}'_\alpha(s)}} \, ds\right)^2
    &\geq 4 \left(  \int_{a}^{b} \frac{\mathcal{F}'_\alpha(s)}{2 \sqrt{\mathcal{F}_\alpha(s)}} \, ds\right)^2\\
    %%%%%%%%%%%%%%%%%%%%%%%%%%%%%
    &= 4 \left(  \sqrt{\mathcal{F}_\alpha(b)} - \sqrt{\mathcal{F}_\alpha(a)} \right)^2,
\end{align*}
where the last equality can be proven by a substitution $u = \sqrt{\mathcal{F}_\alpha(s) }$.
By combining both estimates, we receive
\begin{align*}
    M \geq 4 \left(  \sqrt{\mathcal{F}_\alpha(b)} - \sqrt{\mathcal{F}_\alpha(a)} \right)^2,
\end{align*}
which proves the claim after rearranging terms and using $\underline{b} \leq b \leq \overline{b}$.
\end{proof}

%%%%%%%%%%%%%%%%%%%%%%%%%%%%%%%%%%%
\begin{figure}[ht!]
    \hspace*{-0.7cm}
    \includegraphics[width = 1.1\textwidth]{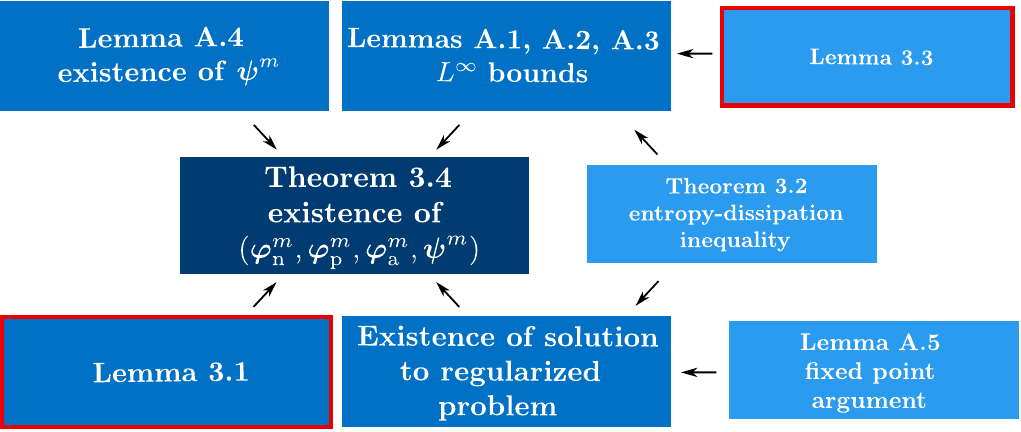}
    \caption{
        Summary of the proof schematic showing the different layers to prove \Cref{thm.exresult-memristor}.
        Lemmas and theorems without red border can be easily adapted from \cite{Abdel2023Existence} while the ones highlighted in red are the additional auxiliary results established in this work to prove the existence of discrete solutions.}
        \label{fig:proof-schematics}
\end{figure}
%%%%%%%%%%%%%%%%%%%%%%%%%%%%%%%%%%%

With this, we can formulate the existence result.
%%%%%%%%%%%%%%%%%%%%%%%%%%%%%%%%
\begin{theorem}(Existence of discrete solution)\label{thm.exresult-memristor}
    For all time steps $m\geq 1$, the implicit-in-time finite volume scheme \eqref{eq:discrete-model-memristor}, \eqref{eq:discrete-Sedan} for the TMDC-based memristor charge transport model \eqref{eq:model-dimless-memristor} supplied with initial \eqref{eq:initial-cond-memristor}, and boundary conditions \eqref{eq:memristor-BC} and \eqref{eq:memristor-Schottky-BC-dimless} (or \eqref{eq:memristor-Dirichlet-BC-dimless}) has at least one solution
    $
        (\boldsymbol{\varphi}_{\electrons}^m, \boldsymbol{\varphi}_{\holes}^m,\boldsymbol{\varphi}_{\ions}^m, \boldsymbol{\psi}^m) \in \mathbb{R}^{\theta}$
        with $\theta=4\textrm{Card}(\mathcal{T})
    $.
    Moreover, this solution satisfies the following $L^{\infty}$ bounds.
    There exists $M_B > 0$ such that
    \begin{align*}
        - M_B
        \leq
        \boldsymbol{\varphi}_{\electrons}^m, \boldsymbol{\varphi}_{\holes}^m,\boldsymbol{\varphi}_{\ions}^m, \boldsymbol{\psi}^m
        \leq M_B,
        \quad \text{for all}\; m\geq 1,
    \end{align*}
    holds componentwise.
    The constant $M_B$ depends on the data (including the non-dimensionalized parameters and the Dirichlet functions) as well as on the temporal and spatial mesh.
\end{theorem}

\begin{proof}
    The proof approach follows the methodology outlined in \cite[Theorem 5.5]{Abdel2023Existence}, adapted to our setting.
    Assuming the existence of a solution at time step $m-1$, we aim to prove the existence of a solution at time step $m$.
    To achieve this, we define the vector of unknowns $\mathbf{X}$, stated in \eqref{def.X}.
    Using \Cref{lem:discrete-poisson} and \Cref{lem:Z-zero}, the discrete electric potential and the boundary quasi Fermi potentials for electrons and holes are expressed uniquely in terms of \reviewerOne{the quasi Fermi potentials of electrons, holes and vacancies}.
    By that, we define the vector of unknowns $\mathbf{X}$, stated in \eqref{def.X}, which is expressed in terms of quasi Fermi potentials.
    \reviewerOne{
        Next, we introduce a continuous vector field $\mathbf{P}_m : \mathbb{R}^{\theta} \to \mathbb{R}^{\theta}$ with $\theta = 3\textrm{Card}(\mathcal{T})$ such that  $\mathbf{P}_m (\mathbf{X})=\mathbf{0}$ is equivalent to the continuity equations \eqref{eq:discrete-mass-balance-elec-memristor}--\eqref{eq:discrete-mass-balance-anions-memristor}.
        That is, $\mathbf{X}$ solves the continuity equations if and only if $\mathbf{P}_m(\mathbf{X}) = \mathbf{0}$.
    }
    %This allows us to define a continuous vector field $\mathbf{P}_m$, which represents the discretization scheme.
    \reviewerOne{To prove the existence of such $\mathbf{X}$, we first consider} a regularized version of $\mathbf{P}_m$ \reviewerOne{and apply} the fixed-point argument from \Cref{lem.Evans}.
    %To establish the existence of discrete solutions for $\mathbf{P}_m$, we first prove the existence for a regularized version of $\mathbf{P}_m$, applying the fixed-point argument from \Cref{lem.Evans}.
    Uniform boundedness in the regularization parameter is then shown using \Cref{lem.bounds.psi}, \Cref{lem.bounds.phialpha}, and \Cref{lem.bounds.anions}.
    For the Schottky boundary model, the proof of \Cref{lem.bounds.phialpha} relies on \Cref{lem:bound-boundary-density}.
    Finally, by passing to the limit with respect to the regularization parameter, we demonstrate the existence of solutions for the original vector field $\mathbf{P}_m$, completing the proof.
    A schematic of the proof is visualized in \Cref{fig:proof-schematics}.
\end{proof}

\reviewerOne{We end this section with two remarks about extensions of the scheme and its properties.
\begin{remark}\label{rem.time_varying}(Time-dependent boundary conditions)
    Let us informally comment on the extension of our results to time-dependent applied voltages $\overline{V}(\mathbf{x},t)$.
    Define $\mathbb{E}_\mathcal{T}^{m,n}$ as the entropy \eqref{eq:discrete-entropy-memristor} at time $t^m$ taken with respect to the boundary data at time $t^n$.
    Then, with the same computations as above, the entropy-dissipation inequality becomes
    \begin{equation*}
        \frac{\mathbb{E}_{\mathcal{T}}^{m,m} - \mathbb{E}_{\mathcal{T}}^{m-1,m}}{\tau^m} + \mathbb{D}_{\mathcal{T}}^{m,m} \leq c_{\varepsilon, \mathbf{\Omega}} + \varepsilon \mathbb{E}_\mathcal{T}^{m,m}.
    \end{equation*}
    This inequality is sufficient to extend Theorem~\ref{thm.exresult-memristor} by the same proof, as it yields a control of $\mathbb{E}_{\mathcal{T}}^{m,m}$ and $\mathbb{D}_{\mathcal{T}}^{m,m}$. Thus, one can get existence of a discrete solution. However, the stability of this solution is not granted as $\mathbb{E}_{\mathcal{T}}^{m,m}$ and $\mathbb{D}_{\mathcal{T}}^{m,m}$ may not be uniformly bounded with respect to $m$ on finite time intervals. More precisely, in order to obtain
    \begin{equation*}
        \frac{\mathbb{E}_{\mathcal{T}}^{m,m} - \mathbb{E}_{\mathcal{T}}^{m-1,m-1}}{\tau^m} + \mathbb{D}_{\mathcal{T}}^{m,m} \leq c_{\varepsilon, \mathbf{\Omega}} + \varepsilon \mathbb{E}_\mathcal{T}^{m,m},
    \end{equation*}
one needs additional regularity assumptions on $\overline{V}$, which are needed to control new remainder terms coming from the estimate of $\mathbb{E}_{\mathcal{T}}^{m-1,m} - \mathbb{E}_{\mathcal{T}}^{m-1,m-1}$.
\end{remark}
%%%%%%%%%%%%%%%%%%%%%%%%%%%%%%%%%%%%%%%%%%%%%%%%%%%
\begin{remark}[Cell centers on boundary]\label{rem.centeronboundary}
    The hypothesis $d_\sigma>0$ for all edges $\sigma$ can be relaxed to cover a wider class of meshes.
    In particular, consider boundary cells with centers located on $\partial\Omega$.
    By letting $d_\sigma \rightarrow 0$, for $\sigma\subset\partial\Omega$ in the scheme \eqref{eq:discrete-model-memristor}, \eqref{eq:discrete-Sedan}, one obtains a limit scheme.
    This limit process could be justified rigorously using the compactness properties provided by the uniform bounds from the entropy-dissipation estimate.
    Up to redefining appropriately the boundary data \eqref{eq:def-integral-average}, the boundedness of the entropy implies that $\psi_K^m -\psi^D_\sigma\rightarrow 0$ as $d_\sigma\rightarrow 0$, since $\tau_\sigma\rightarrow\infty$. Starting from this, one can repeat the elimination process of Section~\ref{sec:discrete-boundary-values}, %check that \eqref{eq:implicit_eq_boundary} becomes $B\left(- Q_{\alpha, K, \sigma}^m \right) n_{\alpha, \sigma}^m - B\left( Q_{\alpha, K, \sigma}^m\right)n_{\alpha, K}^m =0$
    %as $d_\sigma$ goes to $0$,
    to find that $n_{\alpha,K}^m-n_{\alpha,\sigma}^m\rightarrow 0$ as $d_\sigma\rightarrow 0$.
    In particular, if $\textbf{x}_K\in \mathcal{E}^C$, the Schottky boundary flux \eqref{eq:discrete-flux-boundary} simply becomes
    \begin{equation*}
        J_{\alpha,K, \sigma}^m = z_\alpha v_\alpha m_\sigma (n_{\alpha, K}^m - n_{\alpha, \sigma}^D), \quad \text{for} \; \sigma \in \mathcal{E}^C.
    \end{equation*}
\end{remark}}

%%%%%%%%%%%%%%%%%%%%%%%%%%%%%%%%%%%%%%%%%%%%%%%%%%%%%%%%%%%%%%%%%%%%%%%%%%%
%%%%%%%%%%%%%%%%%%%%%%%%%%%%%%%%%%%%%%%%%%%%%%%%%%%%%%%%%%%%%%%%%%%%%%%%%%%
\section{Device simulations} \label{sec:numerics}
%%%%%%%%%%%%%%%%%%%%%%%%%%%%%%%%%%%%%%%%%%%%%%%%%%%%%%%%%%%%%%%%%%%%%%%%%%%
%%%%%%%%%%%%%%%%%%%%%%%%%%%%%%%%%%%%%%%%%%%%%%%%%%%%%%%%%%%%%%%%%%%%%%%%%%%
In this section, we perform numerical experiments to analyze the different electrode models and configurations for their application in memtransistor and memristive device simulations.
We introduce the simulation parameters and geometry in \Cref{sec:params-geometry}.
Then, in \Cref{sec:num-comparison-contact-models}, we compare the Schottky and ohmic contact model using a simplified 1D model geometry.
Subsequently, \Cref{sec:num-2D-sim} explores more complex 2D geometries with variations in electrode lengths and memristive layer thicknesses to examine the impact of side, top, and mixed contacts on the I-V characteristics.
The discretization schemes are implemented in the open-source software tool \texttt{ChargeTransport.jl} \cite{ChargeTransport}, which is used for all simulations.
\reviewerTwo{
    The nonlinear discrete system is solved with a damped Newton method, using analytical Jacobians obtained via automatic differentiation in the programming language Julia.
    With the previous time step as an initial guess, the solver converges within a few iterations.
    Linear systems are handled by the sparse direct solver \texttt{UMFPACK} \cite{Davis2004}, and absolute and relative tolerances of $10^{-8}$ are imposed to ensure accuracy.
}
%%%%%%%%%%%%%%%%%%%%%%%%%%%%%%%%%%%%%%%%%%%%%%%%%%%%%%%%%%%%%%%%%%%%%%%%%%%
%%%%%%%%%%%%%%%%%%%%%%%%%%%%%%%%%%%%%%%%%%%%%%%%%%%%%%%%%%%%%%%%%%%%%%%%%%%
\subsection{Model parameters and geometry} \label{sec:params-geometry}
%%%%%%%%%%%%%%%%%%%%%%%%%%%%%%%%%%%%%%%%%%%%%%%%%%%%%%%%%%%%%%%%%%%%%%%%%%%
%%%%%%%%%%%%%%%%%%%%%%%%%%%%%%%%%%%%%%%%%%%%%%%%%%%%%%%%%%%%%%%%%%%%%%%%%%%
For the numerical experiments, we choose the rescaling factors and non-dimensionalized parameters introduced in \Cref{sec:nondimens-model} such that the resulting solutions correspond to a realistic TMDC-based memristive device with MoS$_2$ as TMDC material.
The parameters for MoS$_2$ and the dimensions of the channel geometry, which we use in our simulations, are provided in \Cref{tab:TMDC-general} and are based on \cite{Spetzler.2024}.
For the simulations in \Cref{sec:num-comparison-contact-models}, the MoS$_2$ layer is modeled by a 1D domain, effectively approximating a side contact configuration.
In \Cref{sec:num-2D-sim}, we extend the geometry to a 2D representation of the MoS$_2$ layer (in the $x-z$ plane of the schematics in \Cref{fig:intro-TMDC}), incorporating contact boundary conditions to simulate side, top, and mixed contact configurations.

\begin{table}[!ht]
    \centering
    {\footnotesize
    \begin{tabular}{|c|l|c|l|}
            \hline
            Physical quantity	& Symbol & Value & Unit  \\\hline&&&\\
            %%%%%%%%%%%%%%%%%%%%%%%%%%%%%%%%%%%%%%%%%
            Channel length & $h_\text{C}$& $ 1$ &  \si{\micro\metre} \\[1.0ex]
            Channel width & $h_\text{W}$ & $ 10$ & \si{\micro\metre}\\[1.0ex]
            Channel thickness & $h_\text{T}$ & $ 15$ & \si{\nano\metre} \\[1.0ex]
            %%%%%%%%%%%%%%%%%%%%%%%%%%%%%%%%%%%%%%%%%
            Relative permittivity & $\varepsilon_r$ & $10$  & \\[1.0ex]
            %%%%%%%%%%%%%%%%%%%%%%%%%%%%%%%%%%%%%%%%%
            Electron mobility	& $\mu_\electrons$ & \num{2.5e-4}  & \si{\metre^2/(\V\s)} \\[0.5ex]
            Hole mobility	& $\mu_\holes$  & \num{2.5e-4} &  \si{\metre^2/(\V\s)}  \\[0.5ex]
            Defect mobility	& $\mu_\ions$  & \num{5e-14}  & \si{\metre^2/(\V\s)}  \\[0.5ex]
            %%%%%%%%%%%%%%%%%%%%%%%%%%%%%%%%%%%%%%%%%
            Conduction band-edge energy$\!$	& $E_\electrons$  & $-4.0$ &\si{\electronvolt}  \\[1.0ex]
            Band gap	& $E_\text{g}$  &$1.3$  &  \si{\electronvolt} \\[1.0ex]
            Valence band-edge energy		& $E_\holes$  &$-5.3$  &  \si{\electronvolt} \\[1.0ex]
            Intrinsic defect energy	& $E_\ions$  & $-4.32$ & eV \\[1.0ex]
            %%%%%%%%%%%%%%%%%%%%%%%%%%%%%%%%%%%%%%%%%
            Eff.\ conduction band DoS & $N_\electrons$ & $ \num{1e25}$ &\si{\metre^{-3}} \\[1.0ex]
            Eff.\ valence band DoS & $N_\holes$ & $ \num{1.5e25}$ &   \si{\metre^{-3}} \\[1.0ex]
            Max.\ defect density & $N_\ions$ & $ \num{1e28}$ & \si{\metre^{-3}} \\[1.0ex]
            %%%%%%%%%%%%%%%%%%%%%%%%%%%%%%%%%%%%%%%%%
            Doping density & $z_CC$ &$\num{1.0e21} $ &  \si{\metre^{-3}}  \\[1.0ex]
            %%%%%%%%%%%%%%%%%%%%%%%%%%%%%%%%%%%%%%%%%
            Schottky barrier & $\phi_{0}$ & $0.001$ &  eV \\[1.0ex]
            Boundary quasi Fermi potential & $\varphi_{0}$ & 0 &  V \\[1.0ex]
            Electron recomb.\ velocity	& $v_{\electrons}^*$  & \num{3.6e4} &\si{\metre/\second} \\[1.0ex]
            Hole recomb.\ velocity & $v_{\holes}^*$  & \num{3.2e4}  & \si{\metre/\second} \\[1.0ex]
            Voltage amplitude & $V_\text{max}$ & $ 13 $  & V  \\[1.0ex]
            \hline
    \end{tabular}
    }
    \caption{
        Summary of the MoS$_2$ material and sample-specific parameters from \cite[S$_1$ in Table 2-4]{Spetzler.2024} for
        a constant temperature $T = 300$ K.
        }
    \label{tab:TMDC-general}
\end{table}

For all simulations, we assume that the migrating ionic defects are sulfur vacancies with charge number $z_\ions = 1$.
At the right contact, two piecewise linear voltage cycles are applied, as illustrated in \Cref{fig:1D-different-BC}a, with a voltage sweep rate of $5$ V/s, i.e., a cycle duration of $10.4$ s.
As demonstrated in \cite{Spetzler.2024}, the parameters provided in \Cref{tab:TMDC-general} produce I-V curves that closely match the experimentally measured data from \cite{DaLi.2018}.

%%%%%%%%%%%%%%%%%%%%%%%%%%%%%%%%%%%%%%%%%%%%%%%%%%%%%%%%%%%%%%%%%%%%%%%%%%%
%%%%%%%%%%%%%%%%%%%%%%%%%%%%%%%%%%%%%%%%%%%%%%%%%%%%%%%%%%%%%%%%%%%%%%%%%%%
\subsection{Comparison of ohmic and Schottky boundary conditions} \label{sec:num-comparison-contact-models}
%%%%%%%%%%%%%%%%%%%%%%%%%%%%%%%%%%%%%%%%%%%%%%%%%%%%%%%%%%%%%%%%%%%%%%%%%%%
%%%%%%%%%%%%%%%%%%%%%%%%%%%%%%%%%%%%%%%%%%%%%%%%%%%%%%%%%%%%%%%%%%%%%%%%%%%

Both Schottky \cite{Spetzler.2024, Sivan.2022} and ohmic boundary models \cite{Jourdana.2023, Jungel.2023, Herda.2024} have been used in the literature.
In the following, we analyze the impact of these two types of boundary conditions on the device characteristics for a realistic MoS$_2$-based memristive device, as described in \Cref{sec:params-geometry}.

\begin{figure}[!ht]
	\hspace*{-1.6cm}
	\includegraphics[width=1.2\textwidth]{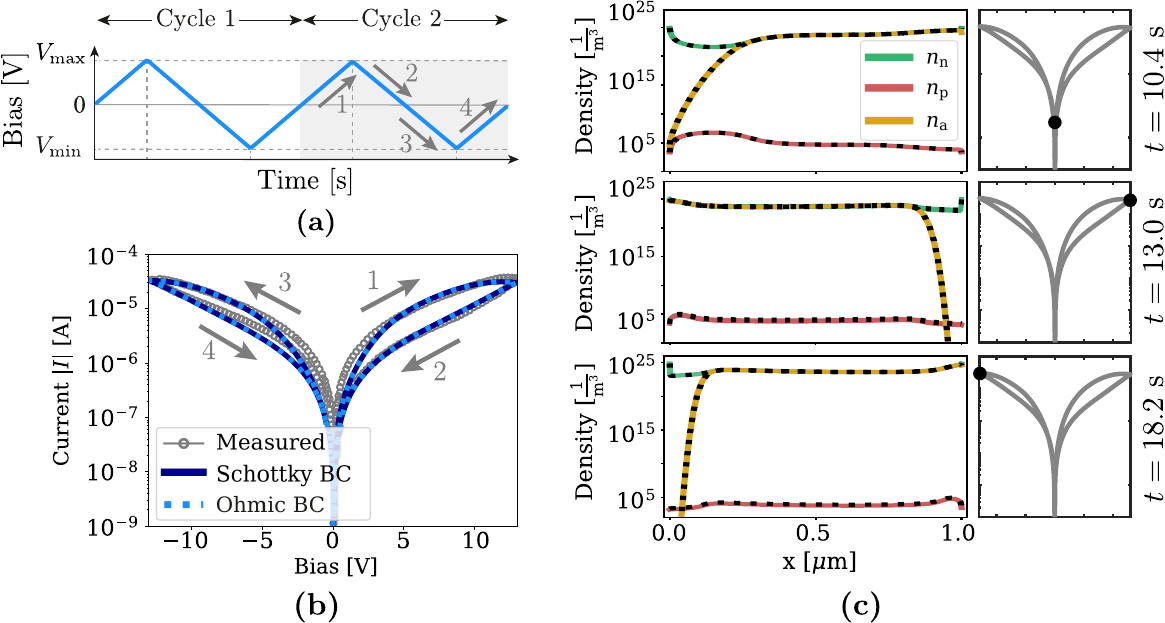}
    \caption{
        Comparison of charge carrier densities and I-V curves for the two contact boundary models \eqref{eq:BC-Schottky-time-dependent} and \eqref{eq:BC-Ohmic-time-dependent}.
        \textbf{(a)} Two consecutive voltage cycles applied at the right contact.
        All subsequent simulations correspond to the shaded second cycle.
        \textbf{(b)} Measured I-V curve from \cite{DaLi.2018} compared with second-cycle I-V simulations using Schottky boundary conditions \eqref{eq:BC-Schottky-time-dependent} as well as ohmic boundary conditions \eqref{eq:BC-Ohmic-time-dependent}.
        Arrows indicate the direction of hysteresis of all curves.
        \textbf{(c)} Electron, hole and defect densities $n_\electrons$, $n_\holes$, $n_\ions$, respectively, during the second cycle at selected times.
        Colored lines correspond to the Schottky model \eqref{eq:BC-Schottky-time-dependent}.
        In contrast, black dotted lines represent the densities computed with ohmic boundary conditions \eqref{eq:BC-Ohmic-time-dependent}.
    }
    \label{fig:1D-different-BC}
\end{figure}

To this end, we conduct two simulations.
For the first simulation, we apply Schottky boundary conditions \eqref{eq:BC-Schottky-time-dependent} with a small Schottky barrier height of $\phi_{0} =0.001\,\mathrm{eV}$.
In the second simulation, we employ ohmic boundary conditions \eqref{eq:BC-Ohmic-time-dependent}.
The simulations use a 1D model geometry representing the MoS$_2$ channel, to approximate a side contact configuration, as illustrated in \Cref{fig:device-schematics-memristor}a.
At the right electrode, piecewise linear voltage cycles are applied, as shown in \Cref{fig:1D-different-BC}a.
As explained in \cite{Spetzler.2024}, we expect the I-V characteristics of the first cycle to differ significantly from the subsequent cycles, primarily due to the varying time scales of ionic defect migration compared to the dynamics of electrons and holes.
Therefore, we focus on the second-cycle I-V curve as it is more representative for the device characteristics.
We define the total current as in \cite[Section 50.8]{Farrell.2017}, including the ionic defect current density contribution.

\begin{figure}[ht!]
	\hspace*{-0.2cm}
	\includegraphics[width=1.02\textwidth]{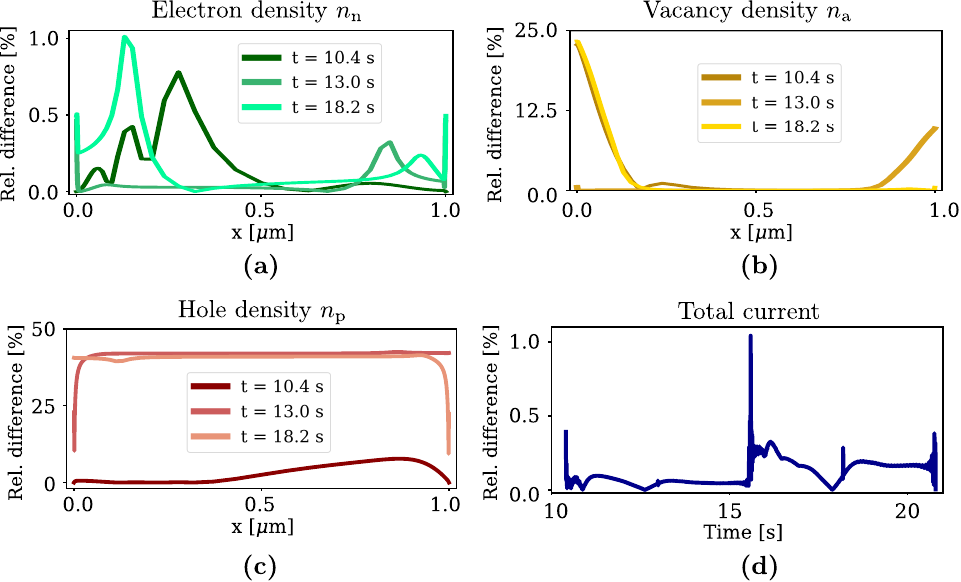}
    \caption{
        Relative differences between the two boundary models \eqref{eq:BC-Schottky-time-dependent} and \eqref{eq:BC-Ohmic-time-dependent} in the calculated carrier densities and the second-cycle total current.
        Spatially resolved differences in \textbf{(a)} electron, \textbf{(b)} vacancy, and \textbf{(c)} hole densities, calculated via \eqref{eq:error-dens}.
        The relative differences are shown for the three selected times in \Cref{fig:1D-different-BC}c.
        \textbf{(d)}  Relative difference in total current over time, as defined by \eqref{eq:error-IV}.
    }
    \label{fig:1D-error}
\end{figure}

\Cref{fig:1D-different-BC}b illustrates the measured I-V curve from \cite{DaLi.2018} alongside the two simulated I-V curves, based on the Schottky (dark blue) and ohmic (light blue, dashed) contact boundary model, respectively.
Both simulations produce almost identical I-V curves that match well with the experimental data.
All three curves show significant hysteresis, with a clockwise direction in the right hysteresis branch and a counterclockwise direction in the left branch, as indicated by the arrows in \Cref{fig:1D-different-BC}b.
The hysteresis is caused by the slow drift and redistribution of mobile vacancies under the applied voltage, leading to the formation and annihilation of a vacancy depletion zone near the contacts as explained previously \cite{Spetzler.2024}.
In \Cref{fig:1D-different-BC}c, we compare the spatially resolved charge carrier densities of electrons $n_\electrons$, holes $n_\holes$, and defects $n_\ions$ for both contact models at three selected times during the second voltage cycle.
No significant deviations between the ohmic and the Schottky boundary cases are visible for any of the three selected times.
At all three points in time, holes are minority carriers with small densities $<10^5\,\mathrm{m}^{-3}$, while the electron density mostly follows the vacancy density with orders of magnitude larger values of up to $10^{25}\,\mathrm{m}^{-3}$.
Hence, the vacancies act as mobile donor-type defects and their concentration locally determines the majority carrier density.
Only in proximity of the left contact ($t = 10.4\,\mathrm{s}$ and $t = 18.2\,\mathrm{s}$) and the right contact ($t = 13\,\mathrm{s}$) the electron and vacancy densities are orders of magnitude lower compared to the rest of the spatial domain.
These vacancy depletion zones are the result of the voltage-driven vacancy dynamics, which is consistent with a previous analysis of the memristive switching mechanism \cite{Spetzler.2024}.

To quantify the errors in the carrier densities and the I-V curves between the two contact boundary models, we define relative differences. The relative difference in the carrier densities for the Schottky $n_{\alpha, \text{SC}}$ and the ohmic boundary models $n_{\alpha, \text{OC}}$ is defined for $\alpha \in \{ \electrons, \holes, \ions\}$ and a fixed time $t^* \geq 0$ as
\begin{align} \label{eq:error-dens}
\left| n_{\alpha, \text{SC}}(\mathbf{x}, t^*) - n_{\alpha, \text{OC}}(\mathbf{x}, t^*) \right| / \left| \, n_{\alpha, \text{SC}}(\mathbf{x}, t^*) \right|,
\quad
\mathbf{x} \in \mathbf{\Omega},
\end{align}
and the relative difference in the total current for the Schottky $I_{\text{SC}}$ and the ohmic boundary model $ I_{\text{OC}}$ is given by
\begin{align}\label{eq:error-IV}
\left| I_{\text{SC}}(t) - I_{\text{OC}}(t) \right| / \left| \, I_{\text{SC}}(t) \right|,
\quad
t \geq 0.
\end{align}

\Cref{fig:1D-error}a-c depicts the relative differences in the charge carrier densities \eqref{eq:error-dens} as a function of space for the three selected times, corresponding to those in \Cref{fig:1D-different-BC}c.
Additionally, \Cref{fig:1D-error}d shows the relative difference in the total current as a function of time \eqref{eq:error-IV}.

Overall, minor discrepancies are observed in the electron densities (\Cref{fig:1D-error}a) and the currents (\Cref{fig:1D-error}d) with relative errors remaining below $1$ \%.
In contrast, significantly larger errors are visible in the vacancy density (\Cref{fig:1D-error}b) and the hole density (\Cref{fig:1D-error}c), with relative differences reaching up to nearly $25$ \% and $40$ \%, respectively.
However, the large differences in the vacancy density are limited to the depletion zones at the left and right boundaries of the simulation domain, where the absolute density values are comparatively small.
Similarly, the large errors in the hole density are negligible since the holes, being minority carriers, contribute far less to the total current than the electrons.

Hence, for our model setups with small Schottky barrier heights  $\phi_{0} = 0.001$  eV, it can be concluded that the ohmic boundary conditions provide a valid approximation of the more complicated Schottky contact boundary model even though larger discrepancies are observed for the minority carriers near the electrodes.
In scenarios with negligible Schottky barriers, charge transport in memristive devices can be adequately described using the simplified Dirichlet boundary conditions used in the ohmic boundary model \eqref{eq:BC-Ohmic-time-dependent}.
However, when Schottky barrier heights become significant, models employing purely Dirichlet boundary conditions, as in \cite{Jourdana.2023, Jungel.2023, Herda.2024}, may fail to capture the device characteristics accurately.
In such cases, the more sophisticated Robin-type boundary conditions of the Schottky boundary model \eqref{eq:BC-Schottky-time-dependent} are more appropriate.

\reviewerTwo{
    The choice between ohmic and Schottky models is thus barrier-dependent: ohmic boundaries align with current mathematical analysis results for weak and discrete solutions (approximation for low barriers), while Schottky boundaries can provide greater physical accuracy in high-barrier regimes at the expense of additional nonlinearities in the model.
}

%%%%%%%%%%%%%%%%%%%%%%%%%%%%%%%%%%%%%%%%%%%%%%%%%%%%%%%%%%%%%%%%%%%%%%%%%%%
%%%%%%%%%%%%%%%%%%%%%%%%%%%%%%%%%%%%%%%%%%%%%%%%%%%%%%%%%%%%%%%%%%%%%%%%%%%
\subsection{Exploring 2D contact configurations} \label{sec:num-2D-sim}
%%%%%%%%%%%%%%%%%%%%%%%%%%%%%%%%%%%%%%%%%%%%%%%%%%%%%%%%%%%%%%%%%%%%%%%%%%%
%%%%%%%%%%%%%%%%%%%%%%%%%%%%%%%%%%%%%%%%%%%%%%%%%%%%%%%%%%%%%%%%%%%%%%%%%%%

The choice of a suitable contact and device geometry is a major step in the design of memristive devices and memtransistors.
This includes in the simplest case the geometry of the semiconductor-metal interface as well as the thickness of the memristive material.
However, implementing realistic geometries in a numerical model typically comes at the expense of high complexity and computational costs as well as an extensive evaluation process \cite{Allain.2015, Kang.2014, Wang.2022}.
Even for idealized model geometries comprising only the memristive material and the electrodes, several types of contact geometries can be devised that correspond to reported experimental cases \cite{Kang.2014}.
In the following, we analyze under which conditions these geometries can lead to similar results using 2D model geometries and if they can be simplified by computationally inexpensive 1D geometries.

\begin{figure}[!ht]
	\hspace*{-2.5cm}
	\includegraphics[width=1.3\textwidth]{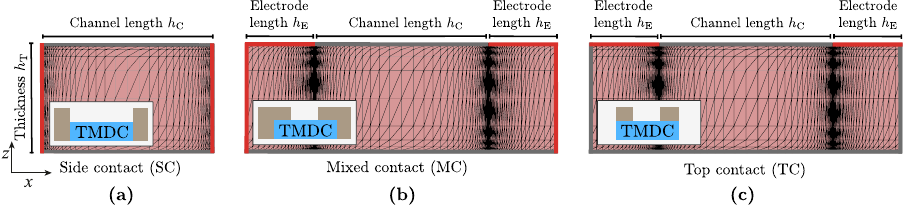}
    \caption{Three different contact configurations.
    More precisely, we have \textbf{(a)} the side contact (SC), \textbf{(b)} the mixed contact (MC), and \textbf{(c)} the top contact (TC).
    The metal-semiconductor interface\reviewerTwo{, where Schottky boundary conditions are imposed,} is highlighted in red.
    \reviewerTwo{Homogeneous Neumann boundary conditions are applied to the remaining parts of the boundary.}
    In the following, we vary the memristive layer thickness $h_{\text{T}}$ and the top electrode length $h_{\text{E}}$ while the channel length $h_\text{C}$ stays the same.}
    \label{fig:2D-contactTypes}
\end{figure}

We implement three different 2D model geometries, illustrated in \Cref{fig:2D-contactTypes}, to represent the $x-z$ plane cross-sections of the three 3D contact configurations, introduced in \Cref{fig:device-schematics-memristor}.
Schottky boundary conditions \eqref{eq:BC-Schottky-time-dependent} are applied at the red-highlighted metal-semiconductor contacts in \Cref{fig:2D-contactTypes}a-c.
\reviewerTwo{Homogeneous Neumann boundary conditions are applied to the remaining parts of the boundary.}
Each configuration is characterized by the thickness of the memristive material  $h_\mathrm{T}$, the channel length $h_\mathrm{C}$, defined as the distance between the top left and right electrode, and the top electrode length $h_\text{E}$.
For the side contact configuration (\Cref{fig:2D-contactTypes}a), we have $h_\text{E} = 0$ nm, as the boundary conditions are exclusively applied at the sides.
In contrast, the mixed contact configuration (\Cref{fig:2D-contactTypes}b) and the top contact configuration (\Cref{fig:2D-contactTypes}c) feature boundary conditions applied to the top of the domain over a length $h_\mathrm{E}$ (electrode length) at left and right contact side.
The mixed configuration includes additional side contacts, while the top contact configuration considers only the top electrodes without side electrodes.
For all simulations, the channel length $h_\text{C}$ is kept constant, while the electrode length $h_\text{E}$ and the memristive layer thickness $h_\text{T}$ are varied systematically.
Moreover, we consider a zero applied voltage at the left and simulate two piecewise linear voltage cycles, as depicted in \Cref{fig:1D-different-BC}a, applied at the right contact.

To measure the error between the mixed and either the side or top contact configuration, we define the relative $l^2$ error of the total current $I_i(t)$ over the two simulated two voltage cycles as
\begin{align} \label{eq:sim-error}
    e_{i, j}= \sqrt{\sum_{k=1}^M \Bigl| |I_{i}(t_k)| - |I_{j}(t_k)| \Bigr|^2} \Big/ \sqrt{\sum_{k=1}^M |I_{i}(t_k)|^2},
\end{align}
where the indices $i \neq j$ with $i, j \in \{ {\text{SC}, \text{MC}, \text{TC}}\}$ represent the different contact configurations: side (SC), mixed (MC), and top contacts (TC).
Furthermore, $M$ refers to the number of time steps.
The error functional \eqref{eq:sim-error} accounts for the overall deviation between the current curves throughout the simulation rather than just capturing the pointwise maximum deviation at a single time step.
In the following, we analyze the defined $l^2$ error as a function of two model parameters, varied during the simulations: the ratio of the electrode $h_\text{E}$ and the channel length $h_\text{C}$, and the thickness of the flake $h_\text{T}$.
For example, a ratio of $h_\text{E}/h_\text{C}=30$\% indicates that 30\% of the fixed channel length has been added on either side of the channel, corresponding to a top or mixed contact configuration.

\begin{figure}[!ht]
	\begin{subfigure}[b]{0.49\textwidth}
       \hspace*{-0.4cm} \includegraphics[width =2.1\textwidth]{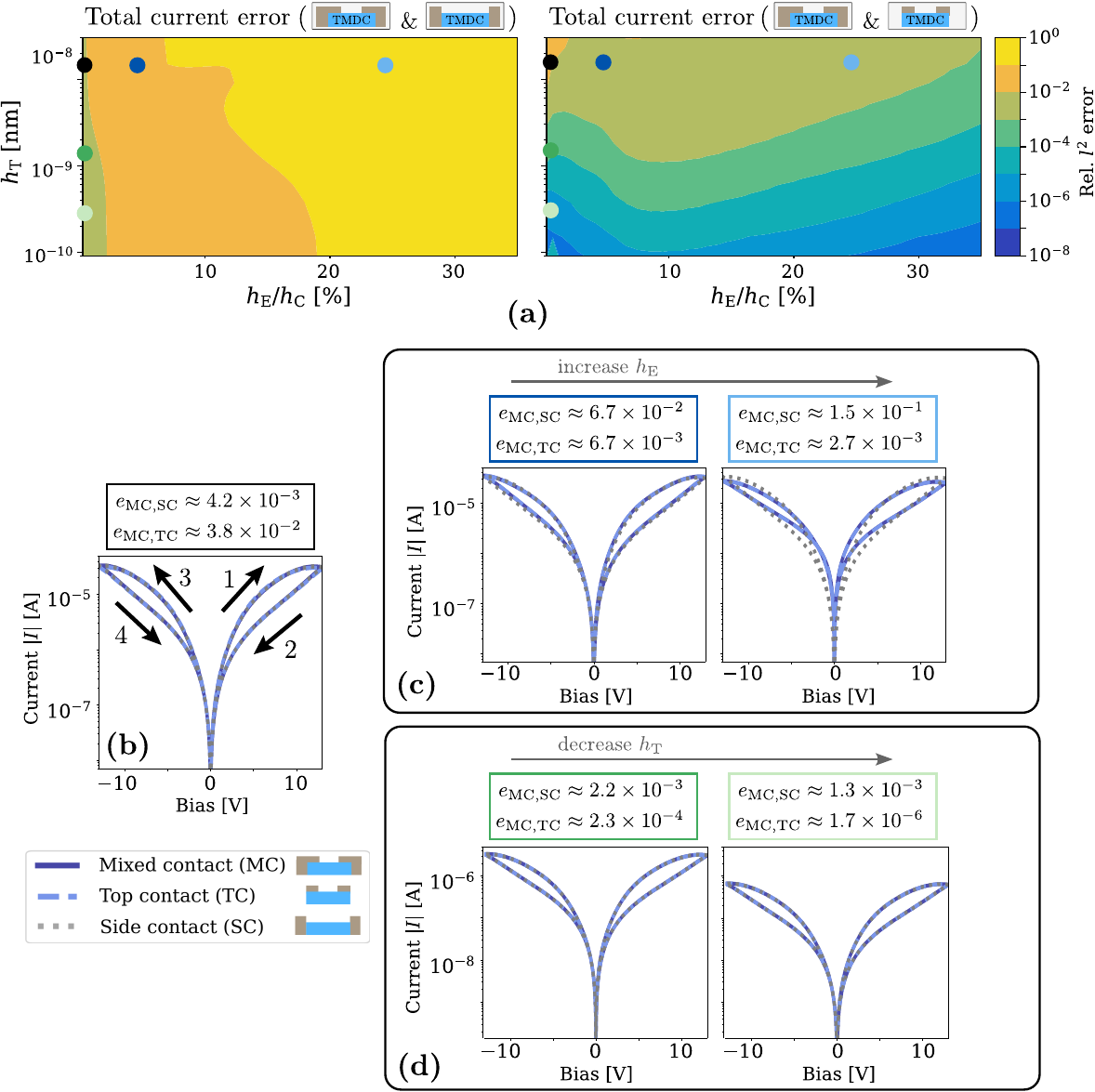}
    \end{subfigure}
    \caption{
        \textbf{(a)} Relative $l^{2}$ error in the total current for varying electrode lengths and MoS$_2$ thicknesses: comparing mixed vs.\ top contact (left) and mixed vs.\ side contact (right).
        Second-cycle I-V curves are shown for mixed (darkblue), top (dashed lightblue), and side contacts (dotted gray).
        We have the I-V curves for
        \textbf{(b)} negligible $h_\text{E}$ and $h_{\text{T}} = 15$ nm, matching with experimental observations \cite{Spetzler.2024},
        \textbf{(c)} for fixed $h_{\text{T}} = 15$ nm, with increasing electrode length $h_\text{E}$ from left to right,
        \textbf{(d)} for fixed and negligible $h_\text{E}$, with decreasing thickness $\text{h}_\text{T}$ from left to right.
        The arrows in (a) indicate the hysteresis direction.
        The relative $l^{2}$ errors are also stated in colored boxes, corresponding to the dots in (a).
        }
        \label{fig:geometry-study}
\end{figure}

\Cref{fig:geometry-study}a illustrates the relative $l^2$ error of the total current $I(t)$  between the mixed and either the side (left panel) or top contact configuration (right panel).
The relative $l^2$ errors are plotted logarithmically as a function of the ratio $h_\text{E}/h_\text{C}$ ($x$-axis) and the flake thickness $h_\text{T}$ ($y$-axis).
Selected second-cycle I-V curves corresponding to specific geometry setups are shown in \Cref{fig:geometry-study}b-d, with the geometries marked as colored dots in \Cref{fig:geometry-study}a.
The I-V curves are illustrated for the mixed (dark blue), top (dashed light blue), and side contact configurations (dotted gray).
In \Cref{fig:geometry-study}b, the initial state I-V characteristics are depicted for negligible $h_\text{E}$ and $h_{\text{T}} = 15$ nm, with indicated hysteresis directions matching with the original device geometry setup from \cite{DaLi.2018}, resulting in the curves shown in \Cref{fig:1D-different-BC}b.
\Cref{fig:geometry-study}c displays second-cycle I-V curves for fixed $h_{\text{T}} = 15$ nm, with increasing electrode length $h_\text{E}$ from left to right, demonstrating the effect of electrode scaling.
In contrast, \Cref{fig:geometry-study}d explores the impact of decreasing the flake thickness $h_\text{T}$ (from left to right) for fixed and negligible $h_\text{E}$.
The relative $l^{2}$ errors for the I-V curves are also reported in colored boxes in \Cref{fig:geometry-study}b-d, corresponding directly to the dots in \Cref{fig:geometry-study}a.

\begin{figure}[!ht]
	\begin{subfigure}[b]{0.53\textwidth}
       \hspace*{-1.9cm} \includegraphics[width =2.3\textwidth]{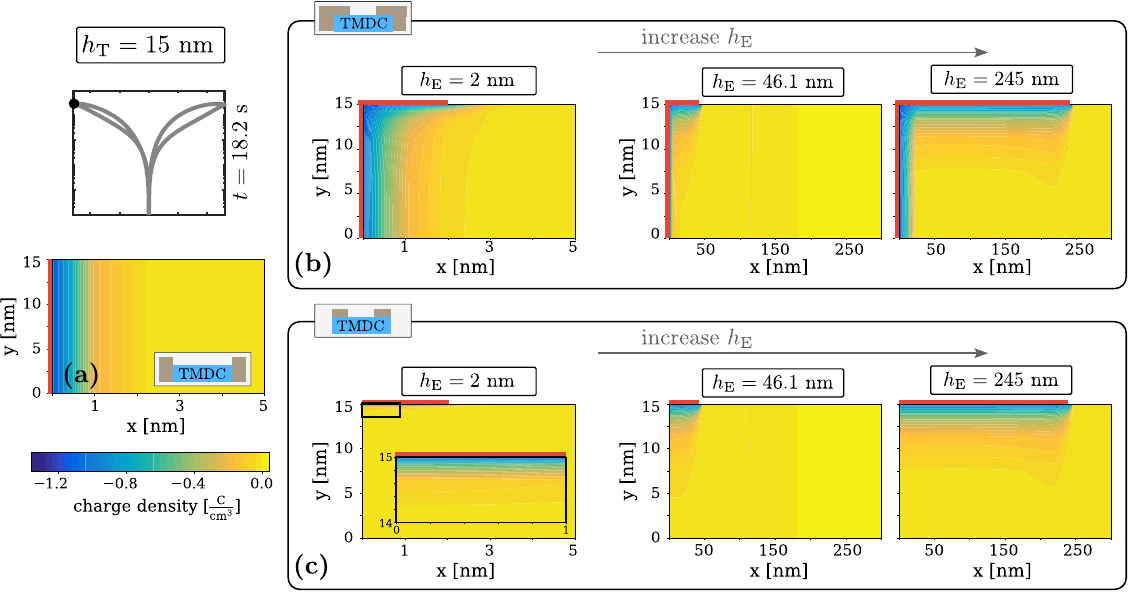}
    \end{subfigure}
    %%%%%%%%%%%%%%%%%%%
    \caption{
        Spatial distribution of the space charge density, i.e., the right-hand side of Poisson's equation at the left contact for a memristive layer thickness of $15$ nm (actual thickness from measurements \cite{DaLi.2018}).
        The charge density is shown at a time $t = 18.2$ s, corresponding to an applied voltage of $V = V_{\text{min}} = -13$ V during the second voltage cycle.
        \textbf{(a)} For the side contact configuration, the charge density is displayed for negligible electrode lengths, providing a baseline for comparison.
        The charge density is visualized for three different electrode lengths, increasing progressively from left to right for the \textbf{(b)} mixed and \textbf{(c)} top contact configuration.
        Note that the figure in (a) and the leftmost panels in (b) and (c) feature a different $x$-axis scaling compared to the other panels, as they represent the case of negligible electrode lengths.
    }
    \label{fig:geometry-study-space-charge-dens-thickness-1p5e-8}
\end{figure}

From \Cref{fig:geometry-study}a and \Cref{fig:geometry-study}c (left), it is apparent that for a fixed flake thickness, when the contact covers only a small portion of the device's top, the I-V curves for all three contact configurations (side, mixed, and top) are nearly identical, with relative $l^2$ errors below $10^{-1}$.
This observation holds even as the flake thickness decrease, as shown in \Cref{fig:geometry-study}d.
However, when a larger portion of the device's top is covered by electrodes, the relative $l^2$ errors in \Cref{fig:geometry-study}a show opposing trends.
While the error between the mixed and top configurations (\Cref{fig:geometry-study}a, right) remains small, the error between the mixed and side contact configuration (\Cref{fig:geometry-study}a, left) increases and appears to level off eventually.
These observations have several implications:
First, the truly 2D configurations (top and mixed) lead to different results compared to a simple 1D setup.
Second, for the total current, the mixed and top configurations yield similar results, suggesting that most of the current flows through the top contact region in both cases.

In \Cref{fig:geometry-study-space-charge-dens-thickness-1p5e-8} and \Cref{fig:geometry-study-space-charge-dens-thickness-1p5e-9}, we illustrate the space charge density, defined as the right-hand side of Poisson's equation \eqref{eq:model-poisson-memristor}, near the left electrode for two different memristive layer thicknesses.
The space charge density is shown at a fixed time of $t = 18.2$, corresponding to the second cycle with an applied voltage of $V = V_{\text{min}} = -13$ V.
This choice of time corresponds to a vacancy depletion zone at the left contact, as seen in \Cref{fig:1D-different-BC}c (bottom).

\begin{figure}[ht!]
	\begin{subfigure}[b]{0.53\textwidth}
       \hspace*{-1.9cm} \includegraphics[width =2.3\textwidth]{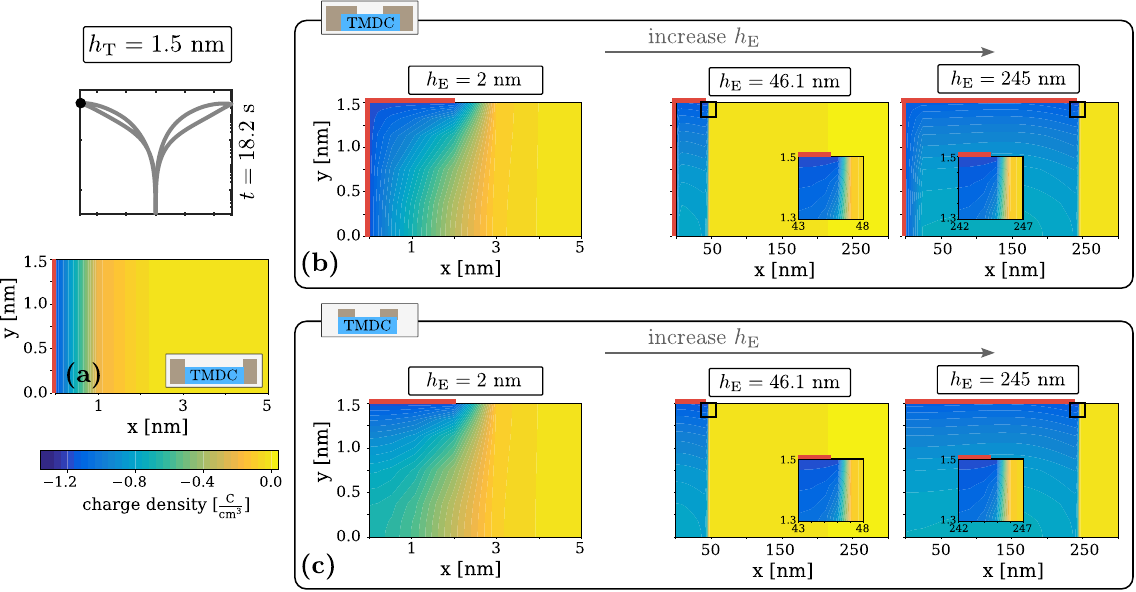}
    \end{subfigure}
    %%%%%%%%%%%%%%%%%%%
    \caption{
        Spatial distribution of the space charge density, i.e., the right-hand side of Poisson's equation at the left contact for a memristive layer thickness of $1.5$ nm.
        The charge density is shown at a time $t = 18.2$ s, corresponding to an applied voltage of $V = V_{\text{min}} = -13$ V during the second voltage cycle.
        \textbf{(a)} For the side contact configuration, the charge density is displayed for negligible electrode lengths, providing a baseline for comparison.
        The charge density is visualized for three different electrode lengths, increasing progressively from left to right for the \textbf{(b)} mixed and \textbf{(c)} top contact configuration.
        Note that the figure in (a) and the leftmost panels in (b) and (c) feature a different $x$-axis scaling compared to the other panels, as they represent the case of negligible electrode lengths.
    }
    \label{fig:geometry-study-space-charge-dens-thickness-1p5e-9}
\end{figure}

\Cref{fig:geometry-study-space-charge-dens-thickness-1p5e-8} shows the space charge density for a thickness of $h_\text{T} = 15$ nm, corresponding to the actual MoS$_2$ thickness from the measurements in \cite{DaLi.2018}, while \Cref{fig:geometry-study-space-charge-dens-thickness-1p5e-9} shows the space charge density for a smaller thickness of $h_\text{T} = 1.5$ nm.
In both figures, \Cref{fig:geometry-study-space-charge-dens-thickness-1p5e-8}a and \Cref{fig:geometry-study-space-charge-dens-thickness-1p5e-9}a, we visualize the charge density for the side contact configuration, where we only consider a negligible electrode length. Since the charge flows only in the horizontal direction for the side contact, this setup can be viewed as an artificially extended 1D configuration.
In \Cref{fig:geometry-study-space-charge-dens-thickness-1p5e-8}b and \Cref{fig:geometry-study-space-charge-dens-thickness-1p5e-9}b, we show the space charge density for the mixed contact, and in \Cref{fig:geometry-study-space-charge-dens-thickness-1p5e-8}c and \Cref{fig:geometry-study-space-charge-dens-thickness-1p5e-9}c for a top contact configuration.
We visualize the space charge density for increasing electrode length from left to right for both the mixed and top contact configurations.
Specifically, we consider three different cases for the ratio $h_\text{E}/h_\text{T}$: i) $h_\text{E}/h_\text{T} < 1 $ \%, ii) $h_\text{E}/h_\text{T} = 4.61 $ \% and iii) $h_\text{E}/h_\text{T} = 24.5 $ \%.

For a negligibly small electrode length of $h_\text{E} = 2$ nm, the space charge density distributions across all contact types remain similar.
Moreover, no major differences are observed between the mixed (\Cref{fig:geometry-study-space-charge-dens-thickness-1p5e-8}b, \Cref{fig:geometry-study-space-charge-dens-thickness-1p5e-9}b) and top contact configurations (\Cref{fig:geometry-study-space-charge-dens-thickness-1p5e-8}c, \Cref{fig:geometry-study-space-charge-dens-thickness-1p5e-9}c) for both memristive layer thicknesses, except for a slight accumulation of charge at the actual left boundary in the mixed contact configuration.
This observation further emphasizes the similarities in charge distribution between the top and mixed configurations.
\Cref{fig:geometry-study-space-charge-dens-thickness-1p5e-9} provides further insight into the behavior of the space charge density in thinner devices.
In these configurations, the charge is concentrated directly beneath the contacts.
This charge accumulation explains the decreased relative $l^2$ error between the top and mixed contact configurations, as observed in \Cref{fig:geometry-study}a.
More precisely, the figures of the space charge densities support the observations from the error plots:
Regardless of the memristive layer thickness, for small electrode lengths (i.e., $ <10$ \% of the channel length), 1D models and simulations accurately approximate more complex multi-dimensional calculations.
Interestingly, there is no significant difference between using mixed or top contacts in actual 2D setups.
The error between these two contact configurations diminishes as the thickness decreases due to the charge accumulation directly beneath the electrodes.

%%%%%%%%%%%%%%%%%%%%%%%%%%%%%%%%%%%%%%%%%%%%%%%%%%%%%%%%%%%%%%%%%%%%%%%%%%%
%%%%%%%%%%%%%%%%%%%%%%%%%%%%%%%%%%%%%%%%%%%%%%%%%%%%%%%%%%%%%%%%%%%%%%%%%%%
\section{Conclusion} \label{sec:conclusion}
%%%%%%%%%%%%%%%%%%%%%%%%%%%%%%%%%%%%%%%%%%%%%%%%%%%%%%%%%%%%%%%%%%%%%%%%%%%
%%%%%%%%%%%%%%%%%%%%%%%%%%%%%%%%%%%%%%%%%%%%%%%%%%%%%%%%%%%%%%%%%%%%%%%%%%%

In this work, we numerically investigated a vacancy-assisted drift-diffusion charge transport model for 2D TMDC-based memristive devices with various electrode models.
First, we developed a novel, physics-preserving, implicit-in-time finite volume scheme for the numerical discretization of the model equations.
We demonstrated the unconditional stability, structure preservation, and existence of discrete solutions in multiple dimensions using the entropy method, thereby ensuring the mathematical robustness of the numerical method.
Building on this, we extended the results from \cite{Abdel2023Existence} to incorporate the Schottky boundary model, addressing it for the first time in the mathematical literature.

Next, we applied the model to explore various contact configurations through a series of numerical simulations. The results provide insights into the conditions under which simplified one-dimensional models are adequate approximations, and when a multi-dimensional approach is required to accurately capture the complexity of charge transport in these systems.
For small electrode lengths, regardless of the memristive layer thickness, we observed that 1D simplifications reliably approximate complex multi-dimensional simulations.
Additionally, we found that the mixed and top contact configurations yield similar results across different electrode and memristive layer thicknesses, as most of the current flows through the top contact region in both cases.
Our findings lay a solid foundation for further refining the modeling techniques for TMDC-based memristive devices and memtransistors.

%%%%%%%%%%%%%%%%%%%%%%%%%%%%%%%%%%%%%%%%%%%%%%%%%%%%%%%%%%%%%%%%%

\vspace{1em}
\noindent
\textit{Acknowledgments}
This work was partially supported by
the Leibniz competition 2020 (NUMSEMIC, J89/2019),
the Deutsche Forschungsgemeinschaft (DFG, German Research Foundation) under Germany's Excellence Strategy -- The Berlin Mathematics Research Center MATH+ (EXC-2046/1, project ID: 390685689),
the LabEx CEMPI (ANR-11-LABX0007)
as well as
the Carl-Zeiss Foundation via the Project Memwerk
and the Deutsche Forschungsgemeinschaft (DFG, German Research Foundation) -- Project-ID 434434223 – SFB 1461.
CCH and MH acknowledge the support of the CDP C2EMPI, together with the French State under the France-2030 programme, the University of Lille, the Initiative of Excellence of the University of Lille, the European Metropolis of Lille for their funding and support of the R-CDP-24-004-C2EMPI project. Part of the research was conducted within the context of the Inria - WIAS Berlin ARISE Associate Team.

%%%%%%%%%%%%%%%%%%%%%%%%%%%%%%%%%%%%%%%%%%%%%%%%%%%%%%%%%%%%%%%%%

\vspace{1em}
\noindent
\textit{Author contributions}
All authors read, edited and approved the manuscript.
DA: Data curation, Investigation, Software, Writing -- original draft.
MH: Formal analysis, Writing -- review \& editing.
MZ: Funding acquisition, Writing -- review \& editing.
CCH: Writing -- review \& editing.
BS: Software, Writing -- original draft, Supervision.
PF: Software, Writing -- original draft, Funding acquisition, Supervision.
%%%%%%%%%%%%%%%%%%%%%%%%%%%%%%%%%%%%%%%%%%%%%%%%%%%%%%%%%%%%%%%%%%%%%%%%%%%
%%%%%%%%%%%%%%%%%%%%%%%%%%%%%%%%%%%%%%%%%%%%%%%%%%%%%%%%%%%%%%%%%%%%%%%%%%%

\appendix

%%%%%%%%%%%%%%%%%%%%%%%%%%%%%%%%%%%%%%%%%%%%%%%%%%%%%%%%%%%%%%%%%%%%%%%%%%%
%%%%%%%%%%%%%%%%%%%%%%%%%%%%%%%%%%%%%%%%%%%%%%%%%%%%%%%%%%%%%%%%%%%%%%%%%%%
\section{Auxiliary results} \label{app:auxiliary-results}
%%%%%%%%%%%%%%%%%%%%%%%%%%%%%%%%%%%%%%%%%%%%%%%%%%%%%%%%%%%%%%%%%%%%%%%%%%%
%%%%%%%%%%%%%%%%%%%%%%%%%%%%%%%%%%%%%%%%%%%%%%%%%%%%%%%%%%%%%%%%%%%%%%%%%%%

To prove the existence result in \Cref{thm.exresult-memristor} we need several auxiliary results which are stated in the following.
All of them are based on an adaption of the results in \cite{Abdel2023Existence}.
In the following, we define the vector of unknowns
\begin{equation}\label{def.X}
    \mathbf{X} = \Bigl( (\varphi_{\electrons,K}-\varphi_{\electrons,K}^D)_{K\in\mathcal{T}}, (\varphi_{\holes,K}-\varphi_{\holes,K}^D)_{K\in\mathcal{T}},(\varphi_{\ions,K}-\psi_{K}^D)_{K\in\mathcal{T}} \Bigr).
\end{equation}
%%%%%%%%%%%%%%%%%%%%%%%%%%%%%%%%%%%%%%%%%%%%%%%%%%%%%%%%%%%%%%%%%%%%%%%%%%%
%%%%%%%%%%%%%%%%%%%%%%%%%%%%%%%%%%%%%%%%%%%%%%%%%%%%%%%%%%%%%%%%%%%%%%%%%%%
\subsection{A priori bounds}
%%%%%%%%%%%%%%%%%%%%%%%%%%%%%%%%%%%%%%%%%%%%%%%%%%%%%%%%%%%%%%%%%%%%%%%%%%%
%%%%%%%%%%%%%%%%%%%%%%%%%%%%%%%%%%%%%%%%%%%%%%%%%%%%%%%%%%%%%%%%%%%%%%%%%%%

First, we start with the \textit{a priori bounds} which follow from a bound on the entropy and dissipation (given by the entropy-dissipation inequality and a discrete Grönwall inequality).

\begin{lemma}(Bound for electrostatic potential; from \cite{Abdel2024Thesis}[Lemma 4.11]) \label{lem.bounds.psi}
    Assume that there exists a constant $M_E>0$ such that ${\mathbb E}_\mathcal{T}(\mathbf{X})\leq M_E$.
    Then, there exists some $M_B>0$ depending on $M_E$, $\lambda$, $\mathcal{T}$, and $\psi^D$ such that
    \begin{equation}\label{bounds.psi}
        -M_B \leq \psi_K  \leq M_B, \quad \forall K\in\mathcal{T}.
    \end{equation}
    \QEDB
\end{lemma}

%%%%%%%%%%%%%%%%%%%%%%%%%%%%%%%%%%%%%
%%%%%%%%%%%%%%%%%%%%%%%%%%%%%%%%%%%%%
For the following two results, we introduce the notation
\begin{align*}
    {\mathbb D}_{\mathcal{T},\electrons}(\mathbf{X} )= \, \frac{ \delta_\electrons }{2 \nu } &\sum_{\sigma \in \mathcal{E} }\tau_\sigma \overline{n}_{\electrons, \sigma}  ( D_{\sigma} \boldsymbol{\varphi}_{\electrons})^2, \quad
    ~
    {\mathbb D}_{\mathcal{T},\holes}(\mathbf{X} )= \frac{ \delta_\electrons \delta_\holes }{2 \nu } \sum_{\sigma \in \mathcal{E} }\tau_\sigma \overline{n}_{\holes, \sigma}  ( D_{\sigma} \boldsymbol{\varphi}_{\holes})^2, \\
    ~
    &{\mathbb D}_{\mathcal{T},\ions}(\mathbf{X})=\frac{ z_{\ions}^2 }{2}\sum_{\sigma \in \mathcal{E} }\tau_\sigma \overline{n}_{\ions,\sigma}  ( D_{\sigma} \boldsymbol{\varphi}_{\ions})^2.
\end{align*}
%%%%%%%%%%%%%%%%%%%%%%%%%%%%%%%%%%%%%%%%%%%%%%
\begin{lemma}(Bounds for electron and hole quasi Fermi potentials)\label{lem.bounds.phialpha}
    Let $\alpha \in \{ \electrons, \holes \}$.
    Assume that there exists $M_E>0$ such that ${\mathbb E}_\mathcal{T}(\mathbf{X})\leq M_E$ and $M_{D}>0$ such that ${\mathbb D}_{\mathcal{T}}(\mathbf{X})\leq M_{D}$.
    Then, there exists some $ M_B > 0$ depending on $M_E$, $M_D$, $\mathcal{T}$, $\psi^D$, $\varphi^D$, $z_\alpha$, and the dimensionless constants $\nu$, $\delta_\electrons$ and $\delta_\holes$ such that
    \begin{equation}\label{bounds.phialpha}
        - M_B \leq \varphi_{\alpha,K} \leq M_B, \quad \forall K\in\mathcal{T}.
    \end{equation}
\end{lemma}
\begin{proof}
    For ohmic boundary conditions \eqref{eq:memristor-Schottky-BC-dimless}, the proof follows a similar approach as in \cite{Abdel2023Existence}[Lemma 5.4].
    The process involves reformulating the dissipation functions, introducing a coercive face dissipation functional, and identifying at least one control volume with a bounded quasi Fermi potential.
    This boundedness is guaranteed by \Cref{lem.bounds.psi} and the Dirichlet boundary conditions \eqref{eq:memristor-Dirichlet-BC-dimless}.
    A similar approach is taken for Schottky boundary conditions \eqref{eq:memristor-Dirichlet-BC-dimless}.
    However, in this case, the boundedness of at least one quasi Fermi potential is established through \Cref{lem.bounds.psi} and \Cref{lem:bound-boundary-density}.
    Specifically, we use \Cref{lem:bound-boundary-density} with $a = z_\alpha (\varphi_{\alpha, \sigma}^m - \psi_\sigma^m)$ and $b = z_\alpha (\varphi_{\alpha, \sigma}^D - \psi_\sigma^D)$.
\end{proof}

%%%%%%%%%%%%%%%%%%%%%%%%%%%%%%%%%%%%%
%%%%%%%%%%%%%%%%%%%%%%%%%%%%%%%%%%%%%

\begin{lemma}(Bound for vacancy quasi Fermi potential; from \cite{Abdel2023Existence}[Lemma 5.3])\label{lem.bounds.anions}
    Assume that there exists $M_{D}>0$ such that ${\mathbb D}_{\mathcal{T},\ions}(\mathbf{X})\leq M_{D}$ and that there also exists ${\bar n}\in (0,1)$ such that
    \begin{align}  \label{eq:lem.mass-conserv-a}
        \frac{1}{| \mathbf{\Omega}|}\sum_{K\in\mathcal{T}_{ \ions }}m_K n_{\ions, K}={\bar n}.
    \end{align}
    Then, there exists some $ M_B >0$ depending on $M_{D}$, ${\bar n}$ and $\mathcal{T}$ such that
    \begin{equation}\label{bounds.phia}
        - M_B \leq \varphi_{\ions,K}\leq M_B, \quad \forall K\in\mathcal{T}_{\ions}.
    \end{equation}
    \QEDB
\end{lemma}

%%%%%%%%%%%%%%%%%%%%%%%%%%%%%%%%%%%%%%%%%%%%%%%%%%%%%%%%%%%%%%%%%%%%%%%%%%%
%%%%%%%%%%%%%%%%%%%%%%%%%%%%%%%%%%%%%%%%%%%%%%%%%%%%%%%%%%%%%%%%%%%%%%%%%%%
\subsection{Existence of potentials}
%%%%%%%%%%%%%%%%%%%%%%%%%%%%%%%%%%%%%%%%%%%%%%%%%%%%%%%%%%%%%%%%%%%%%%%%%%%
%%%%%%%%%%%%%%%%%%%%%%%%%%%%%%%%%%%%%%%%%%%%%%%%%%%%%%%%%%%%%%%%%%%%%%%%%%%

\begin{lemma}(Existence of electric potential; from \cite{Abdel2023Existence}[Lemma 5.1])\label{lem:discrete-poisson}
    Let $\mathbf{X}$ denote the vector containing the unknown quasi Fermi potentials as defined in \eqref{def.X}.
    Suppose that $\mathbf{X}$ is given.
    Then, there exists a unique solution $\boldsymbol{\psi}(\mathbf{X})$ to the discrete nonlinear Poisson equation \eqref{eq:model-poisson-discrete-memristor}.
    Furthermore, the mapping $\mathbf{X} \mapsto \boldsymbol{\psi}(\mathbf{X})$ is continuous.
    \QEDB
\end{lemma}

Lastly, the following corollary of Brouwer's fixed point theorem is needed to establish the existence of discrete quasi Fermi potentials.

\begin{lemma}(From \cite[Section 9.1]{Evans.2010})\label{lem.Evans}
    Let $N\in\mathbb{N} $ and $\mathbf{P}: \mathbb{R}^N\to \mathbb{R}^N$ be a continuous vector field. Assume that there exists $M_B > 0$ such that $\mathbf{P}(\mathbf{X})\cdot \mathbf{X} \geq 0$, if $\Vert \mathbf{X}\Vert = M_B $. Then, there exists $\mathbf{X}^{\ast}\in\mathbb{R}^N$ such that $\mathbf{P}(\mathbf{X}^{\ast})=\mathbf{0}$ and $\Vert \mathbf{X}^{\ast}\Vert \leq M_B$. \QEDB
\end{lemma}

%%%%%%%%%%%%%%%%%%%%%%%%%%%%%%%%%%%%%%%%%%%%%%%%%%%%%%%%%%%%%%%%%%%%%%%%%%%
%%%%%%%%%%%%%%%%%%%%%%%%%%%%%%%%%%%%%%%%%%%%%%%%%%%%%%%%%%%%%%%%%%%%%%%%%%%

\bibliographystyle{elsarticle-num}
\bibliography{literature}

\begin{thebibliography}{10}
\expandafter\ifx\csname url\endcsname\relax
  \def\url#1{\texttt{#1}}\fi
\expandafter\ifx\csname urlprefix\endcsname\relax\def\urlprefix{URL }\fi
\expandafter\ifx\csname href\endcsname\relax
  \def\href#1#2{#2} \def\path#1{#1}\fi

\bibitem{Jones.2018}
N.~Jones, How to stop data centres from gobbling up the world's electricity, Nature 561~(7722) (2018) 163--166.
\newblock \href {https://doi.org/10.1038/d41586-018-06610-y} {\path{doi:10.1038/d41586-018-06610-y}}.

\bibitem{Lin.2021}
H.-Y. Lin, Colors of artificial intelligence, Computer 54~(11) (2021) 95--99.
\newblock \href {https://doi.org/10.1109/MC.2021.3102359} {\path{doi:10.1109/MC.2021.3102359}}.

\bibitem{Bughin.2018}
J.~Bughin, J.~Seong, J.~Manyika, M.~Chui, R.~Joshi, {Notes from the AI frontier: Modeling the impact of AI on the world economy}, McKinsey Global Institute, Discussion Paper (2018).

\bibitem{Rao.2023}
M.~Rao, H.~Tang, J.~Wu, W.~Song, M.~Zhang, W.~Yin, Y.~Zhuo, F.~Kiani, B.~Chen, X.~Jiang, H.~Liu, H.-Y. Chen, R.~Midya, F.~Ye, H.~Jiang, Z.~Wang, M.~Wu, M.~Hu, H.~Wang, Q.~Xia, N.~Ge, J.~Li, J.~J. Yang, {Thousands of conductance levels in memristors integrated on CMOS}, Nature 615~(7954) (2023) 823--829.
\newblock \href {https://doi.org/10.1038/s41586-023-05759-5} {\path{doi:10.1038/s41586-023-05759-5}}.

\bibitem{Xu.2018}
X.~Xu, Y.~Ding, S.~X. Hu, M.~Niemier, J.~Cong, Y.~Hu, Y.~Shi, Scaling for edge inference of deep neural networks, Nature Electronics 1~(4) (2018) 216--222.
\newblock \href {https://doi.org/10.1038/s41928-018-0059-3} {\path{doi:10.1038/s41928-018-0059-3}}.

\bibitem{Ziegler.2018}
M.~Ziegler, C.~Wenger, E.~Chicca, H.~Kohlstedt, Tutorial: concepts for closely mimicking biological learning with memristive devices: Principles to emulate cellular forms of learning, Journal of Applied Physics 124~(15) (2018) 152003.
\newblock \href {https://doi.org/10.1063/1.5042040} {\path{doi:10.1063/1.5042040}}.

\bibitem{Sung.2018}
C.~Sung, H.~Hwang, I.~K. Yoo, Perspective: A review on memristive hardware for neuromorphic computation, Journal of Applied Physics 124~(15) (2018) 151903.
\newblock \href {https://doi.org/10.1063/1.5037835} {\path{doi:10.1063/1.5037835}}.

\bibitem{Choi.2020}
S.~Choi, J.~Yang, G.~Wang, Emerging memristive artificial synapses and neurons for energy-efficient neuromorphic computing, Advanced Materials 32~(51) (2020) e2004659.
\newblock \href {https://doi.org/10.1002/adma.202004659} {\path{doi:10.1002/adma.202004659}}.

\bibitem{Zidan.2018}
M.~A. Zidan, J.~P. Strachan, W.~D. Lu, The future of electronics based on memristive systems, Nature Electronics 1~(1) (2018) 22--29.
\newblock \href {https://doi.org/10.1038/s41928-017-0006-8} {\path{doi:10.1038/s41928-017-0006-8}}.

\bibitem{Song.2023}
M.-K. Song, J.-H. Kang, X.~Zhang, W.~Ji, A.~Ascoli, I.~Messaris, A.~S. Demirkol, B.~Dong, S.~Aggarwal, W.~Wan, S.-M. Hong, S.~G. Cardwell, I.~Boybat, J.-s. Seo, J.-S. Lee, M.~Lanza, H.~Yeon, M.~Onen, J.~Li, B.~Yildiz, J.~A. {Del Alamo}, S.~Kim, S.~Choi, G.~Milano, C.~Ricciardi, L.~Alff, Y.~Chai, Z.~Wang, H.~Bhaskaran, M.~C. Hersam, D.~Strukov, H.-S.~P. Wong, I.~Valov, B.~Gao, H.~Wu, R.~Tetzlaff, A.~Sebastian, W.~Lu, L.~Chua, J.~J. Yang, J.~Kim, Recent advances and future prospects for memristive materials, devices, and systems, ACS Nano 17~(13) (2023) 11994--12039.
\newblock \href {https://doi.org/10.1021/acsnano.3c03505} {\path{doi:10.1021/acsnano.3c03505}}.

\bibitem{Prezioso.2015}
M.~Prezioso, F.~Merrikh-Bayat, B.~D. Hoskins, G.~C. Adam, K.~K. Likharev, D.~B. Strukov, Training and operation of an integrated neuromorphic network based on metal-oxide memristors, Nature 521~(7550) (2015) 61--64.
\newblock \href {https://doi.org/10.1038/nature14441} {\path{doi:10.1038/nature14441}}.

\bibitem{Huang.2024}
Y.~Huang, T.~Ando, A.~Sebastian, M.-F. Chang, J.~J. Yang, Q.~Xia, Memristor-based hardware accelerators for artificial intelligence, Nature Reviews Electrical Engineering 1~(5) (2024) 286--299.
\newblock \href {https://doi.org/10.1038/s44287-024-00037-6} {\path{doi:10.1038/s44287-024-00037-6}}.

\bibitem{Boybat.2018}
I.~Boybat, M.~{Le Gallo}, S.~R. Nandakumar, T.~Moraitis, T.~Parnell, T.~Tuma, B.~Rajendran, Y.~Leblebici, A.~Sebastian, E.~Eleftheriou, Neuromorphic computing with multi-memristive synapses, Nature Communications 9~(1) (2018) 2514.
\newblock \href {https://doi.org/10.1038/s41467-018-04933-y} {\path{doi:10.1038/s41467-018-04933-y}}.

\bibitem{Gebregiorgis.2023}
A.~Gebregiorgis, A.~Singh, A.~Yousefzadeh, D.~Wouters, R.~Bishnoi, F.~Catthoor, S.~Hamdioui, Tutorial on memristor-based computing for smart edge applications, Memories - Materials, Devices, Circuits and Systems 4 (2023) 100025.
\newblock \href {https://doi.org/10.1016/j.memori.2023.100025} {\path{doi:10.1016/j.memori.2023.100025}}.

\bibitem{Li.2023b}
J.~Li, H.~Abbas, D.~S. Ang, A.~Ali, X.~Ju, Emerging memristive artificial neuron and synapse devices for the neuromorphic electronics era, Nanoscale horizons 8~(11) (2023) 1456--1484.
\newblock \href {https://doi.org/10.1039/d3nh00180f} {\path{doi:10.1039/d3nh00180f}}.

\bibitem{Lanza.2022}
M.~Lanza, A.~Sebastian, W.~D. Lu, M.~{Le Gallo}, M.-F. Chang, D.~Akinwande, F.~M. Puglisi, H.~N. Alshareef, M.~Liu, J.~B. Roldan, Memristive technologies for data storage, computation, encryption, and radio-frequency communication, Science 376~(6597) (2022) eabj9979.
\newblock \href {https://doi.org/10.1126/science.abj9979} {\path{doi:10.1126/science.abj9979}}.

\bibitem{Xia.2019}
Q.~Xia, J.~J. Yang, Memristive crossbar arrays for brain-inspired computing, Nature Materials 18~(4) (2019) 309--323.
\newblock \href {https://doi.org/10.1038/s41563-019-0291-x} {\path{doi:10.1038/s41563-019-0291-x}}.

\bibitem{Aguirre.2024}
F.~Aguirre, A.~Sebastian, M.~{Le Gallo}, W.~Song, T.~Wang, J.~J. Yang, W.~Lu, M.-F. Chang, D.~Ielmini, Y.~Yang, A.~Mehonic, A.~Kenyon, M.~A. Villena, J.~B. Rold{\'a}n, Y.~Wu, H.-H. Hsu, N.~Raghavan, J.~Su{\~n}{\'e}, E.~Miranda, A.~Eltawil, G.~Setti, K.~Smagulova, K.~N. Salama, O.~Krestinskaya, X.~Yan, K.-W. Ang, S.~Jain, S.~Li, O.~Alharbi, S.~Pazos, M.~Lanza, Hardware implementation of memristor-based artificial neural networks, Nature Communications 15~(1) (2024) 1974.
\newblock \href {https://doi.org/10.1038/s41467-024-45670-9} {\path{doi:10.1038/s41467-024-45670-9}}.

\bibitem{Berggren.2021}
K.~Berggren, Q.~Xia, K.~K. Likharev, D.~B. Strukov, H.~Jiang, T.~Mikolajick, D.~Querlioz, M.~Salinga, J.~R. Erickson, S.~Pi, F.~Xiong, P.~Lin, C.~Li, Y.~Chen, S.~Xiong, B.~D. Hoskins, M.~W. Daniels, A.~Madhavan, J.~A. Liddle, J.~J. McClelland, Y.~Yang, J.~Rupp, S.~S. Nonnenmann, K.-T. Cheng, N.~Gong, M.~A. Lastras-Monta{\~n}o, A.~A. Talin, A.~Salleo, B.~J. Shastri, T.~F. de~Lima, P.~Prucnal, A.~N. Tait, Y.~Shen, H.~Meng, C.~Roques-Carmes, Z.~Cheng, H.~Bhaskaran, D.~Jariwala, H.~Wang, J.~M. Shainline, K.~Segall, J.~J. Yang, K.~Roy, S.~Datta, A.~Raychowdhury, Roadmap on emerging hardware and technology for machine learning, Nanotechnology 32~(1) (2021) 012002.
\newblock \href {https://doi.org/10.1088/1361-6528/aba70f} {\path{doi:10.1088/1361-6528/aba70f}}.

\bibitem{Yao.2019}
X.~Yao, Y.~Wang, X.~Lang, Y.~Zhu, Q.~Jiang, Thickness-dependent bandgap of transition metal dichalcogenides dominated by interlayer van der waals interaction, Physica E: Low-dimensional Systems and Nanostructures 109 (2019) 11--16.
\newblock \href {https://doi.org/10.1016/j.physe.2018.12.037} {\path{doi:10.1016/j.physe.2018.12.037}}.

\bibitem{Han.2011}
S.~W. Han, H.~Kwon, S.~K. Kim, S.~Ryu, W.~S. Yun, D.~H. Kim, J.~H. Hwang, J.-S. Kang, J.~Baik, H.~J. Shin, S.~C. Hong, Band-gap transition induced by interlayer van der waals interaction in {MoS$_2$}, Physical Review B 84~(4) (2011) 045409.
\newblock \href {https://doi.org/10.1103/PhysRevB.84.045409} {\path{doi:10.1103/PhysRevB.84.045409}}.

\bibitem{Sangwan.2018c}
V.~K. Sangwan, H.-S. Lee, H.~Bergeron, I.~Balla, M.~E. Beck, K.-S. Chen, M.~C. Hersam, Multi-terminal memtransistors from polycrystalline monolayer molybdenum disulfide, Nature 554~(7693) (2018) 500--504.
\newblock \href {https://doi.org/10.1038/nature25747} {\path{doi:10.1038/nature25747}}.

\bibitem{Feng.2021}
X.~Feng, S.~Li, S.~L. Wong, S.~Tong, L.~Chen, P.~Zhang, L.~Wang, X.~Fong, D.~Chi, K.-W. Ang, Self-selective multi-terminal memtransistor crossbar array for in-memory computing, ACS Nano 15~(1) (2021) 1764--1774.
\newblock \href {https://doi.org/10.1021/acsnano.0c09441} {\path{doi:10.1021/acsnano.0c09441}}.

\bibitem{Leng.2023}
Y.-B. Leng, Y.-Q. Zhang, Z.~Lv, J.~Wang, T.~Xie, S.~Zhu, J.~Qin, R.~Xu, Y.~Zhou, S.-T. Han, Recent progress in multiterminal memristors for neuromorphic applications, Advanced Electronic Materials (2023).
\newblock \href {https://doi.org/10.1002/aelm.202300108} {\path{doi:10.1002/aelm.202300108}}.

\bibitem{Ding.2021}
G.~Ding, B.~Yang, R.-S. Chen, W.-A. Mo, K.~Zhou, Y.~Liu, G.~Shang, Y.~Zhai, S.-T. Han, Y.~Zhou, Reconfigurable {2D} {WSe$_2$}-based memtransistor for mimicking homosynaptic and heterosynaptic plasticity, Small 17~(41) (2021) e2103175.
\newblock \href {https://doi.org/10.1002/smll.202103175} {\path{doi:10.1002/smll.202103175}}.

\bibitem{Wali.2023}
A.~Wali, S.~Das, Two--dimensional memtransistors for non--von neumann computing: progress and challenges, Advanced Functional Materials (2023).
\newblock \href {https://doi.org/10.1002/adfm.202308129} {\path{doi:10.1002/adfm.202308129}}.

\bibitem{Lee.2020b}
H.-S. Lee, V.~K. Sangwan, W.~A.~G. Rojas, H.~Bergeron, H.~Y. Jeong, J.~Yuan, K.~Su, M.~C. Hersam, Dual--gated {MoS$_2$} memtransistor crossbar array, Advanced Functional Materials 30~(45) (2020) 2003683.
\newblock \href {https://doi.org/10.1002/adfm.202003683} {\path{doi:10.1002/adfm.202003683}}.

\bibitem{Yan.2022}
X.~Yan, J.~H. Qian, V.~K. Sangwan, M.~C. Hersam, Progress and challenges for memtransistors in neuromorphic circuits and systems, Advanced Materials 34~(48) (2022) e2108025.
\newblock \href {https://doi.org/10.1002/adma.202108025} {\path{doi:10.1002/adma.202108025}}.

\bibitem{Rodder.2020}
M.~A. Rodder, S.~Vasishta, A.~Dodabalapur, Double-gate {MoS$_2$} field-effect transistor with a multilayer graphene floating gate: a versatile device for logic, memory, and synaptic applications, ACS Applied Materials {\&} Interfaces 12~(30) (2020) 33926--33933.
\newblock \href {https://doi.org/10.1021/acsami.0c08802} {\path{doi:10.1021/acsami.0c08802}}.

\bibitem{Rahimifard.2022}
L.~Rahimifard, A.~Shylendra, S.~Nasrin, S.~E. Liu, V.~K. Sangwan, M.~C. Hersam, A.~R. Trivedi, Higher order neural processing with input-adaptive dynamic weights on {MoS$_2$} memtransistor crossbars, Frontiers in Electronic Materials 2 (2022).
\newblock \href {https://doi.org/10.3389/femat.2022.950487} {\path{doi:10.3389/femat.2022.950487}}.

\bibitem{Liu.2024}
S.~E. Liu, T.~T. Zeng, R.~Wu, V.~K. Sangwan, M.~C. Hersam, Low-voltage short-channel {MoS$_2$} memtransistors with high gate-tunability, Journal of Materials Research 39~(10) (2024) 1463--1472.
\newblock \href {https://doi.org/10.1557/s43578-024-01343-3} {\path{doi:10.1557/s43578-024-01343-3}}.

\bibitem{Su.2021}
S.-K. Su, C.-P. Chuu, M.-Y. Li, C.-C. Cheng, H.-S.~P. Wong, L.-J. Li, Layered semiconducting 2d materials for future transistor applications, Small Structures 2~(5) (2021) 2000103.
\newblock \href {https://doi.org/10.1002/sstr.202000103} {\path{doi:10.1002/sstr.202000103}}.

\bibitem{Wang.2022b}
L.~Wang, X.~Shen, Z.~Gao, J.~Fu, S.~Yao, L.~Cheng, X.~Lian, Review of applications of 2d materials in memristive neuromorphic circuits, Journal of Materials Science 57~(8) (2022) 4915--4940.
\newblock \href {https://doi.org/10.1007/s10853-022-06954-x} {\path{doi:10.1007/s10853-022-06954-x}}.

\bibitem{Wang.2020c}
C.-Y. Wang, C.~Wang, F.~Meng, P.~Wang, S.~Wang, S.-J. Liang, F.~Miao, {2D} layered materials for memristive and neuromorphic applications, Advanced Electronic Materials 6~(2) (2020) 1901107.
\newblock \href {https://doi.org/10.1002/aelm.201901107} {\path{doi:10.1002/aelm.201901107}}.

\bibitem{Sangwan.2018}
V.~K. Sangwan, M.~C. Hersam, Electronic transport in two-dimensional materials, Annual review of physical chemistry 69 (2018) 299--325.
\newblock \href {https://doi.org/10.1146/annurev-physchem-050317-021353} {\path{doi:10.1146/annurev-physchem-050317-021353}}.

\bibitem{Geim.2013}
A.~K. Geim, I.~V. Grigorieva, Van der waals heterostructures, Nature 499~(7459) (2013) 419--425.
\newblock \href {https://doi.org/10.1038/nature12385} {\path{doi:10.1038/nature12385}}.

\bibitem{Jadwiszczak.2019}
J.~Jadwiszczak, D.~Keane, P.~Maguire, C.~P. Cullen, Y.~Zhou, H.~Song, C.~Downing, D.~Fox, N.~McEvoy, R.~Zhu, J.~Xu, G.~S. Duesberg, Z.-M. Liao, J.~J. Boland, H.~Zhang, {MoS$_2$} memtransistors fabricated by localized helium ion beam irradiation, ACS Nano 13~(12) (2019) 14262--14273.
\newblock \href {https://doi.org/10.1021/acsnano.9b07421} {\path{doi:10.1021/acsnano.9b07421}}.

\bibitem{DaLi.2018}
{Da Li}, B.~Wu, X.~Zhu, J.~Wang, B.~Ryu, W.~D. Lu, W.~Lu, X.~Liang, Mos$_2$ memristors exhibiting variable switching characteristics toward biorealistic synaptic emulation, ACS Nano 12~(9) (2018) 9240--9252.
\newblock \href {https://doi.org/10.1021/acsnano.8b03977} {\path{doi:10.1021/acsnano.8b03977}}.

\bibitem{Ge.2018}
R.~Ge, X.~Wu, M.~Kim, J.~Shi, S.~Sonde, L.~Tao, Y.~Zhang, J.~C. Lee, D.~Akinwande, Atomristor: nonvolatile resistance switching in atomic sheets of transition metal dichalcogenides, Nano Letters 18~(1) (2018) 434--441.
\newblock \href {https://doi.org/10.1021/acs.nanolett.7b04342} {\path{doi:10.1021/acs.nanolett.7b04342}}.

\bibitem{Sangwan.2015}
V.~K. Sangwan, D.~Jariwala, I.~S. Kim, K.-S. Chen, T.~J. Marks, L.~J. Lauhon, M.~C. Hersam, Gate-tunable memristive phenomena mediated by grain boundaries in single-layer {MoS$_2$}, Nature Nanotechnology 10~(5) (2015) 403--406.
\newblock \href {https://doi.org/10.1038/nnano.2015.56} {\path{doi:10.1038/nnano.2015.56}}.

\bibitem{Thakkar.2024}
P.~Thakkar, J.~Gosai, H.~J. Gogoi, A.~Solanki, From fundamentals to frontiers: a review of memristor mechanisms, modeling and emerging applications, Journal of Materials Chemistry C 12~(5) (2024) 1583--1608.
\newblock \href {https://doi.org/10.1039/D3TC03692H} {\path{doi:10.1039/D3TC03692H}}.

\bibitem{Sivan.2022}
M.~Sivan, J.~F. Leong, J.~Ghosh, B.~Tang, J.~Pan, E.~Zamburg, A.~V.-Y. Thean, Physical insights into vacancy-based memtransistors: toward power efficiency, reliable operation, and scalability, ACS Nano 16~(9) (2022) 14308--14322.
\newblock \href {https://doi.org/10.1021/acsnano.2c04504} {\path{doi:10.1021/acsnano.2c04504}}.

\bibitem{Spetzler.2024}
B.~Spetzler, D.~Abdel, F.~Schwierz, M.~Ziegler, P.~Farrell, The role of vacancy dynamics in two--dimensional memristive devices, Advanced Electronic Materials 10~(1) (2024).
\newblock \href {https://doi.org/10.1002/aelm.202300635} {\path{doi:10.1002/aelm.202300635}}.

\bibitem{Spetzler.2022}
B.~Spetzler, Z.~Geng, K.~Rossnagel, M.~Ziegler, F.~Schwierz, Lateral 2d tmdc memristors -- experiment and modeling, in: 2022 IEEE 16th International Conference on Solid-State {\&} Integrated Circuit Technology (ICSICT), 2022, pp. 1--3.
\newblock \href {https://doi.org/10.1109/ICSICT55466.2022.9963350} {\path{doi:10.1109/ICSICT55466.2022.9963350}}.

\bibitem{Zhou.2023b}
H.~Zhou, V.~Sorkin, S.~Chen, Z.~Yu, K.-W. Ang, Y.-W. Zhang, Design--dependent switching mechanisms of schottky--barrier--modulated memristors based on {2D} semiconductor, Advanced Electronic Materials 9~(6) (2023) 2201252.
\newblock \href {https://doi.org/10.1002/aelm.202201252} {\path{doi:10.1002/aelm.202201252}}.

\bibitem{Aldana.2023b}
S.~Aldana, J.~Jadwiszczak, H.~Zhang, On the switching mechanism and optimisation of ion irradiation enabled {2D MoS$_2$} memristors, Nanoscale 15~(13) (2023) 6408--6416.
\newblock \href {https://doi.org/10.1039/D2NR06810A} {\path{doi:10.1039/D2NR06810A}}.

\bibitem{Aldana.2023}
S.~Aldana, H.~Zhang, Unravelling the data retention mechanisms under thermal stress on 2d memristors, ACS omega 8~(30) (2023) 27543--27552.
\newblock \href {https://doi.org/10.1021/acsomega.3c03200} {\path{doi:10.1021/acsomega.3c03200}}.

\bibitem{Ielmini.2017}
D.~Ielmini, V.~Milo, Physics-based modeling approaches of resistive switching devices for memory and in-memory computing applications, Journal of computational electronics 16~(4) (2017) 1121--1143.
\newblock \href {https://doi.org/10.1007/s10825-017-1101-9} {\path{doi:10.1007/s10825-017-1101-9}}.

\bibitem{Gao.2021}
L.~Gao, Q.~Ren, J.~Sun, S.-T. Han, Y.~Zhou, Memristor modeling: challenges in theories, simulations, and device variability, Journal of Materials Chemistry C 9~(47) (2021) 16859--16884.
\newblock \href {https://doi.org/10.1039/D1TC04201G} {\path{doi:10.1039/D1TC04201G}}.

\bibitem{Farrell.2017}
P.~Farrell, N.~Rotundo, D.~H. Doan, M.~Kantner, J.~Fuhrmann, T.~Koprucki, Drift-diffusion models, in: J.~Piprek (Ed.), Handbook of optoelectronic device modeling and simulation, Series in Optics and Optoelectronics, {CRC Press}, Portland, 2017, pp. 733--772.
\newblock \href {https://doi.org/10.4324/9781315152318-25} {\path{doi:10.4324/9781315152318-25}}.

\bibitem{Markowich.1990}
P.~A. Markowich, C.~Schmeiser, C.~Ringhofer, Semiconductor equations, Springer, Wien, New York, 1990.
\newblock \href {https://doi.org/10.1007/978-3-7091-6961-2} {\path{doi:10.1007/978-3-7091-6961-2}}.

\bibitem{Abdel.2023}
D.~Abdel, N.~E. Courtier, P.~Farrell, Volume exclusion effects in perovskite charge transport modeling, Optical and Quantum Electronics 55~(10) (2023).
\newblock \href {https://doi.org/10.1007/s11082-023-05125-9} {\path{doi:10.1007/s11082-023-05125-9}}.

\bibitem{Strukov.2009}
D.~B. Strukov, J.~L. Borghetti, R.~S. Williams, Coupled ionic and electronic transport model of thin-film semiconductor memristive behavior, Small 5~(9) (2009) 1058--1063.
\newblock \href {https://doi.org/10.1002/smll.200801323} {\path{doi:10.1002/smll.200801323}}.

\bibitem{Marchewka.2016}
A.~Marchewka, B.~Roesgen, K.~Skaja, H.~Du, C.-L. Jia, J.~Mayer, V.~Rana, R.~Waser, S.~Menzel, Nanoionic resistive switching memories: on the physical nature of the dynamic reset process, Advanced Electronic Materials 2~(1) (2016) 1500233.
\newblock \href {https://doi.org/10.1002/aelm.201500233} {\path{doi:10.1002/aelm.201500233}}.

\bibitem{Jourdana.2023}
C.~Jourdana, A.~J{\"u}ngel, N.~Zamponi, Three-species drift-diffusion models for memristors, Mathematical Models and Methods in Applied Sciences 33~(10) (2023) 2113--2156.
\newblock \href {https://doi.org/10.1142/S0218202523500501} {\path{doi:10.1142/S0218202523500501}}.

\bibitem{CuestaLopez.2024}
J.~Cuesta-Lopez, M.~D. Ganeriwala, E.~G. Marin, A.~Toral-Lopez, F.~Pasadas, F.~G. Ruiz, A.~Godoy, Numerical study of synaptic behavior in amorphous hfo2-based ferroelectric-like fets generated by voltage-driven ion migration, Journal of Applied Physics 136~(12) (2024).
\newblock \href {https://doi.org/10.1063/5.0212084} {\path{doi:10.1063/5.0212084}}.

\bibitem{Calado.2022}
P.~Calado, I.~Gelmetti, B.~Hilton, M.~Azzouzi, J.~Nelson, P.~R.~F. Barnes, Driftfusion: an open source code for simulating ordered semiconductor devices with mixed ionic-electronic conducting materials in one dimension, Journal of Computational Electronics 21~(4) (2022) 960--991.
\newblock \href {https://doi.org/10.1007/s10825-021-01827-z} {\path{doi:10.1007/s10825-021-01827-z}}.

\bibitem{Aoki.2014}
Y.~Aoki, C.~Wiemann, V.~Feyer, H.-S. Kim, C.~M. Schneider, H.~Ill-Yoo, M.~Martin, Bulk mixed ion electron conduction in amorphous gallium oxide causes memristive behaviour, Nature Communications 5 (2014) 3473.
\newblock \href {https://doi.org/10.1038/ncomms4473} {\path{doi:10.1038/ncomms4473}}.

\bibitem{Cances.2023}
C.~Canc{\`e}s, C.~Chainais-Hillairet, B.~Merlet, F.~Raimondi, J.~Venel, Mathematical analysis of a thermodynamically consistent reduced model for iron corrosion, Zeitschrift für angewandte Mathematik und Physik 74~(96) (2023).
\newblock \href {https://doi.org/10.1007/s00033-023-01970-6} {\path{doi:10.1007/s00033-023-01970-6}}.

\bibitem{Bataillon.2010}
C.~Bataillon, F.~Bouchon, C.~Chainais-Hillairet, C.~Desgranges, E.~Hoarau, F.~Martin, S.~Perrin, M.~Tupin, J.~Talandier, Corrosion modelling of iron based alloy in nuclear waste repository, Electrochimica Acta 55~(15) (2010) 4451--4467.
\newblock \href {https://doi.org/10.1016/j.electacta.2010.02.087} {\path{doi:10.1016/j.electacta.2010.02.087}}.

\bibitem{Herda.2024}
M.~Herda, A.~Jüngel, S.~Portisch, {Charge transport systems with Fermi-Dirac statistics for memristors} (2024).
\newblock \href {http://arxiv.org/abs/2409.01196} {\path{arXiv:2409.01196}}.

\bibitem{Abdel.2024b}
D.~Abdel, A.~Glitzky, M.~Liero, Analysis of a drift-diffusion model for perovskite solar cells, Discrete and Continuous Dynamical Systems - B 30~(1) (2024) 99--131.
\newblock \href {https://doi.org/10.3934/dcdsb.2024081} {\path{doi:10.3934/dcdsb.2024081}}.

\bibitem{glitzky2024uniqueness}
A.~Glitzky, M.~Liero, Uniqueness and regularity of weak solutions of a drift-diffusion system for perovskite solar cells (2024).
\newblock \href {http://arxiv.org/abs/2411.18223} {\path{arXiv:2411.18223}}.

\bibitem{Selberherr.1984}
S.~Selberherr, Analysis and Simulation of Semiconductor Devices, Springer, Wien, 1984.

\bibitem{Bank.1983}
R.~E. Bank, D.~J. Rose, W.~Fichtner, Numerical methods for semiconductor device simulation, SIAM Journal on Scientific and Statistical Computing 4~(3) (1983) 416--435.
\newblock \href {https://doi.org/10.1137/0904032} {\path{doi:10.1137/0904032}}.

\bibitem{Brezzi.2005}
F.~Brezzi, L.~D. Marini, S.~Micheletti, P.~Pietra, R.~Sacco, S.~Wang, Discretization of semiconductor device problems (i), in: W.~H.~A. Schilders (Ed.), Handbook of Numerical Analysis, Vol.~13, North-Holland, Amsterdam, 2005, pp. 317--441.
\newblock \href {https://doi.org/10.1016/S1570-8659(04)13004-4} {\path{doi:10.1016/S1570-8659(04)13004-4}}.

\bibitem{Mock.1983}
M.~S. Mock, Analysis of mathematical models for semiconductor devices, {Boole Press}, 1983.

\bibitem{Farrell2017a}
P.~Farrell, T.~Koprucki, J.~Fuhrmann, Computational and analytical comparison of flux discretizations for the semiconductor device equations beyond {B}oltzmann statistics, Journal of Computational Physics 346 (2017) 497--513.
\newblock \href {https://doi.org/10.1016/j.jcp.2017.06.023} {\path{doi:10.1016/j.jcp.2017.06.023}}.

\bibitem{Eymard.2000}
R.~Eymard, T.~Gallou{\"e}t, R.~Herbin, Finite volume methods, in: P.~G. Ciarlet, J.~L. Lions (Eds.), Handbook of numerical analysis, Vol.~7 of Handbook of Numerical Analysis, North-Holland, Amsterdam, 2000, pp. 713--1018.
\newblock \href {https://doi.org/10.1016/S1570-8659(00)07005-8} {\path{doi:10.1016/S1570-8659(00)07005-8}}.

\bibitem{Cances.2021}
C.~Canc{\`e}s, C.~Chainais-Hillairet, J.~Fuhrmann, B.~Gaudeul, A numerical-analysis-focused comparison of several finite volume schemes for a unipolar degenerate drift-diffusion model, IMA Journal of Numerical Analysis 41~(1) (2021) 271--314.
\newblock \href {https://doi.org/10.1093/imanum/draa002} {\path{doi:10.1093/imanum/draa002}}.

\bibitem{Kantner.2020}
M.~Kantner, Generalized {Scharfetter--Gummel} schemes for electro-thermal transport in degenerate semiconductors using the kelvin formula for the seebeck coefficient, Journal of Computational Physics 402 (2020) 109091.
\newblock \href {https://doi.org/10.1016/j.jcp.2019.109091} {\path{doi:10.1016/j.jcp.2019.109091}}.

\bibitem{BessemoulinChatard.2012}
M.~Bessemoulin-Chatard, A finite volume scheme for convection--diffusion equations with nonlinear diffusion derived from the {Scharfetter--Gummel} scheme, Numerische Mathematik 121~(4) (2012) 637--670.
\newblock \href {https://doi.org/10.1007/s00211-012-0448-x} {\path{doi:10.1007/s00211-012-0448-x}}.

\bibitem{BessemoulinChatard.2017}
M.~Bessemoulin-Chatard, C.~Chainais-Hillairet, Exponential decay of a finite volume scheme to the thermal equilibrium for drift--diffusion systems, Journal of Numerical Mathematics 25~(3) (2017).
\newblock \href {https://doi.org/10.1515/jnma-2016-0007} {\path{doi:10.1515/jnma-2016-0007}}.

\bibitem{Glitzky.2011}
A.~Glitzky, Uniform exponential decay of the free energy for voronoi finite volume discretized reaction--diffusion systems, Mathematische Nachrichten 284~(17-18) (2011) 2159--2174.
\newblock \href {https://doi.org/10.1002/mana.200910215} {\path{doi:10.1002/mana.200910215}}.

\bibitem{Moatti.2023}
J.~Moatti, A structure preserving hybrid finite volume scheme for semiconductor models with magnetic field on general meshes, ESAIM: Mathematical Modelling and Numerical Analysis 57~(4) (2023) 2557--2593.
\newblock \href {https://doi.org/10.1051/m2an/2023041} {\path{doi:10.1051/m2an/2023041}}.

\bibitem{Abdel2023Existence}
D.~Abdel, C.~Chainais-Hillairet, P.~Farrell, M.~Herda, Numerical analysis of a finite volume scheme for charge transport in perovskite solar cells, IMA Journal of Numerical Analysis (2023) 1--40\href {https://doi.org/10.1093/imanum/drad034} {\path{doi:10.1093/imanum/drad034}}.

\bibitem{Jungel.2016}
A.~J{\"u}ngel, Entropy Methods for Diffusive Partial Differential Equations, Vol. 804 of SpringerBriefs in Mathematics, Springer, Cham, 2016.
\newblock \href {https://doi.org/10.1007/978-3-319-34219-1} {\path{doi:10.1007/978-3-319-34219-1}}.

\bibitem{Jungel.2023}
A.~J\"{u}ngel, M.~Vetter, Degenerate drift-diffusion systems for memristors (2023).
\newblock \href {http://arxiv.org/abs/2311.16591} {\path{arXiv:2311.16591}}.

\bibitem{Abdel.2021}
D.~Abdel, P.~V{\'a}gner, J.~Fuhrmann, P.~Farrell, Modelling charge transport in perovskite solar cells: potential-based and limiting ion depletion, Electrochimica Acta 390 (2021) 138696.
\newblock \href {https://doi.org/10.1016/j.electacta.2021.138696} {\path{doi:10.1016/j.electacta.2021.138696}}.

\bibitem{markowich1985stationary}
P.~A. Markowich, The Stationary Semiconductor Device Equations, Springer-Verlag, Wien, 1985.
\newblock \href {https://doi.org/10.1007/978-3-7091-3678-2} {\path{doi:10.1007/978-3-7091-3678-2}}.

\bibitem{Abdel2024Thesis}
D.~Abdel, Modeling and simulation of vacancy-assisted charge transport in innovative semiconductor devices, {PhD Thesis}, {Freie Universität Berlin} (2024).

\bibitem{Bessemoulin.2014}
M.~Bessemoulin-Chatard, C.~Chainais-Hillairet, M.-H. Vignal, Study of a finite volume scheme for the drift-diffusion system. {Asymptotic} behavior in the quasi-neutral limit, SIAM Journal on Numerical Analysis 52~(4) (2014) 1666--1691.
\newblock \href {https://doi.org/10.1137/130913432} {\path{doi:10.1137/130913432}}.

\bibitem{ChainaisHillairet.2019}
C.~Chainais-Hillairet, M.~Herda, {Large-time behaviour of a family of finite volume schemes for boundary-driven convection–diffusion equations}, IMA Journal of Numerical Analysis 40~(4) (2019) 2473--2504.
\newblock \href {https://doi.org/10.1093/imanum/drz037} {\path{doi:10.1093/imanum/drz037}}.

\bibitem{Yu.1988}
Z.~Yu, R.~Dutton, {SEDAN III} -- {A} one-dimensional device simulator, \url{www-tcad.stanford.edu/tcad/programs/sedan3.html} (1988).

\bibitem{Gaudeul.2021}
B.~Gaudeul, J.~Fuhrmann, {Entropy and convergence analysis for two finite volume schemes for a Nernst-Planck-Poisson system with ion volume constraints}, Numerische Mathematik 151 (2022) 99--149.
\newblock \href {https://doi.org/10.1007/s00211-022-01279-y} {\path{doi:10.1007/s00211-022-01279-y}}.

\bibitem{Abdel.2021b}
D.~Abdel, P.~Farrell, J.~Fuhrmann, Assessing the quality of the excess chemical potential flux scheme for degenerate semiconductor device simulation, Optical and Quantum Electronics 53~(3) (2021) 1--10.
\newblock \href {https://doi.org/10.1007/s11082-021-02803-4} {\path{doi:10.1007/s11082-021-02803-4}}.

\bibitem{Moatti2024}
S.~Lemaire, J.~Moatti, \href{https://www.aimspress.com/article/doi/10.3934/mine.2024005}{Structure preservation in high-order hybrid discretisations of potential-driven advection-diffusion: linear and nonlinear approaches}, Mathematics in Engineering 6~(1) (2024) 100--136.
\newblock \href {https://doi.org/10.3934/mine.2024005} {\path{doi:10.3934/mine.2024005}}.
\newline\urlprefix\url{https://www.aimspress.com/article/doi/10.3934/mine.2024005}

\bibitem{Moatti2023structure}
J.~Moatti, A structure preserving hybrid finite volume scheme for semiconductor models with magnetic field on general meshes, ESAIM: Mathematical Modelling and Numerical Analysis 57~(4) (2023) 2557--2593.
\newblock \href {https://doi.org/10.1051/m2an/2023041} {\path{doi:10.1051/m2an/2023041}}.

\bibitem{ChargeTransport}
D.~Abdel, P.~Farrell, J.~Fuhrmann, Chargetransport.jl -- simulating charge transport in semiconductors (doi: 10.5281/zenodo.6257906).
\newblock \href {https://doi.org/10.5281/zenodo.6257906} {\path{doi:10.5281/zenodo.6257906}}.

\bibitem{Davis2004}
T.~A. Davis, \href{https://doi.org/10.1145/992200.992206}{Algorithm 832: Umfpack v4.3---an unsymmetric-pattern multifrontal method}, ACM Trans. Math. Softw. 30~(2) (2004) 196–199.
\newblock \href {https://doi.org/10.1145/992200.992206} {\path{doi:10.1145/992200.992206}}.
\newline\urlprefix\url{https://doi.org/10.1145/992200.992206}

\bibitem{Allain.2015}
A.~Allain, J.~Kang, K.~Banerjee, A.~Kis, Electrical contacts to two-dimensional semiconductors, Nature Materials 14~(12) (2015) 1195--1205.
\newblock \href {https://doi.org/10.1038/nmat4452} {\path{doi:10.1038/nmat4452}}.

\bibitem{Kang.2014}
J.~Kang, W.~Liu, D.~Sarkar, D.~Jena, K.~Banerjee, Computational study of metal contacts to monolayer transition-metal dichalcogenide semiconductors, Physical Review X 4~(3) (2014).
\newblock \href {https://doi.org/10.1103/PhysRevX.4.031005} {\path{doi:10.1103/PhysRevX.4.031005}}.

\bibitem{Wang.2022}
Y.~Wang, M.~Chhowalla, Making clean electrical contacts on 2d transition metal dichalcogenides, Nature Reviews Physics 4~(2) (2022) 101--112.
\newblock \href {https://doi.org/10.1038/s42254-021-00389-0} {\path{doi:10.1038/s42254-021-00389-0}}.

\bibitem{Evans.2010}
L.~C. Evans, Partial Differential Equations, 2nd Edition, Vol.~19 of Graduate Studies in Mathematics, American Mathematical Society, 2010.

\end{thebibliography}

\end{document}